\documentclass[a4paper, 12pt, bibliography=totoc]{scrartcl}
\usepackage[utf8]{inputenc}

\RequirePackage[l2tabu, orthodox]{nag}
\usepackage{amsmath}
\usepackage{amsthm}
\usepackage{amssymb}
\usepackage[usenames,dvipsnames]{xcolor}
\usepackage{graphicx}
\usepackage{tikz}
\usepackage{hyperref}
\usepackage{enumerate}
\usepackage[nocites,norefs]{refcheck}
\usepackage[all]{xy}

\newtheorem{lemma}{Lemma}[section]
\newtheorem{proposition}[lemma]{Proposition}
\newtheorem{theorem}[lemma]{Theorem}
\newtheorem{corollary}[lemma]{Corollary}

\newtheorem{theoremnn}{Theorem}

\newcommand{\Wlg}{Without loss of generality }

\renewcommand{\tilde}[1]{\widetilde{#1}}

\newcommand{\m}{\mathfrak{m}}

\DeclareMathOperator{\Transp}{Transp}

\DeclareMathOperator{\A}{A}
\DeclareMathOperator{\C}{C}
\DeclareMathOperator{\D}{D}
\DeclareMathOperator{\E}{E}

\DeclareMathOperator{\EU}{EU}
\DeclareMathOperator{\Sp}{Sp}
\DeclareMathOperator{\Ep}{Ep}
\DeclareMathOperator{\diag}{diag}
\DeclareMathOperator{\sdiag}{sdiag}

\DeclareMathOperator{\M}{M}
\DeclareMathOperator{\N}{N}
\DeclareMathOperator{\GL}{GL}
\DeclareMathOperator{\G}{G}

\newcommand*\rectangled[1]{%
  {\tikz[baseline=(R.base)]\node[draw,rectangle,inner sep=0.5pt](R){\scriptsize\rmfamily{#1}};\!}
}
\makeatletter
\def\blfootnote{\gdef\@thefnmark{}\@footnotetext}
\makeatother

\title{Overgroups of elementary block-diagonal subgroups 
in the classical symplectic group over an arbitrary commutative ring}
\author{Alexander Shchegolev\\\normalsize{Saint Petersburg State University}\\\normalsize{ iRyoka@gmail.com}}
\date{}

\begin{document}
\maketitle
\vspace{-1.5cm}
\blfootnote{Research is supported by the Russian Science Foundation grant 17-11-01261.}

\begin{abstract}
    In this paper we prove a sandwich classification theorem for subgroups of the classical 
    symplectic group over an arbitrary commutative ring $R$ that contain the elementary
	block-diagonal (or subsystem) subgroup $\Ep(\nu, R)$ corresponding to a unitary equivalence
	realation $\nu$ such that all self-conjugate equivalence classes of $\nu$ are of size at least 4
	and all not-self-conjugate classes of $\nu$ are of size at least 5. Namely, given a subgroup $H$
	of $\Sp(2n, R)$ such that $\Ep(\nu, R) \le H$ we show that there exists a unique exact major
	form net of ideals $(\sigma, \Gamma)$ over $R$ such that
	$\Ep(\sigma, \Gamma) \le H \le \N_{\Sp(2n,R)}(\Sp(\sigma, \Gamma))$. Further,
	we describe the normalizer $\N_{\Sp(2n,R)}(\Sp(\sigma, \Gamma))$ in terms of 
	congruences.
\end{abstract}
    
\section{Introduction}
%\addcontentsline{toc}{section}{Introduction}
This paper describes the overgroups $H$ of the elementary symplectic block-diagonal
subgroup $\Ep(\nu,R)$ of type $\nu$ of the classical 
symplectic group $\Sp(2n,R)$ over a commutative ring $R$, 
under the assumption that the minimal size of a self-conjugate block
of $\Ep(\nu,R)$ is at least 4 and the minimal size of a non-self-conjugate block of $\Ep(\nu, R)$ is at least 5. 
The main result is the following sandwich classification theorem: let $\nu$, $R$, $H$, $\Ep(\nu, R)$ and $\Sp(2n, R)$
be as above. Then there exists a unique major exact form net of ideals $(\sigma, \Gamma)$ over 
the ground ring such that $H$ fits into the sandwich
\[
	\Ep(\sigma, \Gamma) \le H \le 
	\N_{\Sp(2n,R)}(\Sp(\sigma, \Gamma)),
\]
where $\N_{\Sp(2n,R)}(\Sp(\sigma, \Gamma))$ denotes the normalizer in $\Sp(2n,R)$ 
of the form net subgroup $\Sp(\sigma, \Gamma)$ 
of $\Sp(2n,R)$ of level $(\sigma, \Gamma)$ and $\Ep(\sigma, \Gamma)$ 
the elementary form net subgroup of $\Sp(\sigma, \Gamma)$.

To put this result into context we provide a brief overview of related
results describing the subgroup structure of linear 
groups over fields. In \cite{DynMaxClass1957, DynSimisimpleSubalg1957} Dynkin determined
the maximal closed connected subgroups of classical algebraic groups over
$\mathbb{C}$. In particular, he showed that all reductive maximal closed
connected subgroups are precisely the stabilizers of totally isotropic
or non-degenerate subspaces. Similar results for classical groups over algebraically closed fields 
of positive characteristic were 
obtained by Gary Seitz \cite{SeitzMaxClass1987} and for exceptional groups by Donna Testerman \cite{TestermanIrred1988}.
In the papers \cite{AschMaxFCG1984, AschStrFG1985,
AschFSG1986} Michael Asch\-bacher described the maximal subgroups of finite 
simple classical groups. The subgroup structure theorem of Aschbacher says that every maximal
subgroup of a finite simple classical group belongs to either one of 8
explicitly defined classes $\mathcal{C}_1$---$\mathcal{C}_8$ of large 
subgroups or to the class $\mathcal{S}$ of almost simple groups in irreducible
representations. An exposition of results regarding the members of Aschbacher classes
in finite classical groups can be found in the book of Kleidman 
and Liebeck \cite{KleidLiebSubgrBook1990}. 
The Aschbacher classes which are relevant for us are the classes
$\mathcal{C}_1$ and $\mathcal{C}_2$. The subgroups of class $\mathcal{C}_1$ are 
stabilizers of proper totally isotropic or non-degenerate submodules of the module
on which the group is acting. The subgroups of class $\mathcal{C}_2$ are the stabilizers
of direct decompositions of that module into the summands of a fixed dimension.
Given a member $H$ of an Aschbacher class the book of Kleidman and Liebeck provides a recipe 
for constructing a maximal overgroup of $H$ that is in turn also a member of some
not necessarily the same Aschbacher class. Unfortunately, a similar result for 
classical groups over arbitrary 
commutative rings (or unitary groups over form rings) has not yet
been obtained. However, it is possible to describe 
the lattice structure of the set of all overgroups of a given member of an Aschbacher class 
or of an appropriate modified notion thereof
in terms of the structure of the ground ring. 
%In this dissertation the ground ring
%will be quasi-finite, i.e. a direct limit of module finite rings.

The literature contains several modifications of the notion of Aschbacher classes.
The current paper focuses on a specific simultaneous modification of the Aschbacher classes 
$\mathcal{C}_1$ and $\mathcal{C}_2$ which we call \textit{block-diagonal subgroups}.
These subgroups are the stabilizers of certain direct decompositions of the quadratic module 
on which the unitary group is acting into totally isotropic
or non-degenerate submodules. As our methods only employ the elementary 
matrices contained in these block-diagonal subgroups, we will describe
the overgroups of elementary block-diagonal subgroups instead of full block-diagonal
subgroups. The problem of describing overgroups of elementary block-diagonal
matrices was first considered for the case of the general linear group
in papers \cite{BorVavNarGLDed1970, BorVavDiag1982, VavDiagSemiloc1983, BorVavGL,
VavDedArithm1987} of Z.I. Borewicz, N. A. Vavilov and W. Narkiewicz
over commutative rings and rings satisfying a stable rank condition. 
These papers do not use localization. Over quasi-finite rings
the problem is solved in \cite{BakStepNets} using localization.
For the case of other classical groups over a commutative ring with 2 invertible,
this classification was generalized in chapter V of the habilitation
of Nikolai Vavilov, although complete proofs were only published much later in
\cite{VavOrthog} for the split orthogonal case and \cite{VavSympl} for the
symplectic case. Important auxiliary results can be also found in
\cite{VavSymplComm1993}, \cite{VavOrthogSMZ88} and in the references therein.
Roughly the main results in the above references can be described as
follows. Let $\G$ denote a Chevalley group of type $A_l, B_l, C_l$ or $D_l$ over 
a commutative ring $R$ or the general linear group $\GL(n, R)$ over a quasi-finite
ring $R$. If $\G \neq \G(A_l, R)$ or $\GL(n,R)$, assume that $2 \in R^*$. Let 
$H$ be a subgroup of $\G$ containing an elementary block-diagonal
subgroup of $G$ whose minimal block size is sufficiently large. Then
there exists a unique net of ideals $\sigma$ such that $H$ fits into the sandwich
\[ 
    \E(\sigma) \le H \le \N(\sigma), 
\]
where $\E(\sigma)$ is the elementary subgroup associated with $\sigma$ and
$\N(\sigma)$ is the normalizer in $G$ of the net subgroup $\G(\sigma)$. In
\cite{BorVavGL} such a description was called \textit{standard}. We shall
refer to it as \textit{standard sandwich classification}.
  
Unfortunately, due to known counterexamples the standard sandwich
classification in \cite{VavSympl} for the symplectic group $\Sp(2n,R)$ 
when 2 is invertible in the ground ring $R$ does not generalize to
the case of an arbitrary commutative ring. 
The obstacle is that the notion of a net of ideals is not fine enough when $2$
is not invertible in the ring and has to be replaced by the notion of a form net 
of ideals. This is analogous to the situation encountered in the sandwich 
classification of subgroups of Bak unitary groups \cite{BakThesis,BakVavStructure2000}, which are normalized by the 
elementary subgroup. Here the 
notion of ideal had to be refined by the notion of form ideal
when 2 is not invertible in the ground ring.
Significant work on developing the concept of a form net of ideals in the context 
of even unitary groups over fields and simple Artinian 
rings was done already by E. Dybkova \cite{DybFormNets1998, DybDiagUni1999,
DybkovaBorHypUni2006, DybArtinI2007, DybDiagUni2009}. 
Further review of known results on the problem of describing overgroups of 
subsystem subgroups in classical-like and some exceptional groups can be
found in \cite{VavShchegLevelsEng2012}. As mentioned there, 
no steps have yet been taken towards describing overgroups of subsystem
subgroups in classical (other than $GL$) groups over an arbitrary 
commutative ring or Bak unitary groups over form rings other than simple
Artinian form rings. 

This paper is the first one in the series devoted to the describing the 
overgroups of elementary block-diagonal subgroups of even unitary groups
over quasi-finite rings which constitutes the PhD thesis \cite{ShchegBielThesis}
of the author. The statement of the result for the even unitary case can also be found
in \cite{ShchMainRes2016}. Now we state the main result of this paper.

Let $R$ denote a commutative ring and $\Sp(2n, R)$ the classical simplectic group 
with coefficients in $R$. Let $\Ep(2n, R)$ denote the elementary subgroup of $\Sp(2n, R)$.
Given a form net of ideals $(\sigma, \Gamma)$ denote by $\Ep(\sigma, \Gamma)$ the 
subgroup of $\Ep(2n, R)$ generated by all short and long symplectic transvections $T_{ij}(\xi),  T_{i,-i}(\alpha)$, 
where $i,j \in I$, $i \neq \pm j$, $\xi \in \sigma_{ij}$ and $\alpha \in \Gamma_i$. The central result of this paper is the following theorem (see Sections \ref{sec:notations} and \ref{sec:symplss:prelim} for the definitions
omitted here). 

\begin{theoremnn}
    \label{theorem:spgen:main}
    Let $\nu$ be a unitary equivalence relation on the index set $I$ such that
    $h(\nu) \geq (4,5)$.
    Let $H$ be a subgroup of $\Sp(2n,R)$ such that $\Ep(\nu, R) \le H$. Then there exists a unique 
    exact form net of ideals $(\sigma, \Gamma) \ge [\nu]_R$ such that 
    \[ 
        \Ep(\sigma, \Gamma) \le H \le \N_{\Sp(2n,R)}(\Sp(\sigma, \Gamma)). 
    \]
\end{theoremnn}

%\begin{theoremnn}
%    \label{theorem:symplss:main:long}
%    Let $\nu$ be a unitary equivalence relation on the index set $I$
%    such that \\ $h(\nu) \geq (4,5)$. 
%    Let $\Ep(\nu, R)$ denote the block-diagonal elementary subgroup 
%    of $\Sp(2n, R)$ defined by $\nu$.    
%    Let $H$ be a subgroup of $\Sp(2n,R)$ such that 
%    $\Ep(\nu, R) \le H$. Denote by $(\sigma, \Gamma)$ the form net of 
%    ideals associated with $H$. Then
%    \[ 
%        H \le \Transp_{\Sp(2n,R)}(\Ep^L(\sigma, \Gamma), \Sp(\sigma, \Gamma)). 
%    \]
%\end{theoremnn}

Another important result is that the normalizer $N_{\Sp(2n,R)}(\Sp(\sigma, \Gamma))$ 
coincides with the transporter 
\begin{align*}
	\Transp_{\Sp(2n,R)}(\Ep(\sigma, \Gamma), &\Sp(\sigma, \Gamma)) = \\&
	\{ 
		a \in \Sp(2n, R) \;|\; \text{for all } \tau \in \Ep(\sigma, \Gamma) \; a \tau a^{-1} \in \Sp(\sigma, \Gamma)
	\}
\end{align*}
in $\Sp(2n,R)$ from $\Ep(\sigma, \Gamma)$ to 
$\Sp(\sigma, \Gamma)$ and can be described in terms of congruences.

\begin{theoremnn}
    \label{theorem:symplss:transpdescr}
    Let $\nu$ be a unitary equivalence relation on the index set $I$ such that all 
    the equivalence classes of $\nu$ contain at least 3 elements. Let $(\sigma, \Gamma)$ be a form net of ideals
    over $R$ such that $[\nu]_R \le (\sigma, \Gamma)$. Then 
    the normalizer $\N_{\Sp(2n, R)}(\Sp(\sigma, \Gamma))$
    coincides with the transporter
    $\Transp_{\Sp(2n, R)}(\Ep(\sigma, \Gamma), \Sp(\sigma, \Gamma))$
    and consists precisely of all matrices $a$ in $\Sp(2n,R)$ such that 
    the following three properties hold:
    \begin{enumerate}
        \item[\textrm{\textnormal{(T1)}}] $a_{ij} \sigma_{jk} a'_{kl} \le \sigma_{il}$ for all $i,j,k,l \in I$
        \item[\textrm{\textnormal{(T2)}}] $a_{ij}^2 \sigma_{jk}^\rectangled{2} S_{k,-k}(a^{-1}) \in \Gamma_i$ for all $i,j,k \in I$
        \item[\textrm{\textnormal{(T3)}}] $a_{ij}^2 \Gamma_j \le \Gamma_i$ for all $i,j \in I$,
    \end{enumerate}    
    where $\sigma_{jk}^\rectangled{2}$ stands for $\{ \xi^2 \;|\; \xi \in \sigma_{jk} \}$.
\end{theoremnn}

The rest of this paper is organized as follows. In Sections \ref{sec:notations} and
\ref{sec:symplss:prelim} we give all the required notation and definitions. 
In Section \ref{sec:symplss:assoc:net:and:norm} we construct the net associated with a
 subgroup and prove Theorem \ref{theorem:symplss:transpdescr}. 
 Sections \ref{sec:spgen:stsetting}---\ref{sec:spgen:localization}
comprise the proof of Theorem
\ref{theorem:spgen:main}. We will show that 
$H \le \Transp_{\Sp(2n, R)}(\Ep(\sigma, \Gamma), \Sp(\sigma, \Gamma))$, 
where $(\sigma,\Gamma)$ is the net associated with $H$. Given a matrix $a \in H$ and 
an elementary 
symplectic transvection $T_{sr}(\xi) \in \Ep(\sigma,\Gamma)$ it will be shown that 
$b = a T_{sr}(\xi) a^{-1} \in \Sp(\sigma, \Gamma)$. By definition of $\Sp(\sigma, \Gamma)$
the inclusion $b \in \Sp(\sigma, \Gamma)$ is equivalent to the following series of inclusions:
\begin{align*}
	T_{ij}(b_{ij}) \in H & \text{ for all } i \neq \pm j&
	T_{i,-i}(S_{i,-i}(b)) \in H & \text{ for all } i \in I,
\end{align*}
where $S_{i,-i}(b) = \sum_{j>0}b_{ij} b'_{j,-i}$. In order to prove these inclusions
we will express the matrices $T_{ij}(b_{ij})$ and $T_{i,-i}(S_{i,-i}(b))$ as products
of elements of $\Ep(\sigma, \Gamma)$ as well as marices $b$ and $b^{-1}$. We shall refer
to this procedure as \textit{extraction of transvections}. First, in Section 
\ref{sec:spgen:stsetting} we introduce a useful
abstraction that simplifies loclization based versions of extraction of transvections.
In Section \ref{sec:spgen:extract} we collect the results on extracting transvections 
from parabolic subgroups. In Sections \ref{sec:spgen:jac} and \ref{sec:spgen:local} we 
deal with the case of a local ground ring and the Section \ref{sec:spgen:localization} 
provides the proof of Theorem \ref{theorem:spgen:main} for an arbitrary commutative ring
using \textit{localization and patching}.

\section{Notations}
\label{sec:notations}

Throughout this paper we will adhere to the following notations and
conventions. By \textit{ring} we will always mean associative unital ring.
Given a ring $R$ and a natural number $n$ we will denote by $\M(n,R)$ the
full matrix ring of rank $n$ over $R$ and by $\GL(n,R)$ the group of
invertible elements of $\M(n,R)$. For any matrix $a$ in $\M(n,R)$ let $a_{ij}$ 
denote the entry of $a$ at the position $(i,j)$ and 
$a'_{ij}$ the corresponding entry of the matrix $a^{-1}$ inverse to $a$. We will
denote by $a^t$ the matrix transpose of $a$, i.e. the matrix in $\M(n,R)$
such that $(a^t)_{ij} = a_{ji}$. By $a_{i*}$ and $a_{*j}$ we will denote
the $i$'th row and $j$'th column of $a$ respectively. Naturally, $a'_{i*}$
and $a'_{*j}$ should be read as the $i$'th row and $j$'th column of
$a^{-1}$ respectively. We will also use the notation 
$\diag(\xi_1, \dots, \xi_n)$ for the diagonal matrix with entries 
$\xi_1, \dots \xi_n$ reading from the top-left corner and 
$\sdiag(\xi_1, \dots, \xi_n)$ for the skew-diagonal matrix with entries 
$\xi_1, \dots \xi_n$ reading from the top-right corner. When the rank
of the matrix ring is clear from the context, we will also denote by
$\diag$ the diagonal embedding of $R$ into $\M(n,R)$, i.e.
$\diag : R \rightarrow \M(n,R)$ is a ring homomorphism sending each
$\xi \in R$ to the diagonal matrix $\diag(\xi) = \diag(\xi, \dots, \xi)$.
We will denote by $e_n$ the unity of the matrix ring $\M(n,R)$. When the rank
is clear from the context, we will simply write $e$. The entries of $e$,
as an exception from the above convention, will be denoted by $\delta_{ij}$ 
(Kronecker delta), while $e_{ij}$ will stand for the corresponding 
standard matrix unit, i.e the matrix in $\M(n,R)$ whose $(i,j)$'th entry equals $1$ and whose other
entries are zero.
Given a ring morphism $\varphi : R \rightarrow Q$ we will denote by $\M_n(\varphi) = \M(\varphi)$ the induced
ring morphism of the matrix rings $\M(n, R)$ and $\M(n, Q)$. If we consider $\M(\varphi)$
as a morphism of the multiplicative monoids of $\M(n, R)$ and $\M(n, Q)$ then
its kernel is precisely the set $\M(n, R, \ker(\varphi)) = \{ a \in \M(n,R) \;|\; a_{ij} \equiv \delta_{ij} \mod \ker(\varphi) \text{ for all } i,j \}$. Note that 
$\GL(2n, R, \ker(\varphi)) = \GL(2n,R) \cap \M(2n, R, \ker(\varphi))$ is a normal
subgroup in $\GL(2n,R)$.

In our applications it is convenient to index the rows and columns of $2n \times 2n$ matrices
by the ordered set $I = I_{2n} = \{ 1, \dots, n , -n, \dots, -1 \}$. We equip
the poset $I$ with the sign map $\varepsilon : I \longrightarrow \{\pm 1\}$, defined by
\[ \varepsilon(i) = \begin{cases}
	+1 & i > 0 \\
	-1 & i < 0
\end{cases}.\]
For the sake of shortening formulas we will also denote $\varepsilon(i)$
by $\varepsilon_i$. 

Now consider an equivalence relation $\nu$ on the set $I$. If two
indices $i$ and $j$ are equivalent under $\nu$, we will write 
$i \sim^\nu j$ or just $i \sim j$ when the equivalence relation is clear 
from the context. The equivalence class of an index $i$ will be denoted 
by $\nu(i)$. Call $\nu$ \textit{unitary} if for any equivalent indices 
$i$ and $j$ the indices $-i$ and $-j$ are also equivalent. All the 
equivalence relations mentioned in this paper are unitary and thus 
we will sometimes omit the word ``unitary''. The index set $I$ can be 
decomposed as a disjoint union of equivalence classes: 
\[ 
    I = C_1 \sqcup C_2 \sqcup \dots \sqcup C_t. 
\] 
We introduce 
a left action of the group $\{\pm 1\}$ on the set $\mathcal{C}l(\nu) = \{C_1, \dots, C_t \}$ of all equivalence 
classes $C_l$ of $\nu$ by putting: 
$1 \cdot \nu(i) = \nu(i)$ and $-1 \cdot \nu(i) = \nu(-i)$, for any $i \in I$. 
Following \cite{VavSympl} we will call the classes stable under this action
\textit{self-conjugate} (i.e. the classes $C_l$ such that for every $i \in C_l$
one has also $-i \in C_l$). Accordingly the non-stable classes will be called
\textit{non-self-conjugate}.
 We will denote by $h(\nu)$ the ordered pair consisting 
of the minimum size (as a set) of all self-conjugate equivalence classes 
of $\nu$ and the minimum size (also as a set) of all non-self-conjugate
equivalence classes of $\nu$. Note, that an arbitrary equivalence relation 
does not necessary have equivalent classes of both types; therefore
$h(\nu)$ is an element in $\mathbb{N} \cup \{ \infty \} \times \mathbb{N} 
\cup \{ \infty \}$. We will always view $\mathbb{N} \cup \{ \infty \} \times \mathbb{N} \cup \{ \infty \}$ as a partially ordered set with the
\textit{product order}, i.e. $(a_1, b_1) \le (a_2, b_2)$ if and only if $a_1 \le a_2$ and $b_1 \le b_2$.

We call a $k$-tuple $(i_1, \dots, i_k)$ of indices in $I$ a \textit{$\C$-type base $k$-tuple 
\textup[of indices\textup]} if for each $1 \le r \neq s \le k$, we have $i_r \neq \pm i_s$ and 
$i_r \sim i_s \sim -i_s \sim -i_r$. 
Similarly, we call a $k$-tuple $(i_1, \dots, i_k)$ of indices in $I$ an \textit{$\A$-type 
base $k$-tuple \textup{[}of indices\textup{]}} if for each $1 \le r \neq s \le k$, we have
$i_r \neq \pm i_s$ and $i_r \sim i_s$. The condition $h(\nu) \geq (a,b)$ 
is equivalent to the condition that every index $i \in I$ can be included 
in either an $\A$-type base $b$-tuple, or a $\C$-type base 
$\left\lceil\frac{a}{2}\right\rceil$-tuple. This simple observation 
will be used repeatedly without specific mention in the rest of the paper.

Another convention concerns the elements of a localization of a ring.
Let $R$ be a unital associative ring and $S$ be a multiplicative set in $R$.
Let $S^{-1} R$ denote the localization of $R$ at $S$. Let $F : R \rightarrow S^{-1} R$
be the localization morphism. Let $\xi \in R$ and $s \in S$. By the fraction $\frac{\xi}{s}$
we will always denote the element $F(\xi) \cdot F(s)^{-1}$ of $S^{-1} R$ in contrast to
$\xi s^{-1}$ which refers to an element of $R$ and only makes sense if $s \in R^*$.
For example $\frac{\xi}{1}$ is synonymic to $F(\xi)$.

Finally, by angular brackets $\langle \cdot \rangle$ we will denote
the subgroups and ideals defined in terms of generators. The rest of the
notations are standard for the field of this research.

\section{Preliminaries}
\label{sec:symplss:prelim}
\paragraph{Symplectic group} Let $R$ be a commutative ring. Fix a natural
number $n$. Let $\Sp(2n,R)$ denote the classical symplectic group of 
rank $2n$. It is known that a matrix $a$ in $\GL(2n, R)$ belongs to
$\Sp(2n, R)$ if and only if the equality 
\begin{equation*}
	%\label{eq:matrix:elements:of:inverse}
	a'_{ij} = \varepsilon_i \varepsilon_j a_{-j,-i}
\end{equation*}
holds for all possible indices.
In future, this property will be used without reference. 
%We will also use 
%the extension to a column or a row rather then single matrix element.
%The extension is given in the following proposition and can be checked by a 
%straightforward calculation.
%
%\begin{proposition}
%    \label{prop:symplss:inverse:row:column}
%    Let $\cdot^t$ denote the transpose operator in matrices.
%    Let $a$ be a matrix in $\Sp(2n,R)$. Fix arbitrary indices $i$ and $j$. Then
%    \begin{align*}
%        a_{i*} &= \varepsilon_i (\mathfrak{p} a'_{*,-i})^t &
%        \text{and} &&
%        a_{*j} &= - \varepsilon_j (a_{-j,*} \mathfrak{p})^t,
%    \end{align*}
%    where $\mathfrak{p} = \sdiag(1,..,1,-1,..,-1)$. Obviously
%    \[ \mathfrak{p}^{-1} = \mathfrak{p}^t = - \mathfrak{p}. \]
%\end{proposition}

Given an element $\xi \in R$ and an index $i \in  I$ we will call 
the matrix 
\[ T_{i,-i}(\xi) = e + \xi e_{i,-i} \]
\textit{the \textup{[}elementary\textup{]} long \textup{[}symplectic\textup{]} transvection}.
Given an additional index $j \neq \pm i$ we will call the matrix
\[ T_{ij}(\xi) = e + \xi e_{ij} - \varepsilon_i \varepsilon_j \xi e_{-j,-i} \]
the \textit{\textup{[}elementary\textup{]} short \textup{[}symplectic\textup{]} transvection}. 

It's a well known fact, that all the long and short elementary symplectic transvections are
containd in $\Sp(2n, R)$ and satisfy the following relations known as the \textit{Steinberg relations}
for all $\xi,\zeta$ in $R$:
\begin{description}
    \item[\mdseries\rmfamily(R1)] $T_{ij}(\xi) = T_{-j,-i}(-\varepsilon_i \varepsilon_j \xi)$ for all $i \neq \pm j$
    \item[\mdseries\rmfamily(R2)] $T_{ij}(\xi) T_{ij}(\zeta) = T_{ij}(\xi + \zeta)$ for all $i \neq j$
    \item[\mdseries\rmfamily(R3)] $[T_{ij}(\xi), T_{hk}(\zeta)] = e$ for all $h \neq j,-i$ and $k \neq i, -j$
    \item[\mdseries\rmfamily(R4)] $[T_{ij}(\xi), T_{jh}(\zeta)] = T_{ih}(\xi \zeta)$ for all $i,h \neq \pm j$ and $i \neq \pm h$
    \item[\mdseries\rmfamily(R5)] $[T_{ij}(\xi), T_{j,-i}(\zeta)] = T_{i,-i}(2 \xi \zeta)$ for all $i$
    \item[\mdseries\rmfamily(R6)] $[T_{i,-i}(\xi), T_{-i,j}(\zeta)] = T_{ij}(\xi \zeta) T_{-j,j}(\varepsilon_i \varepsilon_j \xi \zeta^2)$ for all $i\neq \pm j$.
\end{description}
In future we will occasionally use these relations without a reference. 

For any matrix $g$ in $\Sp(2n,R)$, any short symplectic transvection $T_{sr}(\xi)$ and any 
long symplectic transvection $T_{s,-s}(\zeta)$, we call the matrices ${}^g T_{sr}(\xi) =
g T_{sr}(\xi) g^{-1}$ and ${}^g T_{s,-s}(\zeta) = g T_{s,-s}(\zeta) g^{-1}$
\textit{\textup{[}elementary\textup{]} short and long root elements} respectively.

\paragraph{Equivalence relations and block-diagonal subgroups}

Given a unitary equivalence relation $\nu$ on the index set $I$ we will 
call the subgroup
\[ 
    \Ep(\nu) = \Ep(\nu, R) = 
    \left\langle T_{i,-i}(\xi), T_{jk}(\xi) \;|\; i \sim -i, j \sim k, j \neq \pm k, \xi \in R \right\rangle
\]
the \textit{elementary block-diagonal subgroup of type $\nu$} in $\Sp(2n,R)$.
Note that $\Ep(\nu, R)$ does not necessarily consist of block-diagonal 
matrices. It can only be the case if all (or at least all but one of) the equivalence
classes of $\nu$ are non-self-conjugate. However, these groups behave like block-diagonal matrices
because for each such group $\Ep(\nu, R)$ there exists a permutation matrix $B$ in $\GL(2n, R)$, but
not necessarily in $\Sp(2n, R)$, such that $B \cdot \Ep(\nu, R) \cdot B^{-1}$  is block-diagonal in the usual sense.

From the point of view of Chevalley groups, the elementary block-diagonal
subgroup is precisely \textit{the elementary subsystem subgroup}. Namely, let 
$C_1, -C_{1}, \dots C_{s}, -C_{s}$ be all the non-self-conjugate classes, and
$C_{s+1}, \dots, C_{t}$ be the self-conjugate ones. Set $n_i = | C_i |$ 
for $1 \le i \le s$ and $l_i = |C_i|/2$ for $s+1 \le i \le t$. Then 
$n_1 + \dots + n_s + l_{s+1} + \dots + l_t = n$ and 
\[ 
    \Ep(\nu, R) \approx \E(n_1, R) \times \dots \times \E(n_s, R) \times
    \Ep(2l_{s+1}, R) \times \dots \times \Ep(2l_t, R), 
\]
where $\E(n_i, R)$ denotes the usual elementary subgroup of $\GL(n_i, R)$ 
appearing in the hyperbolic embedding and the product is meant as a 
product of linear groups. From the viewpoint of algebraic groups this is
exactly the elementary Chevalley group of type $\Delta$, where
\[ 
    \Delta = A_{n_1 -1} + \dots + A_{n_s -1} + C_{l_{s+1}} + \dots + C_{l_t}. 
\] 
The reader is referred to \cite{VavSympl} and the references therein for
further details. 

This analogy allows the following geometric interpretation of concepts of $\A$-type and $\C$-type base tuples.
By choosing a unitary equivalence relation $\nu$ we fix a
subsystem $\Delta \le C_n$ consisting of all roots $\alpha_{ij} \in C_n$ where $i \sim^\nu j$.
Irreducible components of $\Delta$ are in one to one correspondence with equivalence classes of $\nu$,
namely the components of type $C_l$ correspond to self-conjugate equivalence classes and the components
of type $A_l$ correspond to pairs of non-self-conjugate classes.
Then an A- or C- type base $k$-tuple $(i_1, \dots, i_k)$ provides us with a root subsystem in $\Delta$
of type $A_{k-1}$ or $C_k$ respectively, namely
$\left\langle \alpha_{i_1, i_2}, \dots, \alpha_{i_{k-1}, i_k} \right\rangle$ or 
$\left\langle \alpha_{i_1, i_2}, \dots, \alpha_{i_{k-1}, i_k}, \alpha_{i_k} \right\rangle$, respectively.
Moreover, both generating sets above can be chosen as systems (or bases) of simple roots in the subsystems
they generate.

\paragraph{Form nets of ideals and corresponding groups}
Consider a square array $\sigma = (\sigma_{ij})_{i,j \in I}$ of $(2n)^2$
ideals of the ring $R$. We will call it \textit{a net of ideals over} $R$ 
if for any indices $i, j$ and $k$, we have the following inclusions:
\[ 
    \sigma_{ik} \sigma_{kj} \le \sigma_{ij}.
\]
A net of ideals $\sigma$ is called \textit{unitary}, if 
$\sigma_{ij} = \sigma_{-j,-i}$ for each $i$ and $j$. We will call
$\sigma$ a \textit{$\D$-net}, if $\sigma_{ii} = R$ for every $i$ in $I$. 
Equip the net of ideals $\sigma$ with $2n$ additive subgroups 
$\Gamma = (\Gamma_i)_{i \in I}$ of $R$ such that
for any indices $i,j \in I$ the following inclusions hold:
\begin{enumerate}
    \item $2 \sigma_{i,-i} \le \Gamma_i \le \sigma_{i,-i}$
    \item $\sigma_{ij}^{\rectangled{2}} \Gamma_j \le \Gamma_i$,
\end{enumerate}
where $2 \sigma_{i,-i} = \{2\alpha \;|\; \alpha \in \sigma_{i,-i} \}$ and
$\sigma_{ij}^{\rectangled{2}} = \{ \xi^2 | \xi \in \sigma_{ij} \}$.
In this situation $\Gamma$ is called \textit{a column of form parameters for
$\sigma$} and the pair $(\sigma, \Gamma)$ \textit{a form net of ideals [over $R$]}. 
It is the analogue for nets of ideals of the concept of form ideal of Bak \cite{BakThesis} for form rings.
A form net of ideals $(\sigma, \Gamma)$ is said to be \textit{exact} if 
for any index $i$ the equality
\[
    \sigma_{i,-i} = \sum_{k \neq \pm i} \sigma_{ik}\sigma_{k,-i} +
    \left\langle\Gamma_i\right\rangle
\]
holds. From now on we only deal with unitary exact form $\D$-nets of ideals.

Introduce a partial ordering on the set of all form nets of ideals over $R$
by setting $(\sigma', \Gamma') \le (\sigma'', \Gamma'')$ whenever for all 
$i,j \in I$ the inclusions 
$\sigma'_{ij} \le \sigma''_{ij}$ and $\Gamma'_i \le \Gamma''_i$ hold.
As a matter of convenience, given an element $\xi \in R$ and indices 
$s$ and $r$ we will write ``$\xi \in (\sigma, \Gamma)_{sr}$'' instead of
``$\xi \in \sigma_{sr}$ if $r \neq -s$ and $\xi \in \Gamma_s$ otherwise''.
\label{conv:sigma:Gamma:sr}

We can associate two kinds of subgroups of $\Sp(2n,R)$ to each form net 
of ideals $(\sigma, \Gamma)$  over $R$. We call
the subgroup
\[ 
    \Ep(\sigma, \Gamma) = \left\langle T_{ij}(\xi), T_{i,-i}(\alpha) 
    \;|\; i \neq \pm j, \xi \in \sigma_{ij}, \alpha \in \Gamma_{i} \right\rangle
\] 
\textit{the elementary form net subgroup of level $(\sigma, \Gamma)$}.
We call the above generators of $\Ep(\sigma, \Gamma)$ \textit{the short and long $(\sigma, \Gamma)$-elementary symplectic transvections}, respectively. Note that any unitary equivalence relation $\nu$ on the set of indices $I$ defines a form net 
\[
    [\nu]_R = (\sigma_\nu, \Gamma_\nu),
\] where 
\begin{align*}
    (\sigma_\nu)_{ij} &= \begin{cases} R, & \text{ if } i \sim^\nu j \\
    	 0 , & \text{ if } i \nsim^\nu j \end{cases} &
    (\Gamma_\nu)_{i} &= \begin{cases} R , & \text{ if } i \sim^\nu -i \\ 	
    	0 & i \nsim^\nu -i \end{cases}.
\end{align*}
This is clearly a $\D$-net.
Thus the elementary block-diagonal subgroup $\Ep(\nu, R) = \Ep(\sigma_\nu, \Gamma_\nu)$ is a special case of an elementary form net subgroup.
We will call a form net of ideals $(\sigma, \Gamma)$ \textit{major \textup{[}with respect to $\nu$\textup{]}} if
$[\nu]_{R} \le (\sigma, \Gamma)$.
The notion of elementary form net subgroup is a generalization of that of relative elementary
subgroup of even unitary groups introduced in Bak \cite[p. 66]{BakAnn}. The concept of the relative
principal congruence subgroups therein is generalized as follows. We will call the subgroup
\[ 
    \Sp(\sigma, \Gamma) = \{ g \in \Sp(2n,R) \;|\; \forall i,j \in I \; g_{ij} \in \sigma_{ij}, S_{i,-i}(g) \in \Gamma_i \},
\]
\textit{the form net subgroup of level $(\sigma, \Gamma)$},
where 
\[
    S_{i,-i}(g) = \sum_{j > 0} g_{ij} g'_{j,-i}
\] is the so called \textit{length of the row $g_{i*}$}. 
The element $S_{i,-i}(g)$ is clearly in $\sigma_{i,-i}$, by definition of a net of ideals, and 
insisting that $S_{i,-i}(g) \in \Gamma_i \le \sigma_{i,-i}$ is a further restrinction on $g$.
The word ``length'' was introduced by You in \cite{YouNorm2012}.

In the situation when $\Gamma_i = R$ for all $i$ the form net subgroup coincides with the
regular \textit{net subgroup} $\Sp(\sigma) = \{ g \in \Sp(2n,R) \;|\; \forall i,j \in I \; g_{ij} \in \sigma_{ij} \}$.

Next we compute the lengths of rows of certain products of matrices. An immediate reason to do so is
to prove that the above defined set $\Sp(\sigma, \Gamma)$ is indeed a group. Besides, these computations
are used repeatedly in most of the proofs in this paper. Despite of the importance of the proposition
below and it's corollaries, it's proof is quite routine and thus is omitted (cf. \cite{ShchegBielThesis} 
for the proof).

\begin{proposition}
    \label{prop:symplss:length:of:product}
    Let $a$ and $b$ be two matrices in $\Sp(2n, R)$. Then
    \begin{align*}
        S_{i,-i}(ab) &=
        S_{i,-i}(a) +
        \sum_{k} a_{ik} S_{k,-k}(b) a'_{-k,-i} - \\
        &\quad
        -2 \sum_{j,k,l > 0} a_{i,l} b_{l,-j} b'_{-j, k} a'_{k,-i} - \\
        &\quad
        -2 \sum_{j,k>0} \sum_{l > k} (a_{i,-k} b_{-k,-j} b'_{-j,l} a'_{l,-i} +
        a_{ik} b_{k,-j} b'_{-j,-l} a'_{-l,-i}).
    \end{align*}
\end{proposition}

\begin{corollary}
    \label{cor:symplss:sp:sigma:gamma:is:a:group}
    Let $(\sigma, \Gamma)$ be a form net of ideals over $R$. Suppose $a, b \in \Sp(\sigma)$,
    then
    \begin{equation}
        \label{eq:cor:symplss:sp:sigma:gamma:is:a:group:statement}
        S_{i,-i}(ab) \equiv S_{i,-i}(a) +
        \sum_{k} a^2_{ik} S_{k,-k}(b) \mod \Gamma_i.        
    \end{equation}
    In particular, the  form net subgroup $\Sp(\sigma, \Gamma)$ is indeed a group.
    \begin{proof}
        Clearly $e \in \Sp(\sigma, \Gamma)$. Next the congruences 
        \eqref{eq:cor:symplss:sp:sigma:gamma:is:a:group:statement} together with
        the condition $\sigma_{ik}^{\rectangled{2}} \Gamma_k \le \Gamma_i$ show that
        $\Sp(\sigma, \Gamma)$ is closed under taking products.
        Let $a \in \Sp(\sigma, \Gamma)$. By
        \eqref{eq:cor:symplss:sp:sigma:gamma:is:a:group:statement}
        we get that
        \[
            0 = S_{i,-i}(e) = S_{i,-i}(a^{-1} a) \equiv S_{i,-i}(a^{-1}) 
            + \sum_{k} (a'_{ik})^2 S_{k,-k}(a) \equiv S_{i,-i}(a^{-1}) \mod \Gamma_i.
        \]
        Therefore $a^{-1} \in \Sp(\sigma, \Gamma)$ and $\Sp(\sigma,\Gamma)$ is a group.
    \end{proof}
\end{corollary}

The following two corollaries allow us to compute the lengths of rows in products of matrices in $\Sp(\sigma)$
and short $(\sigma, \Gamma)$-elementary symplectic transvection as well as lengths of rows of some root elements.

\begin{corollary}
    \label{cor:symplss:length:of:left:mult:by:transv}
    Let $a \in \Sp(\sigma)$ and $T_{pq}(\xi)$ be a short elementary symplectic transvection in $\Ep(\sigma, \Gamma)$.
    Then \[
        S_{i,-i}(T_{pq}(\xi) a) \equiv
        \begin{cases}
            S_{i,-i}(a) & \text{ if } i \neq p, -q \\
            S_{p,-p}(a) + \xi^2 S_{q,-q}(a) & \text{ if } i = p \\
            S_{-q,q}(a) + \xi^2 S_{-p,p}(a) & \text{ if } i = -q
        \end{cases}
        \mod \Gamma_{i}
    \]
    and for all $i \in I$ 
    \[ S_{i,-i}(a T_{pq}(\xi)) \equiv S_{i,-i}(a) \mod \Gamma_i.\]
\end{corollary}

\begin{corollary}
    \label{cor:symplss:length:of:root:element}
    Let $a \in \Sp(\sigma)$, $T_{sr}(\xi), T_{st}(\zeta)$ be short elementary symplectic transvections
    in $\Ep(\sigma)$ and $s \neq \pm r, \pm t$ and $r \neq \pm t$. 
    Then
    \begin{align*}
        S_{i,-i}(a T_{sr}(\xi) T_{st}(\zeta) a^{-1}) &\equiv
                a_{is}^2 \zeta^2 S_{t,-t}(a^{-1}) +
                a_{is}^2 \xi^2 S_{r,-r}(a^{-1}) + \\
                &+
                a_{i,-t}^2 \zeta^2 S_{-s,s}(a^{-1}) +
                a_{i,-r}^2 \xi^2 S_{-s,s}(a^{-1}) \mod \Gamma_i.
    \end{align*}
    In particular if $\zeta = 0$ then $T_{st}(\zeta) = e$ and
    \[
        S_{i,-i}(a T_{sr}(\xi) a^{-1}) \equiv a_{is}^2 \xi^2 S_{r,-r}(a^{-1}) 
                + a_{i,-r}^2 \xi^2 S_{-s,s}(a^{-1}) \mod \Gamma_i.    
    \]
%    \begin{proof}
%        Clearly $a T_{sr}(\xi) T_{st}(\zeta) a^{-1} \in \Sp(\sigma)$. 
%        Then combining Proposition \ref{prop:symplss:length:of:product}
%        and Corollary \ref{cor:symplss:length:of:left:mult:by:transv} 
%        we get 
%        \begin{align*}
%            S_{i,-i}(a T_{sr}(\xi) T_{st}(\zeta) a^{-1}) &\equiv
%                S_{i,-i}(a) + \sum_{k} \varepsilon_i \varepsilon_k a_{ik}^2 S_{k,-k}(T_{sr}(\xi) T_{st}(\zeta) a^{-1}) \\
%            &\equiv
%                S_{i,-i}(a) + \sum_k \varepsilon_i \varepsilon_k a_{ik}^2 
%                (S_{k,-k}(a^{-1}) + \\ &\quad + \delta_{ks} \zeta^2 S_{t,-t}(a^{-1}) 
%                                  + \delta_{ks} \xi^2 S_{r,-r}(a^{-1}) +\\
%                                  &\quad 
%                                  + \delta_{k,-t} \zeta^2 S_{-s,s}(a^{-1})
%                                  + \delta_{k,-r} \xi^2 S_{-s,s}(a^{-1})) \\
%            &\equiv S_{i,-i} (a \cdot a^{-1}) +
%                \varepsilon_i \varepsilon_s a_{is}^2 \zeta^2 S_{t,-t}(a^{-1}) +
%                \varepsilon_i \varepsilon_s a_{is}^2 \xi^2 S_{r,-r}(a^{-1}) + \\
%                &\quad +
%                \varepsilon_i \varepsilon_{-t} a_{i,-t}^2 \zeta^2 S_{-s,s}(a^{-1}) +
%                \varepsilon_i \varepsilon_{-r} a_{i,-r}^2 \xi^2 S_{-s,s}(a^{-1}).
%        \end{align*}
%        Clearly $S_{i,-i}(e) = 0$ for all $i$ and the signs are insignificant due to the inclusion
%        $2 \sigma_{i,-i} \le \Gamma_i$. Therefore we get the required congruence.
%    \end{proof}
\end{corollary}

We finish this section with two technical results which will be used repeatedly and without
reference in proofs throughout the paper. The first one shows that
major form nets of ideals are partitioned into blocks in which all ideals are equal and all form parameters are equal.
The second one allows simplifying reasoning dealing with case-by-case analysis 
of small equivalence classes. The proofs of both results can be checked straightforwardly and are left to the reader.

\begin{proposition}
    \label{prop:symplss:change:indices}
    Let $(\sigma, \Gamma)$ be a form net of ideal such that 
    $\Ep(\nu, R) \le \Ep(\sigma, \Gamma)$, with $h(\nu) \geq (4,3)$. Then for 
    any indices $i, j, k, l$ such that $k \sim i$, $l \sim j$, we have:
    \begin{enumerate}
        \item $ \sigma_{kj} = \sigma_{ij} = \sigma_{il}$
        \item $ \Gamma_i = \Gamma_k$.
    \end{enumerate}
\end{proposition}

\begin{proposition}
    \label{prop:symplss:three:indices}
    Let $\nu$ be a unitary equivalence relation on the index set 
    $I$ such that $h(\nu) \geq (4,3)$. 
    Let $i, j$ be two indices in $I$ such that $i \neq j$.     
    Then
    one of the following holds:
    \begin{enumerate}
        \item $\nu(i) = \{i, -i, j, -j\}$
        \item There exists an index $k$ in $I$ such that 
        $k \neq \pm i, \pm j$ and $k \sim^\nu i$.
    \end{enumerate}
%    \begin{proof}
%        Indeed, let $i \nsim^\nu j$. If the class $\nu(i)$ is self-adjoint
%        and the only member of $\nu(i)$ different from $i$ is $\pm j$,
%        then, obviously, $j \sim^\nu i$ which contradicts with our
%        assertion. If the class $\nu(i)$ is not self-adjoint, then
%        a) it contains at least 3 elements b) for any element $t \in \nu(i)$,
%        the element $-t$ lies in the opposite class and thus doesn't lie
%        in $\nu(i)$ (otherwise, $\nu(i)$ would be self-adjoint). Thus,
%        there exists an index $k$ in $\nu(i)$ satisfying the second alternative.
%        
%        Now let $i \sim^\nu j$ and $\nu(i) \neq \{i, -i, j, -j\}$. Then 
%        $\nu(i)$ is either a not self-adjoint class of order at least 3
%        and we already now, that the required $k$ exists, or is a self-adjoint
%        class of order at least 6 and the required $k$ can be chosen as any
%        element of $\nu(i) \setminus \{i, -i, j, -j\}$.
%    \end{proof}
\end{proposition}

\section{Form net associated with a subgroup and the description of the transporter}
\label{sec:symplss:assoc:net:and:norm}

\paragraph{Form net of ideals associated with a subgroup.}
Let $\nu$ be a unitary equivalence relation on the index set $I = \{1, \dots, n, -n, \dots, -1\}$ 
such that $h(\nu) \geq (4,3)$. Let 
$H$ be a subgroup of $\Sp(2n,R)$ such that $\Ep(\nu, R) \le H$. 
An exact form net of ideals $(\sigma, \Gamma)$ is called \textit{the form net [of ideals] associated with $H$} if
$\Ep(\sigma, \Gamma) \le H$ and if for any exact form net of ideals $(\sigma', \Gamma')$ 
such that $\Ep(\sigma', \Gamma') \le H$, it follows that $(\sigma', \Gamma') \le (\sigma, \Gamma)$.
Clearly, if $(\sigma, \Gamma)$ exists then it is unique. The next lemma shows that
$(\sigma, \Gamma)$ exists. 

\begin{lemma}
    \label{lemma:symplss:ass:net:is:a:net}
    Let $\nu$ be a unitary equivalence relation on $I$ such that 
    $h(\nu) \geq (4,3)$ and $H$ a subgroup of $\Sp(2n, R)$ that contains the 
    subgroup $\Ep(\nu, R)$. For each $i \neq \pm j \in I$ set
    \begin{equation*}
       % \label{eq:lemma:symplss:ass:net:is:a:net:def}
        \begin{aligned}
            \sigma_{ij} = \{\xi \in R \;|\; T_{ij}(\xi) \in &H \}, 
            \quad        
            \Gamma_i = \{\alpha \in R \;|\; T_{i,-i}(\alpha) \in H \},  
            \quad
            \sigma_{ii} = R, \\
            &\sigma_{i,-i} = \sum_{j \neq \pm i}\limits \sigma_{ij}\sigma_{j,-i} + \left< \Gamma_i \right>. 
        \end{aligned}
    \end{equation*}
    Then $(\sigma, \Gamma)$ is the form net of ideals associated with $H$.
\end{lemma}

We won't prove this lemma in the present form. In section
\ref{sec:spgen:stsetting} we will prove a slightly stronger version of this result,
Proposition \ref{prop:spgen:sigma:g:is:almost:a:net}. For the sake of simplicty, the proof
Proposition \ref{prop:spgen:sigma:g:is:almost:a:net} relies on Lemma 
\ref{lemma:symplss:ass:net:is:a:net}. This does not create a loop in theory, as we can almost
literaly repeat the proof of Proposition \ref{prop:spgen:sigma:g:is:almost:a:net} for
Lemma \ref{lemma:symplss:ass:net:is:a:net}. This is done explicitly in 
\cite[Chapter 1, Lemma 1.2.1]{ShchegBielThesis}.

\paragraph{Description of the transporter.}
The rest of this section is devoted to the proof of Theorem \ref{theorem:symplss:transpdescr}. 
The following proposition allows to compute lengths of rows of conjugates of an arbitrary matrix
by a matrix satisfying property (T1) of Theorem \ref{theorem:symplss:transpdescr}.

\begin{proposition}
    \label{prop:spgen:conj:N:congr}
    Let $(\sigma, \Gamma)$ be an exact form net of ideals. Suppose $a \in \Sp(2n,R)$ satisfies the condition 
    \[ 
        a_{ij} \sigma_{jk} a'_{kl} \le \sigma_{il}
    \]
    for all $i,j,k,l \in I$. Then for any matrix $g \in \Sp(\sigma,\Gamma)$
    and any $i \in I$ the following congruence holds:
%    \[
%        S_{i,-i}(a g a^{-1}) \equiv
%        S_{i,-i}(a) + \sum_{k \in I} a_{ik}^2 
%                    \left( S_{k,-k}(g) + \sum_{t \in I} g_{kt}^2 S_{t,-t}(a^{-1}) \right)
%        \mod \Gamma_i.
%    \]
    \[
        S_{i,-i}(a g a^{-1}) \equiv
        	\sum_{k \in I} a_{ik}^2 
            \left( S_{k,-k}(g) + 
            S_{k,-k}(a^{-1}) +
            \sum_{t \in I} g_{kt}^2 S_{t,-t}(a^{-1})            
            \right)
        \mod \Gamma_i.
    \]
    \begin{proof}
        By Proposition \ref{prop:symplss:length:of:product} we get
        \begin{equation}
            \label{eq:prop:spgen:conj:N:congr:1}
            \begin{aligned}
                S_{i,-i}(a g a^{-1}) &= S_{i,-i}(a) + \sum_{k \in I} \varepsilon_i \varepsilon_k a_{ik}^2 
                    S_{k,-k}(g a^{-1}) \\
                &\quad - 2 \sum_{j,k,l > 0} a_{il} (g a^{-1})_{l,-j} (a g^{-1})_{-j,k} a'_{k,-i} \\
                &\quad - 2 \sum_{j,k > 0} \sum_{l > k} 
                    ( a_{i,-k} (g a^{-1})_{-k,-j} (a g^{-1})_{-j,l} a'_{l,-i} + \\
                &\quad\quad\quad\quad\quad\quad\;\;\;
                      a_{ik} (g a^{-1})_{k,-j} (a g^{-1})_{-j,-l} a'_{-l,-i}).
            \end{aligned}
        \end{equation}
        Consider an individual summand of the second big sum above. By the assumption that $a_{ij} \sigma_{jk} a'_{kl} \le \sigma_{il}$
        we get 
        \begin{equation}
            \label{eq:prop:spgen:conj:N:congr:1.1}
            a_{il} (g a^{-1})_{l,-j} (a g^{-1})_{-j,k} a'_{k,-i} =
            \sum_{p,q \in I} (a_{il} g_{lp} a'_{p,-j}) (a_{-j,q} g'_{qk} a'_{k,-i})
            \le \sigma_{i,-j} \sigma_{-j,-i} \le \sigma_{i,-i}
        \end{equation}
        and therefore the doubled second big sum in \eqref{eq:prop:spgen:conj:N:congr:1} is contained
        in $2 \sigma_{i,-i} \le \Gamma_i$. Applying the same principle to the last summand
        in \eqref{eq:prop:spgen:conj:N:congr:1} we get the inclusion
        \begin{equation}
            \label{eq:prop:spgen:conj:N:congr:1.2}
            a_{i,-k} (g a^{-1})_{-k,-j} (a g^{-1})_{-j,l} a'_{l,-i} + 
            a_{ik} (g a^{-1})_{k,-j} (a g^{-1})_{-j,-l} a'_{-l,-i}) \in \sigma_{i,-i}.
        \end{equation}
        Combining \eqref{eq:prop:spgen:conj:N:congr:1}, \eqref{eq:prop:spgen:conj:N:congr:1.1}
        and \eqref{eq:prop:spgen:conj:N:congr:1.2} we get
        \begin{equation}
            \label{eq:prop:spgen:conj:N:congr:2}
            S_{i,-i}(a g a^{-1}) \equiv S_{i,-i}(a) + \sum_{k \in I} \varepsilon_i \varepsilon_k a_{ik}^2 
                    S_{k,-k}(g a^{-1}) \mod \Gamma_i.
        \end{equation}
        Expand \eqref{eq:prop:spgen:conj:N:congr:2} further using Proposition 
        \ref{prop:symplss:length:of:product}:
        \begin{equation}
            \label{eq:prop:spgen:conj:N:congr:3}
            \begin{aligned}
                S_{i,-i}(a g a^{-1}) 
                &\equiv S_{i,-i}(a) + \sum_{k \in I} \varepsilon_i \varepsilon_k a_{ik}^2 S_{k,-k}(g a^{-1}) \\
                &\equiv S_{i,-i}(a) + \sum_{k \in I} \varepsilon_i \varepsilon_k a_{ik}^2 
                    \left( S_{k,-k}(g) + \sum_{t \in I} \varepsilon_k \varepsilon_t g_{kt}^2 S_{t,-t}(a^{-1}) \right. \\
                &\quad \left. - 2 \sum_{j,t,l > 0} g_{kl} a'_{l,-j} a_{-j,t} g'_{t,-k} \right. \\
                &\quad \left. - 2 \sum_{j,t > 0} \sum_{l > t} \left(
                    g_{k,-t} a'_{-t,-j} a_{-j,l} g'_{l,-k} +
                    g_{it} a'_{t,-j} a_{-j,-l} g'_{-l,-l}
                    \right) \right) \\
                &\equiv S_{i,-i}(a) + \sum_{k \in I} \varepsilon_i \varepsilon_k a_{ik}^2 
                    \left( S_{k,-k}(g) + \sum_{t \in I} \varepsilon_k \varepsilon_t g_{kt}^2 S_{t,-t}(a^{-1}) \right) \\
                &\quad - 2 \sum_{j,t,l > 0} a_{ik} g_{kl} a'_{l,-j} a_{-j,t} g'_{t,-k} a'_{-k,-i} \\
                &\quad - 2 \sum_{j,t > 0} \sum_{l > t} \left(
                    a_{ik} g_{k,-t} a'_{-t,-j} a_{-j,l} g'_{l,-k} a'_{-k,-i} \right. \\
                &\qquad\qquad\quad
                    \left. + a_{ik} g_{it} a'_{t,-j} a_{-j,-l} g'_{-l,-l} a'_{-k,-i} \right) \\
                &\mod \Gamma_i.
            \end{aligned}
        \end{equation}
        Using the same trick as before we may conclude that both doubled terms of
        \eqref{eq:prop:spgen:conj:N:congr:3} are contained in $2 \sigma_{i,-i} \le \Gamma_i$.
        Summing up, we get the congruence
        \begin{equation}
            \label{eq:prop:spgen:conj:N:congr:4}
            \begin{aligned}
                S_{i,-i}(a g a^{-1}) 
                &\equiv S_{i,-i}(a) + \sum_{k \in I} \varepsilon_i \varepsilon_k a_{ik}^2 
                    \left( S_{k,-k}(g) + \sum_{t \in I} \varepsilon_k \varepsilon_t g_{kt}^2 S_{t,-t}(a^{-1}) \right) 
                    \mod \Gamma_i.        
            \end{aligned}
        \end{equation}
        Finally it's easy to see that 
        \[
            \varepsilon_i \varepsilon_k a_{ik}^2 S_{k,-k}(g) = \sum_{l > 0, p \in I} 
                (a_{ik} g_{kl} a'_{lp})(a_{pl} g'_{l,-k} a'_{-k,-i}) \in \sigma_{i,-i}
        \] and
        \[
            \varepsilon_i \varepsilon_k \varepsilon_k \varepsilon_t 
                a_{ik}^2 g_{kt}^2 S_{t,-t}(a^{-1})
            = \sum_{l > 0} (a_{ik} g_{kt} a'_{tl})(a'_{l,-t} g'_{-t,-k} a'_{-k,-i}) \in \sigma_{i,-i}.
        \]
        Therefore the choice of signs in
        \eqref{eq:prop:spgen:conj:N:congr:4} is insignificant and we can 
        rewrite \eqref{eq:prop:spgen:conj:N:congr:4} as follows
        \begin{equation}
            \label{eq:prop:spgen:conj:N:congr:5}
            \begin{aligned}
                S_{i,-i}(a g a^{-1}) 
                &\equiv S_{i,-i}(a) + \sum_{k \in I} a_{ik}^2 
                    \left( S_{k,-k}(g) + \sum_{t \in I} 
                    g_{kt}^2 S_{t,-t}(a^{-1}) \right) 
                    \mod \Gamma_i.        
            \end{aligned}
        \end{equation}        
        It' clear that $S_{i,-i}(e) = 0$ for all $i$. Rewrite the formula
        \eqref{eq:prop:spgen:conj:N:congr:5} for $g = e$. As 
        $S_{i,-i}(a) = \sum_{j > 0} a_{ij} \delta_{jj} a'_{j,-i} \le
        \sigma_{i,-i}$, we can also change the sign at the first term:
        \begin{equation}
            \label{eq:prop:spgen:conj:N:congr:6}
            \begin{aligned}
                0 = S_{i,-i}(a \cdot a^{-1}) 
                &\equiv - S_{i,-i}(a) + \sum_{k \in I} a_{ik}^2 
                    S_{k,-k}(a^{-1}) \mod \Gamma_i.        
            \end{aligned}
        \end{equation}        
        Finally, adding \eqref{eq:prop:spgen:conj:N:congr:6} to
        \eqref{eq:prop:spgen:conj:N:congr:5} we get the required inclusion 
        \begin{equation*}
            \begin{aligned}
                S_{i,-i}(a g a^{-1}) 
                &\equiv \sum_{k \in I} a_{ik}^2 
                    \left( S_{k,-k}(g) + S_{k,-k}(a^{-1}) + \sum_{t \in I} 
                    g_{kt}^2 S_{t,-t}(a^{-1}) \right) 
                    \mod \Gamma_i.        
            \end{aligned}
        \end{equation*}                 
        This completes the proof.
    \end{proof}        
\end{proposition}

\begin{proof}[\textbf{\textup{Proof of Theorem \ref{theorem:symplss:transpdescr}}}]
   	Denote by $N$ the set of all matrices in $\Sp(2n,R)$ satisfying the
   	conditions (T1) -- (T3). It's easy to see that 
   	$N \le \N_{\Sp(2n, R)}(\Sp(\sigma, \Gamma))$. Indeed, pick any
   	$g \in \Sp(\sigma, \Gamma)$ and any $a \in N$. Then condition
   	(T1) guarantees that 
   	\[
   		(a g a^{-1})_{ij} = \sum_{p,q \in I} a_{ip} g_{pq} a'_{qj}    	
   		\le \sum_{p,q \in I} a_{ip} \sigma_{pq} a'_{qj} \le \sigma_{ij}
   	\]
   	for all $i,j \in I$. Now applying Proposition
   	\ref{prop:spgen:conj:N:congr} we get
   	\[
        S_{i,-i}(a g a^{-1}) \equiv
   	    	\sum_{k \in I} \left( a_{ik}^2 
       	     S_{k,-k}(g) + 
           	a_{ik}^2 S_{k,-k}(a^{-1}) +
            \sum_{t \in I} a_{ik}^2 g_{kt}^2 S_{t,-t}(a^{-1})            
   	        \right)
        \mod \Gamma_i.
       \]
	Observe that by condition (T3) it follows that 
	$a_{ik}^2 S_{k,-k}(g^{-1}) \in a_{ik}^2 \Gamma_k \le \Gamma_i$.
	Next, by condition (T2) we get 
	$a_{ik}^2 g_{kt}^2 S_{t,-t}(a^{-1}) \in 
	a_{ik}^2 \sigma_{kt}^\rectangled{2} S_{t,-t}(a^{-1}) \le \Gamma_i$
    and
    $a_{ik}^2 S_{k,-k}(a^{-1}) = a_{ik}^2 \cdot 1^2 \cdot S_{k,-k}(a^{-1})
    \in a_{ik}^2 \sigma^2_{kk} S_{k,-k}(a^{-1}) \le \Gamma_i$. Therefore,
    $S_{i,-i}(aga^{-1})) \in \Gamma_i$ for all $i$. It follows that
    $aga^{-1} \in \Sp(\sigma, \Gamma)$ and thus 
    $a \in \N_{\Sp(2n, R)}(\Sp(\sigma, \Gamma))$.
        
    The proof of the inclusion $\Transp_{\Sp(2n, R)}(\Ep(\sigma, \Gamma),
    \Sp(\sigma, \Gamma)) \le N$ is slightly trickier. Consider an 
    arbitrary matrix $a$ in $\Transp_{\Sp(2n,R)}(\Ep(\sigma, \Gamma),
    \Sp(\sigma))$ and a short $(\sigma,\Gamma)$-elementary transvection 
    $T_{rs}(\xi)$. By definition of transporter we get
	\begin{equation}
		\label{eq:norm:descr:pair}
		\delta_{ij} + a_{ir} \xi a'_{sj} -  \varepsilon(r)\varepsilon(s) 
		a_{i,-s} \xi a'_{-r, j} 
		= ({}^a T_{rs}(\xi))_{ij} \in \sigma_{ij}.
	\end{equation}
	Now given two short $(\sigma, \Gamma)$-elementary transvections $T_{rs}(\xi)$ and
	$T_{st}(\zeta)$ such that $r \neq \pm t$ we get by a straightforward
	computation
	\begin{equation}
	    \label{eq:norm:descr:short:short}
	    \delta_{ij} + a_{ir} \xi \zeta a'_{tj} = 
	    (a T_{rs}(\xi) T_{st}(\zeta) a^{-1})_{ij} - 
	    (a T_{rs}(\xi) a^{-1})_{ij} -
	    (a T_{st}(\zeta) a^{-1})_{ij} + \delta_{ij}.
	\end{equation}
	And therefore using \eqref{eq:norm:descr:pair} we get the inclusions
	$a_{ir} \sigma_{rs} \sigma_{sr} a'_{tj} \le \sigma_{ij}$ for all
	$i,j,s,r,t \in I$ such that $s \neq \pm r, \pm t$ and $r \neq \pm t$.

	Next for a long $(\sigma, \Gamma)$-elementary transvection $T_{s,-s}(\alpha)$ we get
    \begin{equation}
    	\label{eq:norm:descr:long}
        \delta_{ij} + a_{is} \alpha a'_{-s,j} = ({}^a T_{s,-s}(\alpha))_{ij}
        \in \sigma_{ij}.
    \end{equation}
	Finally for $r,s \in I$ such that $r \neq \pm s$ we get
    \begin{equation}
    	\label{eq:norm:descr:short:long}
        \delta_{ij} + a_{ir} \xi \alpha a'_{-s,j} = 
        (a T_{rs}(\xi) T_{s,-s}(\alpha) a^{-1})_{ij} - 
		(a T_{rs}(\xi) a^{-1})_{ij} -
		(a T_{s,-s}(\alpha) a^{-1})_{ij} + \delta_{ij}
    \end{equation}
    and therefore by \eqref{eq:norm:descr:long} we get the inclusions
	$a_{ir} \sigma_{rs} \Gamma_{s,-s} a'_{rj} \le \sigma_{ij}$ for all 
	$i,j \in I$ and all $r,s \in I$ such that $s \neq \pm r$.
        
	Now let $r$ and $t$ be two indices such that $r \neq \pm t$. Then
    either $\nu(r) = \{\pm r, \pm t \}$, or there exists an index
    $s \sim t$ such that $s \neq \pm r, \pm t$. In the former case
    using \eqref{eq:norm:descr:short:long} we get
    \[
    	a_{ir} \sigma_{rt} a'_{tj} = a_{ir} R \sigma_{rt} a'_{tj} =
        a_{ir} \sigma_{r,-t} \Gamma_{-t,t} a'_{tj} \le \sigma_{ij}
    \]
	for all $i,j \in I$. In the latter case using
    \eqref{eq:norm:descr:short:short} we get
    \[
    	a_{ir} \sigma_{rt} a'_{tj} = a_{ir} \sigma_{rt} R a'_{tj} =
        a_{ir} \sigma_{rs} \sigma_{st} a'_{tj} \le \sigma_{ij}
	\]
    for all $i, j \in I$.
        
    Now assume $t = r$. Then there exists an index $s \sim r$ such that
    $s \neq \pm r$ and using the fact that $\sum_{l \in I} a'_{sl} a_{ls} = 1$
    we get
	\[ 
    	a_{ir} \sigma_{rr} a'_{rj} = \sum_{l \in I} (a_{ir} \sigma_{rs}
    	a'_{sl})
       (a_{ls} \sigma_{sr} a'_{rj}) \le \sum_{l \in I} \sigma_{il} 
       \sigma_{lj} \le \sigma_{ij}. 
	\]
    Finally if $t = -r$ then $\sigma_{r,-r} = \sum_{l \neq \pm r}
	\sigma_{rl} \sigma_{l,-r} + \left\langle \Gamma_r \right\rangle$.
	By \eqref{eq:norm:descr:long} it follows that
	\[ 
		a_{ir} \left\langle \Gamma_r \right\rangle a'_{-r,j} \le 
    	\sigma_{ij}. 
	\]
	It's only left to notice that
	\[ 
		a_{ir} \sigma_{rl} \sigma_{l,-r} a'_{rj} = \sum_k (a_{ir}
    	\sigma_{rl} a'_{lk})
        (a_{kl} \sigma_{l,-r} a'_{-r,j}) \le \sum_k \sigma_{ik} \sigma_{kj}
        \le \sigma_{ij}. 
	\]   
    Therefore, any matrix $a$ in the transporter satisfies condition
    (T1). In particular, we can apply Proposition
    \ref{prop:spgen:conj:N:congr} to any such matrix $a$. 
        
    Pick any short $(\sigma, \Gamma)$-elementary transvection $T_{jk}(\xi)$. By 
    Proposition \ref{prop:spgen:conj:N:congr} we get        
	\begin{equation}
    	\label{eq:theorem:spgen:transporter:descr:short}
        S_{i,-i}(a T_{jk}(\xi) a^{-1}) \equiv a_{ij}^2 \xi^2 S_{k,-k}(a^{-1}) +
        a_{i,-k}^2 \xi^2 S_{-j,j}(a^{-1}) \mod \Gamma_i.
    \end{equation}
	Now, given a long $(\sigma, \Gamma)$-elementary transvection $T_{j,-j}(\alpha)$
	we obtain by the same proposition
    \begin{equation}
    	\label{eq:theorem:spgen:transporter:descr:long}
        S_{i,-i}(a T_{j,-j}(\alpha) a^{-1}) \equiv a_{ij}^2 \alpha + 
        a_{ij}^2 \alpha^2 S_{-j,j}(a^{-1}) \mod \Gamma_i.
    \end{equation}
	Given two short $(\sigma, \Gamma)$-elementary transvections $T_{jk}(\xi)$ and
    $T_{km}(\zeta)$ such that $j \neq \pm m$ we get 
    \begin{equation}
    	\label{eq:theorem:spgen:transporter:descr:short:short}
        \begin{aligned}
        	S_{i,-i}(a T_{jk}(\xi) T_{km}(\zeta) a^{-1}) 
        	&\equiv a_{ij}^2 \xi^2 S_{k,-k}(a^{-1}) +
			a_{i,-k}^2 \xi^2 S_{-j,j}(a^{-1}) \\
            &\quad
            + a_{ik}^2 \zeta^2 S_{m,-m}(a^{-1}) + 
            a_{i,-m}^2 \zeta^2 s_{-k,k}(a^{-1}) \\
            &\quad 
            + a_{ij}^2 \xi^2 \zeta^2 S_{m,-m}(a^{-1}) \mod \Gamma_i.
		\end{aligned}
	\end{equation}
    Finally given a short and a long $(\sigma,\Gamma)$-elementary transvections 
    $T_{jk}(\xi)$ and $T_{k,-k}(\alpha)$ we get 
	\begin{equation}
    	\label{eq:theorem:spgen:transporter:descr:short:long}
        \begin{aligned}
        	S_{i,-i}(a T_{jk}(\xi) T_{k,-k}(\zeta) a^{-1}) 
        	&\equiv a_{ij}^2 \xi^2 S_{k,-k}(a^{-1}) +
			a_{i,-k}^2 \xi^2 S_{-j,j}(a^{-1}) \\
            &\quad
            + a_{ik}^2 \alpha + a_{ik}^2 \alpha^2 S_{-k,k}(a^{-1}) \\
            &\quad 
            + a_{ij}^2 \xi^2 \alpha^2 S_{-k,k}(a^{-1}) \mod \Gamma_i.
        \end{aligned}
	\end{equation}
    Comparing \eqref{eq:theorem:spgen:transporter:descr:short:short} and
    \eqref{eq:theorem:spgen:transporter:descr:short} we get the inclusions
	\[
    	a_{ij}^2 \sigma_{jk}^\rectangled{2} \sigma_{km}^\rectangled{2} 
    	S_{k,-k}(a^{-1}) \in \Gamma_i
	\]
    for all $i,j,k,m \in I$ such that $j \neq \pm k, \pm m$ and $k \neq \pm m$.
    Similarly comparing 
    \eqref{eq:theorem:spgen:transporter:descr:short:long} with
    \eqref{eq:theorem:spgen:transporter:descr:long} and
    \eqref{eq:theorem:spgen:transporter:descr:short} we get the inclusions
	\[
    	a_{ij}^2 \sigma_{jm}^\rectangled{2} \Gamma_m^\rectangled{2} 
    	S_{-m,m}(a^{-1}) \in \Gamma_i
	\]
    for all $i,j,m \in I$ such that $j \neq \pm m$.
        
    Now let $j,m \in I$ such that $j \neq \pm m$. As $h(\nu) \geq (4,3)$ 
    either there exists an index $k \sim m$ such that $k \neq \pm j, \pm m$ 
    or $-m \sim m$. In the first case we get 
    \[
		a_{ij}^2 \sigma_{jm}^\rectangled{2} S_{m,-m}(a^{-1}) = 
        a_{ij}^2 \sigma_{jk}^\rectangled{2} R^\rectangled{2} S_{m,-m}(a^{-1}) =             
		a_{ij}^2 \sigma_{jk}^\rectangled{2} \sigma_{km}^\rectangled{2} 
		S_{m,-m}(a^{-1}) \le \Gamma_i
	\]
    for all $i \in I$. In the second case we get similarly
    \[
    	a_{ij}^2 \sigma_{jm}^\rectangled{2} S_{m,-m}(a^{-1}) = 
        a_{ij}^2 \sigma_{j,-m}^\rectangled{2} R^\rectangled{2} 
        S_{m,-m}(a^{-1}) = 
		a_{ij}^2 \sigma_{j,-m}^\rectangled{2} \Gamma_{-m}^\rectangled{2} 
		S_{m,-m}(a^{-1}) \le \Gamma_i
	\]
    for all $i \in I$. To prove the inclusions (T2) for the matrix $a$ it's
    only left to consider the cases when $m = j$ and $m = -j$. Fix an index 
    $k \sim j$ such that $k \neq \pm j$. Observe that 
    \[
    	1 = \left( \sum_{t \in I} a'_{kt} a_{tk} \right)^2 \equiv 
    	\sum_{t \in I} a^{\prime 2}_{kt} a_{tk}^2 \mod 2 R
	\]
    and therefore
    \[
    	\begin{aligned}
        	a_{ij}^2 \sigma_{jm}^\rectangled{2} S_{m,-m}(a^{-1}) & = 
	        a_{ij}^2 \sigma_{jk}^\rectangled{2} \left( \sum_{t \in I} 
    	    a'_{kt} a_{tk} \right)^2 \sigma_{km}^\rectangled{2} 
    	    S_{m,-m}(a^{-1}) 
	        \\
			&\quad 
        	\equiv \sum_t \left( a_{ij}^2 \sigma_{jk}^\rectangled{2} 
	        (a^{\prime 2}_{kt}) \right)
    	    \left( a_{tk}^2 \sigma_{km}^\rectangled{2} S_{m,-m}(a^{-1}) \right) 
        	\\
			&\le
    	    \sum_{t \in I} \sigma_{it}^2 \Gamma_t \le \Gamma_i,
	    \end{aligned}
	\]
    where the congruence is meant modulo $\Gamma_i$.
        
    Thus $a$ satisfies condition (T2). Finally using 
    \eqref{eq:theorem:spgen:transporter:descr:long} and (T2) we get the
    inclusions (T3) for all $i,j \in I$.
%     such that $i \neq \pm j$. The inclusion 
%    for $i = j$ follows from the definition of a form net of ideals. 
%    Finally we need to prove the inclusion 
%    $a_{-j,j}^2 \Gamma_j \le \Gamma_{-j}$. 
%	Fix some $k \sim j$ such that $k \neq \pm j$. Then 
%	$a_{-k,j}^2 \Gamma_j \le \Gamma_i$. Clearly, if $a$ is in
%	$\Transp_{\Sp(2n,R)}(\Ep(\sigma, \Gamma),\Sp(\sigma, \Gamma))$ 
%	then $b = T_{-k,-j}(1) a$ is also contained in the same transporter. 
%	Then 
%	\[
% 		 b_{-k,j}^2 \Gamma_j = (a_{-k,j} + a_{-j,j})^2 \Gamma_j \equiv
%        a_{-k,j}^2 + a_{-j,j}^2 \Gamma_j \le \Gamma_i.
%   \] 
%    Thus also $a_{-j,j}^2 \Gamma_j \le \Gamma_i$. 
    Thus we have proved that
    \[
		\Transp_{\Sp(2n,R)}(\Ep(\sigma, \Gamma),\Sp(\sigma, \Gamma)) \le N.
	\]
        
    Finally, it is easy to see that $\Transp$ is contravariant in the first variable and therefore
    \[
    	\N_{\Sp(2n,R)}(\Sp(\sigma, \Gamma)) \le
        \Transp_{\Sp(2n,R)}(\Ep(\sigma, \Gamma), \Sp(\sigma, \Gamma)).
	\] 
	Hence also
    \[
    	\N_{\Sp(2n,R)}(\Sp(\sigma, \Gamma)) =
        \Transp_{\Sp(2n,R)}(\Ep(\sigma, \Gamma), \Sp(\sigma, \Gamma)) = N.
	\] 
    \end{proof}

\section{Standard setting}
\label{sec:spgen:stsetting}

From this section on we focus on proving Theorem \ref{theorem:spgen:main}. Our proof is a
\textit{loclization} based proof. In order to avoid all the hustle with zero divisors,
injectivity of the localization morphism and so forth we introduce the concept of
a \textit{standard setting}. In the end of this section we present two conceptual examples
of standard settings which motivate the definition.

Let $R$ be a commutative associative unital ring, $R'$ a unital subring of $R$ 
and $S$ a subset of the intersection $R' \cap R^*$, where $R^*$ stands 
for the set of invertible elements of the ring $R$. We call the triple
$(R, R', S)$ \textit{a standard setting} if for any $\xi \in R$ 
there exist an element $x$ in $S$ such that $x \xi \in R'$. 
Clearly, the canonical ring homomorphism $S^{-1} R' \rightarrow R$ is an isomorphism.
Now let $(\sigma', \Gamma')$ be an exact form net of ideals of rank 
$2n$ over $R'$ such that $[\nu]_{R'} \le (\sigma', \Gamma')$. For each 
$i,j \in I$ set
\begin{align*}
    \sigma_{ij} &= \{ \xi \in R \;|\; \exists x \in S \; 
    	x \xi \in \sigma'_{ij} \} \\
    \Gamma_i &= \{ \alpha \in R \;|\; \exists x \in S \; 
    	x^2 \alpha \in \Gamma'_i \}.
\end{align*}
We will call the pair $(\sigma, \Gamma)$ \textit{the $S$-closure of the form net
of ideals  $(\sigma', \Gamma')$ \textup{[}in $R$\textup{]}.} We
will show (Proposition \ref{prop:spgen:localized:nets:are:nets:case:g=e})
that $S$-closures of exact form $\D$-nets of ideals over $R'$ are 
exact form $\D$-nets of ideals over $R$.

Fix a subgroup $H$ of $\Sp(2n, R)$. We call a form net of ideals $(\sigma', \Gamma')$ 
over $R'$ \textit{$S$-associated with the subgroup $H$} if the following two conditions are fulfilled:
\begin{enumerate}
	\item $\Ep(\sigma', \Gamma') \le H$
	\item For any elementary symplectic transvection $T_{sr}(\xi)$ contained in $H$
	    there exists an element $x \in S$ such that
    	$x^{(1 + \delta_{r,-s})} \xi \in (\sigma', \Gamma')_{sr}$.
\end{enumerate}
It is easy to see that a subgroup may have several different $S$-associated nets, but
their $S$-closures in $R$ will coincide.

We will introduce now a family of net-like objects. For an arbitrary $g \in \Sp(2n,R)$ set
\begin{equation*}
%\label{eq:sec:spgen:stsetting:sigma:g}
\begin{aligned}
    \sigma^g_{ij} &= \{ \xi \in R \;|\; \exists x \in S \; \forall \theta \in R' \; {}^g T_{ij}(x \theta \xi) 
    \in H \}, i \neq \pm j \\
    \sigma^g_{ii} &= R \\
    \Gamma^g_{i} &= \{ \alpha \in R \;|\; \exists x \in S \; \forall \theta \in R' \; {}^g T_{i,-i}(x^2 \theta^2 \alpha)
    \in H \} \\
    \sigma^g_{i,-i} &= \sum_{j \neq \pm i} \sigma^g_{ij} \sigma^g_{j,-i} + \left\langle \Gamma^g_{i} \right\rangle_R,
\end{aligned}
\end{equation*}
where the product $\sigma^g_{ij} \sigma^g_{j,-i}$ denotes 
the product of ideals, that is the ideal generated by all products $\xi \zeta$, where
$\xi \in \sigma_{ij}^g$ and $\zeta \in \sigma_{j,-i}^g$. In general there is no guarantee that the
objects $(\sigma^g, \Gamma^g)$, defined in the obvious way from the above data are form nets of ideals. 
We will show that in cases of interest to us
the objects ($\sigma^g, \Gamma^g)$ are form nets of ideals and coincide with the $S$-closure of any
net which is $S$-associated with the subgroup $H$.

For the rest of this section we fix a standard setting $(R, R', S)$, a unitary equivalence relation $\nu$
and a subgroup $H$ of $\Sp(2n, R)$.

\begin{proposition}
    \label{prop:spgen:localized:nets:are:nets:case:g=e}
    Let $(\sigma', \Gamma')$ be an exact major form net of ideals over $R'$ and $(\sigma, \Gamma)$
    the $S$-closure of $(\sigma', \Gamma')$ in $R$. Then $(\sigma, \Gamma)$
    is an exact major form net of ideals over $R$.
    Further, assume that $h(\nu) \geq (4,3)$ and that $(\sigma', \Gamma')$ is $S$-associated
    with the subgroup $H$. Then the form net of ideals $(\sigma, \Gamma)$
    is coordinate-wise equal to $(\sigma^e, \Gamma^e)$.
    \begin{proof}
		Clearly $\sigma_{ij} = R$ whenever $i \sim j$ and
		$\Gamma_i = R$ whenever $i \sim -i$. We will show first that for
		all $i,j \in I$ the sets $\sigma_{ij}$ and $\Gamma_i$ are additive
		subgroups of $R$. Let $\xi, \zeta \in (\sigma, \Gamma)_{ij}$.
		By definition, there exist elements $x, y$ in $S$ such that 
		$x^{(1+\delta_{j,-i})} \xi, y^{(1+\delta_{j,-i})} \zeta \in 
		(\sigma', \Gamma')_{ij}$. As $(\sigma', \Gamma')$ is a form net
		of ideals, it follows that $(xy)^{(1+\delta_{j,-i})} \xi,
		(xy)^{(1+\delta_{j,-i})} \zeta \in (\sigma', \Gamma')_{ij}$
		and thus also $(xy)^{(1+\delta_{j,-i})} (\xi + \zeta) \in 
		(\sigma', \Gamma')_{ij}$. Therefore $\xi + \zeta \in (\sigma,
		\Gamma)_{ij}$. The rest of the properties of $(\sigma, \Gamma)$
		as an exact form net of ideals can be deduced in the same way
		from the corresponding properties of $(\sigma', \Gamma')$.
		
		Assume $h(\nu) \geq (4,3)$. It's obvious that 
		$(\sigma^e, \Gamma^e)_{ij} \le (\sigma, \Gamma)_{ij}$ for all 
		possible indices $i$ and $j$ and thus also that $\sigma^e_{i,-i} \le
		\sigma_{i,-i}$ for all $i \in I$. The reverse inclusions are obtained in the
		following way. Fix some $i \nsim j$ and $\xi \in (\sigma,
		\Gamma)_{ij}$. By definition, there exists an element
		$x \in S$ such that $T_{ij}(x^{(1+\delta_{i,-j})} \xi) \in H$. Assume first,
		$i \neq -j$. Then, as $h(\nu) \geq (4,3)$, there exists another 
		index $k \sim j$ such that $k \neq \pm j, \pm i$. Then $T_{jk}(\theta),
		T_{kj}(1) \in H$ for all $\theta \in R'$ and therefore
		\[
			T_{ij}(x \theta \xi) = 
			[[T_{ij}(x \xi), T_{jk}(\theta)], T_{kj}(1)] \in H.
		\]
		Hence, $\xi \in \sigma^e_{ij}$. If $i = -j$ then there exists
		another index $k \sim i$ such that $k \neq \pm i$. As 
		$(\sigma, \Gamma)$ is an exact form net of ideals, it follows by
		Proposition \ref{prop:symplss:change:indices} that
		$x^2 \xi \in \Gamma'_{k,-k}$. Thus we get
		\[
			T_{i,-i}(-\varepsilon_i \varepsilon_j x^2 \theta^2 \xi) 
			T_{k,-i}(x^2 \theta \xi) = [T_{k,-k}(x^2 \xi), T_{-k,-i}(\theta)]
			\in H.
		\]
		If $k \sim -i$, then $T_{k,-i}(x^2 \theta \xi) \in H$ and therefore
		$T_{i,-i}(-\varepsilon_i \varepsilon_j x^2 \theta^2 \xi) \in H$.
		If $k \nsim -i$, there exists another index $l \sim k$ such that 
		$l \neq \pm k, \pm i$. By relation (R4)
		\[
			T_{k,-i}(- x^2 \theta \xi) = 
				[T_{kl}(1), [T_{lk}(-1), [T_{k,-k}(x^2 \xi), 
				T_{-k,-i}(\theta)]]] \in H.
		\]		
		Therefore 
		$T_{i,-i}(-\varepsilon_i \varepsilon_j x^2 \theta^2 \xi) \in H$
		and $\xi \in \Gamma_{i}^e$. Summing up, $(\sigma, \Gamma)_{ij} \le
		(\sigma, \Gamma)^e_{ij}$ for all $i,j \in I$. As $(\sigma, \Gamma)$ is exact, it follows that
		$\sigma_{i,-i} \le \sigma_{i,-i}^e$. This completes the
		proof.
    \end{proof}   
\end{proposition}

The last proposition allows us to consider the elementary form net subgroup 
$\Ep(\sigma, \Gamma)$ of $\Sp(2n, R)$. The following proposition establishes
certain properties of the objects $(\sigma^g, \Gamma^g)$
which follow directly from their definition and the Steinberg relations.
This shows that $(\sigma^g, \Gamma^g)$ is ``almost a form net of ideals''.

\begin{proposition}
	\label{prop:spgen:sigma:g:is:almost:a:net}
	Assume $h(\nu) \geq (4,3)$. Let $g$ be an element of $\Ep(\sigma, \Gamma)$.
	If $[\nu]_{R} \le (\sigma^g, \Gamma^g)$ coordinate-wise then 
	the following inclusions hold:
	\begin{enumerate}
		\item $\sigma^g_{ij} \sigma^g_{jk} \le \sigma^g_{ik}$ for all 
			$i \neq \pm j$, $j \neq \pm k$ 
		\item $\Gamma^g_i \sigma^g_{-i,k} \le \sigma^g_{ik}$ and 
			$\sigma^g_{i,-k} \Gamma^g_{-k} \le \sigma^g_{ik}$ for all $i, k \in I$
		\item $(\sigma_{ij}^g)^\rectangled{2} \Gamma^g_j \le \Gamma_i^g$ for all $i \neq \pm j$
		\item $2 \sigma^g_{ij} \sigma^g_{j,-i} \le \Gamma_i^g$ for all 
		$i \neq \pm j$,
	\end{enumerate}	
    where products are Minkowski products of sets, i.e. sets of products of elements of the factors.
	\begin{proof}
		1. The first property follows directly from the Steinberg relation
		(R4). Indeed, pick any $\xi \in \sigma_{ij}$ and any $\zeta \in
		\sigma_{jk}$ such that $i \neq \pm j, \pm k$ and $j \neq \pm k$.
		Then there exist elements $x_\xi, x_\zeta \in S$ such that
		${}^g T_{ij}(x_\xi \xi), {}^g T_{jk}(x_\zeta \theta \zeta) \in H$ 
		for all $\theta \in R'$. By relation (R4) we get
		\[
			{}^g T_{ik} (x_\xi x_\zeta \theta \xi \zeta) =
			[{}^g T_{ij}(x_\xi \xi), {}^g T_{jk}(x_\zeta \theta \zeta)]
			\in H
		\]
		for all $\theta \in R'$. Therefore $\xi \zeta \in \sigma^g_{ik}$.
		The corresponding inclusions for the cases when $i = \pm k$ trivially
		follow from the definition of $(\sigma^g, \Gamma^g)$.

		2. The second inclusion is trivial when $i = \pm k$ for the same
		reason as above. Assume $i \neq \pm k$. We will prove the inclusion
		$\Gamma^g_i \sigma^g_{-i,k} \le \sigma^g_{ik}$. The other one can be
		treated similarly. 	Pick any $\alpha \in \Gamma_i$ and 
		$\xi \in \sigma_{-i,k}$. Then there exist elements $x_\alpha, x_\xi \in S$
		such that for any $\theta \in R'$ we have
		\[
			{}^g T_{i,-i}(x_\alpha^2 \alpha),
			{}^g T_{-i,k}(x_\xi \theta \xi) \in H.
		\] 
		By relation (R6) it follows that
		\begin{equation}
			\label{eq:prop:spgen:sigma:g:is:almost:a:net:1}		
			{}^g T_{ik}(x_\alpha^2 x_\xi \theta \alpha \xi)
			{}^g T_{-k,k}(x_\alpha^2 x_\xi^2 \theta^2 \alpha \xi^2) = 
			[{}^g T_{i,-i}(x_\alpha^2 \alpha),
			{}^g T_{-i,k}(x_\xi \theta \xi)] \in H.
		\end{equation}
		If $k \sim -k$ then by the definition of $\Gamma^g_{-k}$
		we get ${}^g T_{-k,k}(x_\alpha^2 x_\xi^2 \theta^2 \alpha \xi^2) \in H$.
		Thus we get ${}^g T_{ik}(x_\alpha^2 x_\xi \theta \alpha \xi) \in H$ and
		$\alpha \xi \in \sigma^g_{ik}$. If $k \nsim -k$ then,
		as $h(\nu) \geq (4,3)$, there exists another index $l \sim i$
		such that $l \neq \pm i, \pm k$. Then there exist elements $x_1,x_2 \in S$
		such that ${}^g T_{li}(x_1), {}^g T_{il}(x_2) \in H$.		
		By the Steinberg relations
		(R3) and (R4) together with \eqref{eq:prop:spgen:sigma:g:is:almost:a:net:1}
		we get
		\[
			{}^g T_{ik} (x_1 x_2 x_\alpha^2 x_\xi \theta \alpha \xi) =
			[{}^g T_{il}(x_2), [{}^g T_{li}(x_1),
			{}^g T_{ik}(x_\alpha^2 x_\xi \theta \alpha \xi)
			{}^g T_{-k,k}(x_\alpha^2 x_\xi^2 \theta^2 \alpha \xi^2)]] \in H
		\]
		for all $\theta \in R'$. It follows that $\alpha \xi \in \sigma^g_{ik}$.
		
		3. The next series of inclusions is established similarly. Fix some indices $i \neq \pm j$,
		an element $\xi \in \sigma^g_{ij}$ and an element $\alpha \in \Gamma^g_j$. Then there exist elements 
        $x_\xi, x_\alpha \in S$ such that 
        \begin{equation}
			\label{eq:prop:spgen:sigma:g:is:almost:a:net:3}		
            {}^g T_{i,-j}(x_\xi x_\alpha^2 \theta \xi \alpha) 
            \cdot T_{i,-i}(x_\xi^2 x_\alpha^2 \theta^2 \xi^2 \alpha) = 
            [{}^g T_{ij}(x_\xi \theta \xi), {}^g T_{j,-j}(x_\alpha^2 \alpha)] \in H
        \end{equation}
        for all $\theta \in R'$. By assertion (2) of the current lemma, the first
        term of the left-hand side of \eqref{eq:prop:spgen:sigma:g:is:almost:a:net:3}
        is contained in $H$ whenever $\theta$ is a multiple of 
        some $x_0 \in S$. Therefore the second term of \eqref{eq:prop:spgen:sigma:g:is:almost:a:net:3}
        is also contained in $H$ for the same values of parameter $\theta$. This shows that $\xi^2 \alpha \in 
        \Gamma_i^g$. 
        
        4. Finally, fix an index $i \neq \pm j$, an element $\xi \in \sigma^g_{ij}$ and
        an element $\zeta \in \sigma^g_{j,-i}$. Then there exist $x_\xi, x_\zeta \in S$
        such that ${}^g T_{ij}(x_\xi \theta \xi), 
        {}^g T_{j,-i}(x_\zeta \theta \zeta)
        \in H$ for all $\theta \in R'$; in particular 
        ${}^g T_{ij}(x_\xi x_\zeta \theta \xi), 
        {}^g T_{j,-i}(x_\zeta x_\xi \theta
        \zeta) \in H$ for all $\theta \in R'$. 
        By the Steinberg relation (R5) it follows that
        \[
        	{}^g T_{i,-i}(2 x_\xi^2 x_\zeta^2 \theta^2 \xi \zeta) \in H
        \]
        for all $\theta \in R'$. Hence, $\xi \zeta \in \Gamma_i$. 
	\end{proof}
\end{proposition}

\begin{lemma}
    \label{lemma:spgen:delocalization}
    Assume $h(\nu) \geq (4,3)$. Let $(\sigma', \Gamma')$ be an exact major form net 
    of ideals over $R'$, which is $S$-associated
    with $H$. Let $(\sigma, \Gamma)$ denote the $S$-closure of $(\sigma', \Gamma')$ in $R$. 
    Then for every $g \in \Ep(\sigma, \Gamma)$ the coordinate-wise equality 
    \begin{equation}
        \label{eq:lemma:spgen:delocalization:statement:1}
        (\sigma, \Gamma) = (\sigma^g, \Gamma^g)
    \end{equation}            
    holds.
    In particular, each such $(\sigma^g, \Gamma^g)$ is an exact major form net of ideals over $R$.
    \begin{proof}
        We will prove this lemma by induction on the word length 
        $L(g)$ of $g$ in terms of the
        generators of $\Ep(\sigma, \Gamma)$. Proposition
        \ref{prop:spgen:localized:nets:are:nets:case:g=e}
        serves as a base of induction, namely it shows that when $L(g) = 0$
        and $g = e$ we have the equality $(\sigma, \Gamma) = (\sigma^e,
        \Gamma^e)$.

        Before proving the induction step, we will prove a slightly stronger statement. Namely,
        assume $g \in \Ep(\sigma, \Gamma)$ such that $(\sigma, \Gamma) \le (\sigma^g, \Gamma^g)$.
        Fix an element $T_{pq}(\zeta) \in \Ep(\sigma, \Gamma)$. We will show that 
        $(\sigma^g, \Gamma^g) \le (\sigma^{g T_{pq}(\zeta)}, \Gamma^{g T_{pq}(\zeta)})$.
        Note that, as $(\sigma, \Gamma) \le (\sigma^g, \Gamma^g)$, it follows that 
        $\zeta \in (\sigma^g, \Gamma^g)_{pq}$.
        Fix any $\xi \in (\sigma^g, \Gamma^g)_{sr}$ for some indices $s \neq r$. Then there exists
        an element $x_\xi \in S$ such that for every $\theta \in R'$ the inclusion
        ${}^g T_{sr}(x_\xi^\kappa \theta^\kappa \xi) \in H$ holds, where $\kappa = 1 + \delta_{s,-r}$.
        For any $x \in S$ we have the equality
        \begin{equation}
            \label{eq:lemma:spgen:delocalization:statement:commutator}
            {}^{g T_{pq}(\zeta)} T_{sr}(x^\kappa \theta^\kappa \xi) =
            {}^g [T_{pq}(\zeta), T_{sr}(x^\kappa \theta^\kappa \xi)] \cdot 
            {}^g T_{sr}(x^\kappa \theta^\kappa \xi).
        \end{equation}
        Below we will construct an element $x_0$ such that after the
        substitution $x = x_0$ the right-hand side of \eqref{eq:lemma:spgen:delocalization:statement:commutator}
        is contained in $H$ for all $\theta \in R'$. It will follow that
        $\xi \in (\sigma^{g T_{pq}(\zeta)}, \Gamma^{g T_{pq}(\zeta)})_{sr}$.
        
        Clearly the second term of the right-hand side of 
        \eqref{eq:lemma:spgen:delocalization:statement:commutator} is contained in 
        $H$ whenever $x$ is a multiple of $x_\xi$. The first term, which we 
        will denote by $h = h(\theta)$, requires a more detailed investigation.
        First, assume that the transvections $T_{sr}(*)$ and $T_{pq}(*)$
        commute. In this case, $h = e$ and thus we can put $x_0 = x_\xi$.
        Assume that $h \neq e$. The following six alternatives 
        exhaust all possibilities:
        \begin{enumerate}[(1)]
            \item $s \neq \pm r, p \neq \pm q$ and one of the following
            	holds
                \begin{enumerate}[(i)]
                    \item $s = q, r \neq \pm p$
                    \item $r = -q, s \neq \pm p$
                    \item $s = -p, r \neq \pm q$ 
                    \item $-r = -p, s \neq \pm q$.
                \end{enumerate}
                Then $h$ is a single short transvection. We will prove
                only the case (i). The other ones can be treated
                similarly. By the Steinberg relation (R4), 
                $h = {}^g T_{pr}(x \theta \zeta \xi)$. Recall that 
                \[
                    [\nu]_R \le (\sigma, \Gamma) \le (\sigma^g, \Gamma^g).
                \]
                By Proposition \ref{prop:spgen:sigma:g:is:almost:a:net} we get
                $\zeta \xi \in \sigma^g_{pr}$. Therefore there exists an element $x_\zeta \in S$ such that
                $h$ is contained in $H$ whenever $x$ is a multiple of $x_\zeta$. Put $x_0 = x_\xi x_\zeta$.
			\item $p \neq \pm q, s \neq \pm r$ and one of the following
                holds:
                \begin{enumerate}[(i)]
                    \item $s = q, r = -p$
                    \item $s = -p, r = q$
                    \item $s = p, r = -q$
                    \item $s = -q, r = p$.
                \end{enumerate}
                In this case we can compute $h$ using the Steinberg relation
                (R5). Again, we will prove only the case (i). As
				$\zeta \in \sigma^g_{pq}$, there exists an element $x_\zeta \in S$
				such that ${}^g T_{pq}(x_\zeta \zeta) \in H$. Then $h \in H$ for all 
				$\theta \in R'$ and $x_0 = x_\xi x_\zeta$. Indeed,
				\[
					h = [{}^g T_{pq}(\zeta), 
						{}^g T_{q,-p}(x_\zeta x_\xi \theta \xi)]
					  = {}^g T_{p,-p}(2 \zeta x_\zeta x_\xi \theta \xi)
					  = [{}^g T_{pq}(\zeta x_\zeta), 
						{}^g T_{q,-p}(x_\xi \theta \xi)] \in H
				\]
				for every $\theta \in R'$ due to the choice of $x_\xi$ and
				$x_\zeta$.
			\item $q = -p, s \neq \pm r$ and either $s = -p$ or $r = p$. 
				In both cases $h$ is a product of a long and a short 
				symplectic
				elementary transvection. We will consider only the 
				first option. By relation (EU6),
				\begin{equation}
         		   \label{eq:lemma:spgen:delocalization:h:long:short}
					h = [{}^g T_{p,-p}(\zeta), {}^g T_{-p,r}(x \theta \xi)]
					  = {}^g T_{pr}(x \theta \xi \zeta)
					    {}^g T_{-r,r}(\pm x^2 \theta^2 \xi^2 \zeta).
				\end{equation}
				By Proposition \ref{prop:spgen:sigma:g:is:almost:a:net}
				it follows that
				$\xi \zeta \in \Gamma_p^g \sigma_{-p,r} \le \sigma_{pr}^g$
				and $\xi^2 \zeta \in (\sigma^g_{-p,r})^\rectangled{2} \Gamma_p^g
				\le \Gamma_{-r}^g$. Therefore there exist elements $x_{\xi\zeta},
				x_{\xi^2\zeta} \in S$ such that the first term of the right-hand side
				of \eqref{eq:lemma:spgen:delocalization:h:long:short} 
				belongs to $H$ whenever $x$ is a multiple
				of $x_{\xi\zeta}$ and the second term whenever $x$ is a 
				multiple of $x_{\xi^2\zeta}$. Put
				$x_0 = x_\xi x_{\xi\zeta} x_{\xi^2\zeta}$.
				
			\item $p \neq \pm q, r = -s$ and either $s = q$ or $s = -p$. Then
				 $h$ is a product a long and a short transvection. 
				 We prove only the first option, $s = q$. By the Steinberg relations (R1) and (R6) we have
				\begin{equation}
         		   \label{eq:lemma:spgen:delocalization:h:short:long}
					h = [{}^g T_{pq}(\zeta), {}^g T_{q,-q}(x^2 \theta^2 \xi)]
					  = {}^g T_{p,-q}(\pm x^2 \theta^2 \xi \zeta)
					    {}^g T_{p,-p}(\pm x^2 \theta^2 \zeta^2 \xi).
				\end{equation}
				As before, 	by Proposition
				 \ref{prop:spgen:sigma:g:is:almost:a:net} we have 
				$\xi \zeta \in \sigma_{pq}^g \Gamma_q^g \le \sigma_{p,-q}^g$
				and $\zeta^2 \xi \in (\sigma^g_{pq})^\rectangled{2} \Gamma_q^g
				\le \Gamma_{p}^g$. Therefore there exist elements $x_{\xi\zeta},
				x_{\zeta^2 \xi} \in S$ such that the right-hand side of
				\eqref{eq:lemma:spgen:delocalization:h:short:long} is contained in $H$
				whenever $x$ is a multiple of $x_{\xi\zeta} x_{\zeta^2 \xi}$.
				Put $x_0 = x_\xi x_{\xi\zeta} x_{\zeta^2 \xi}$.
			\item Either $s = q, r = p$ or $s= -p, r = -q$. 
				\Wlg we can assume the former.
				In this case, we can't apply any of the Steinberg relations
				directly, but we can first decompose $T_{sr}(*)$ as a
				product of transvections for which we know the commutators
				with $T_{pq}(*)$. As $h(\nu) \geq (4,3)$
				there exists either another index $h \sim p$ such that 
				$h \neq \pm p,\pm q$, or $p \sim -p$. In the first case,
				\begin{equation}
		    		\label{eq:lemma:spgen:delocalization:h:short:opposite:s}
					\begin{aligned}
						{}^{g T_{pq}(\zeta)} T_{qp}(x \theta \xi) &=
						{}^{g T_{pq}(\zeta)} [T_{qh}(y \theta \xi), T_{hp}(z)]\\
						&= 
				  		{}^g [[T_{pq}(\zeta), T_{qh}(y \theta \xi)] 
					  	T_{qh}(y \theta \xi), 
					  	[T_{pq}(\zeta), T_{hp}(z)] T_{hp}(z)] \\
					  	&=
					  	[{}^g T_{ph}(y \theta \zeta \xi) \cdot {}^g 
					  	T_{qh}(y \theta \xi),
					  	{}^g T_{hq}(- z \zeta) \cdot {}^g T_{hp}(z)]
					\end{aligned}
				\end{equation}	
				whenever $x = yz$. Observe that
				$\xi \in \sigma^g_{qp} \le \sigma^g_{qp} R = \sigma^g_{qp}
				\sigma^g_{ph} \le \sigma^g_{qh}$, $\xi \zeta \in
				\sigma^g_{pq} \sigma^g_{qh} \le \sigma^g_{ph}$,
				$\zeta \in \sigma^g_{pq} \le R \sigma^g_{pq} = 
				\sigma^g_{hp} \sigma^g_{pq} \le \sigma^g_{hq}$
				and $1 \in \sigma^g_{hp}$. 	
				Thus we can choose $y$ and $z$ in $S$ such that all 
				four terms of the right-hand side of
				\eqref{eq:lemma:spgen:delocalization:h:short:opposite:s}
				are contained in $H$ for all $\theta \in R'$. Then
				we can put $x_0 = yz x_\xi$.
				
				If the equivalence class of $p$ equals $\{\pm p, \pm q \}$ then
				we can decompose $T_{sr}(*)$ in a different way, using
				long transvections. Namely,
				\begin{equation}
		    		\label{eq:lemma:spgen:delocalization:h:short:opposite:l}
					\begin{aligned}
						{}^{g T_{pq}(\zeta)} T_{qp}(x \theta \xi) &=
						{}^{g T_{pq}(\zeta)} 
						\left( [T_{q,-q}(y^2), T_{-q,p}(z \theta \xi)]
						T_{-p,p}(\pm z^2 \theta^2 \xi^2 y)
						\right) = \\
						&= {}^g \left( [T_{p,-q}(\zeta y^2) 
						T_{p,-p}(\pm \zeta^2 y^2) T_{q,-q}(y^2),\right. \\
						&\quad \left.
						T_{-q,q}(-2 \zeta z \theta \xi) T_{-q,p}(z \theta \xi)]
						T_{q,-p}(\pm \zeta z^2 \theta^2 \xi^2 y^2)
						\right. 
						\\
						&\quad \left.
						T_{q,-q}(\pm \zeta^2 z^2 \theta^2 \xi^2 y^2)
						T_{-p,p}(\pm z^2 \theta^2 \xi^2 y^2)
						\right),
					\end{aligned}
				\end{equation}	
				whenever $x = y^2 z$. Using the previous cases, we can
				choose $y$ and $z$ such that the right-hand side
				of \eqref{eq:lemma:spgen:delocalization:h:short:opposite:l}
				is contained in $H$ for all $\theta \in R'$. Put $x_0 = y^2 z$.
			\item $q = s = -p, r = p$. Then there exists
				an index $h \sim p$ such that $h \neq \pm p$. Then
				\begin{equation}
		    		\label{eq:lemma:spgen:delocalization:h:long:opposite}
					\begin{aligned}
						{}^{g T_{p,-p}(\zeta)} T_{-p,p}(x^2 \theta^2 \xi) &=
						{}^{g T_{p,-p}(\zeta)} 
						\left( [T_{-h,h}(y^2 \theta^2 \xi), 
						T_{h,p}(\pm  z)]
						\times \right. \\
						&\quad \left.
						T_{-h,p}(\pm  y^2 z \theta^2 \xi)
						\right) \\
						&=
						{}^g 
						T_{-h,-p}(y^2 z \theta^2 \zeta \xi) \cdot
						{}^g T_{-h,h}(y^4 z^2 \theta^2 \xi^2 \zeta) \cdot \\
						&\quad 
						{}^g T_{-h,p}(\pm y^2 z \theta^2 \xi),
					\end{aligned}
				\end{equation}
				whenever $x = y^2 z$. Observe that 
				\begin{align*}
					\xi &\in \Gamma^g_{-p} \le R \Gamma^g_{-p} = \sigma^g_{-h,-p}
						\Gamma^g_{-p} \le \sigma^g_{-h,p} \\
				    \xi \zeta &\in \sigma^g_{-h,p} \Gamma_p^g \le 
				    	\sigma^g_{-h,-p} \\
				    \xi^2 \zeta &\in (\sigma^g_{-h,p})^\rectangled{2} \; \Gamma^g_p
					    \le \Gamma_{-h}^g.
				\end{align*}
				Hence we can choose elements $y$ and $z$ in $S$ such that every term
				of the right-hand side of 
				\eqref{eq:lemma:spgen:delocalization:h:long:opposite} is contained
				in $H$ for all $\theta \in R'$. Put $x_0 = y^2 z$.
        \end{enumerate}
        
        The alternatives above are exhaustive.
		Therefore $(\sigma^g, \Gamma^g)_{sr} \le (\sigma^{g T_{pq}(\zeta)},
		\Gamma^{g T_{pq}(\zeta)})_{sr}$ for all $s \neq r \in I$. The
		inclusions $\sigma_{ii}^{g T_{pq}(\zeta)} \le \sigma^g_{ii}$ and
		$\sigma_{i,-i}^{g T_{pq}(\zeta)} \le \sigma^g_{i,-i}$ follow
		easily from the definition of $(\sigma^g, \Gamma^g)$. Therefore we have proved that
		$(\sigma^g, \Gamma^g) \le (\sigma^{g T_{pq}(\zeta)},
		\Gamma^{g T_{pq}(\zeta)})$ coordinate-wise, whenever 
		$(\sigma, \Gamma) \le (\sigma^g, \Gamma^g)$.

		The induction step looks as follows.
        Assume that for all elements $g \in \Ep(\sigma, \Gamma)$ such that 
        $L(g) \leq L_0$, the equality 
        \eqref{eq:lemma:spgen:delocalization:statement:1} 
        holds. Let $T_{pq}(\zeta)$ be an elementary transvection 
        in $\Ep(\sigma, \Gamma)$ such that $L(g \cdot T_{pq}(\zeta)) = L_0 + 1$. 
		Then, as we have proved above, 
		$(\sigma^g, \Gamma^g) \le 
		(\sigma^{g T_{pq}(\zeta)}, \Gamma^{g T_{pq}(\zeta)})$, in particular
		$(\sigma, \Gamma) \le (\sigma^{g T_{pq}(\zeta)}, 
		\Gamma^{g T_{pq}(\zeta)})$. 
		For the same reason
		\[
			(\sigma^{g T_{pq}(\zeta)}, \Gamma^{g T_{pq}(\zeta)}) \le
			(\sigma^{g T_{pq}(\zeta) T_{pq}(-\zeta)}, \Gamma^{g T_{pq}(\zeta) 
			T_{pq}(-\zeta)}) = (\sigma^g, \Gamma^g).
		\]
		Summing up, by induction we get the required equality
		\eqref{eq:lemma:spgen:delocalization:statement:1} for all 
		$g \in \Ep(\sigma, \Gamma)$.
    \end{proof}
\end{lemma}

We will also use the lemma above in the form of the following obvious corollary. It represents the concept
of a common denominator for a finite family of fractions.
\begin{corollary}
    \label{cor:spgen:common:denominator}
    Assume $h(\nu) \geq (4,3)$. Let $(\sigma', \Gamma')$ be an exact major form net of ideals
    which is $S$-associated with the subgroup $H \ge \Ep(\nu, R)$ and let $(\sigma, \Gamma)$ be the $S$-closure
    of $(\sigma', \Gamma')$ in $R$. Then for any finite family
    $\{ T_{s_i,r_i}(\xi_i) \}_{i \in L}$ of $(\sigma,\Gamma)$-elementary
    transvections and any finite family $\{ g_i \}_{i \in K}$
    of elements of $\Ep(\sigma, \Gamma)$ there exists an element $x \in S$ such that
    \[
        {}^{g_i} T_{s_j,r_j}((x \theta)^{(1 + \delta_{s_j, -r_j})} \xi_j) \in H
    \] for all $i \in K, j \in J$ and $\theta \in R'$.
\end{corollary}

\paragraph{Examples of standard settings}\ \\
Before we continue with studying the properties of $S$-associated form nets of ideals and their closures,
let's consider two main examples of a standard setting. In fact, exactly the opportunity to treat these
two cases uniformly was the main motivation to consider this concept at the first place.

1. \textbf{Trivial standard setting}. Let $\tilde{R}$ be a commutative ring. Let
$R = R' = \tilde{R}$ and $S = \{1\}$. Clearly, $(R,R',S)$ is a standard setting. In this
case the concepts of the form net of ideals, associated with a subgroup $H$, 
a form net of ideals $S$-associated with $H$ and the closure of such coincide.

2. \textbf{Local standard setting}. Let $R$ be a commutative ring
and $\m$ a maximal ideal of $R$. Let $S$ denote the compliment $R \setminus \m$,
$R_\m$ the localization $S^{-1} R$ and $F_\m$ the localization morphism
$R \rightarrow R_\m$. Finally, let $R'_\m = F_\m(R)$ and $S_\m = F_\m(S)$. 
Then $(R_\m, R'_\m,S_\m)$ is a standard setting. Given a subgroup $H$ of $\Sp(2n, R)$ 
containing $\Ep(\nu, R)$ with $h(\nu) \geq (4,5)$ and the form net of ideals
$(\sigma, \Gamma)$ associated with $H$ we can consider the coordinate-wise
image $(\sigma'_\m, \Gamma'_\m)$ of $(\sigma, \Gamma)$, i.e.
\[ \sigma'_{\m,ij} = F_\m(\sigma_{ij}) \text{ and } \Gamma'_{\m,i} = F_\m(\Gamma_i) \text{ for any } i, j \in I.\]
It's easy to see that $(\sigma'_\m, \Gamma'_\m)$ is an exact form net of ideals over $R'_\m$.
We will show in Section \ref{sec:spgen:localization} that if $R$ is Noetherian 
then $(\sigma'_\m, \Gamma'_\m)$
is $S_\m$-associated with $H$.

\section{Extraction of transvections}
\label{sec:spgen:extract}

In this section we perform the extraction of transvections first using matrices in small
parabolic subgroups and then using long and short root elements. The methods of this section
are rather standard and mostly mimic those of \cite{VavSymplComm1993}, \cite{VavSympl},
\cite{VavOrthogSMZ88} and \cite{VavOrthog}. Nontheless, we can't simply refer to the similar
results of the cited papers. For our purpuses we have to keep track of denominators
while working in a localization of the ground ring. For this sake it is convenient to use the concept 
of a standard setting.

Throughout this section we fix a standard setting $(R, R', S)$, a unitary equivalence 
relation $\nu$, a subgroup $H$ of $\Sp(2n, R)$ and an exact major form net of ideals $(\sigma', \Gamma')$
which is $S$-associated with $H$. Let $(\sigma, \Gamma)$ denote the $S$-closure of $(\sigma', \Gamma')$ in $R$.

\paragraph{Extraction of transvections in parabolic subgroups}

\begin{lemma}
    \label{lemma:spgen:delta:outside:one:row}
    Let $a$ be a matrix in $\Sp(2n,R)$
    such that for some index $p \in I$ the following conditions hold:
    \begin{enumerate}
        \item $a_{pp} = a_{-p,-p} = 1$
        \item $a_{ij} = \delta_{ij}$ whenever $i \neq -p$ and $j \neq p$.
    \end{enumerate}
    Then
    \begin{equation}
        \label{eq:lemma:spgen:delta:outside:one:row:decomp}
        a = \left( \prod_{1 \le j \neq \pm p \le n} T_{-p,j}(a_{-p,j}) T_{-p,-j}(a_{-p,-j})
        \right) T_{-p,p}(S_{-p,p}(a)).
    \end{equation}
    Further, suppose $h(\nu) \geq (4,3)$ and there exists an element $g \in \Ep(\sigma,\Gamma)$ such that ${}^g a \in H$. 
    Then $a \in \Ep(\sigma, \Gamma)$.
    \begin{proof}
        The decomposition \eqref{eq:lemma:spgen:delta:outside:one:row:decomp} can
        be checked by a straightforward calculation. Suppose $j \neq \pm p$.
        If $j \sim -p$ then the inclusion $a_{-p,j} \in \sigma_{-p,j}$
        is trivial as $\sigma$ is major. From now on, assume 
        $j \nsim -p$. As $h(\nu) \geq (4,3)$, we can choose an index $k \sim j$ such that
        $k \neq \pm j, \pm p$.        
        Choose using Corollary \ref{cor:spgen:common:denominator} an element $x_1 \in S$
        such that ${}^g T_{jk}(x_1) \in H$. Then
        \[
            X = {}^g ( T_{-p,k}(a_{-p,j} x_1) T_{-p,-j}(\pm a_{-p,-k} x_1) ) = {}^g [a, T_{jk}(x_1)] \in H.
        \]
        If the equivalence class of $j$ is non-self-conjugate then either $j \sim p$ or 
        there exists another
        index $h \sim j$ such that $h \neq \pm k, \pm j, \pm p$. In the former case,
        $\pm a_{-p,-k} x_1 \in \sigma^g_{-p,-j} = R$ and by Corollary \ref{cor:spgen:common:denominator}
        the element $x_1$ can be chosen such that ${}^g T_{-p,-j}(\pm a_{-p,-k} x_1) ) \in H$
        and thus also ${}^g T_{-p,k}(a_{-p,j} x_1) \in H$. It follows that 
        $a_{-p,j} \in \sigma^g_{-p,k} = \sigma^g_{-p,j}$. In the latter case
        choose using Corollary \ref{cor:spgen:common:denominator} an element $x_2 \in S$ such that 
        ${}^g T_{kh}(x_2), {}^g T_{hk}(x_2 \theta) \in H$ for any $\theta \in R'$.
        Then for the same $\theta$ we get
        \[
            T_{-p,j}(x_1 x_2^2 \theta a_{-p,j}) = [[X, T_{kh}(x_2)], T_{hk}(x_2 \theta)] \in H.
        \]
        If the equivalent class of $j$ is self-conjugate, it contains at least the elements
        $\pm j, \pm k$. Pick using Corollary \ref{cor:spgen:common:denominator} 
        an element $x_2 \in S$ such that
        ${}^g T_{k,-j}(x_2), {}^g T_{-j,k}(x_2), {}^g T_{kj}(x_2 \theta) \in H$
        for all $\theta \in R'$. Then      
        \[     
            {}^g T_{-p,j}(x_1 x_2^3 \theta a_{-p,j}) = 
                [[[X, {}^g T_{k,-j}(x_2)], {}^g T_{-j,k}(x_2)], {}^g T_{kj}(x_2 \theta)] \in H.
        \]
        Therefore $a_{-p,j} \in \sigma^g_{-p,j}$ and by Lemma
        \ref{lemma:spgen:delocalization} $a_{-p,j} \in \sigma_{-p,j}$ for all $j \neq \pm p$.

        In order to prove that $a \in \EU(\sigma, \Gamma)$ it only remains to show that 
        $T_{-p,p}(S_{-p,p}(a)) \in \Ep(\sigma, \Gamma)$. If $-p \sim p$ then $\Gamma_{-p} = R$ and
		the inclusion $T_{-p,p}(S_{-p,p}(a)) \in \Ep(\sigma, \Gamma)$ is trivial. Assume $p \nsim -p$. Set 
		\begin{align*}
			g_1 &= g \prod_{j>0 , j \neq \pm p} T_{-p,j}(a_{-p,j}) T_{-p,-j}(a_{-p,-j}).
		\end{align*}
		Then $g_1 T_{-p,p}(S_{-p,p}(a)) g^{-1} = {}^g a \in H$ and $g_1, g^{-1} \in 
		\EU(\sigma, \Gamma)$. 
		As $p \nsim -p$, we can choose two more indices
		$q$ and $t$ such that $(p,q,t)$ is an $A$-type base triple.
		Pick an element $y_1 \in S$ such that ${}^g T_{pq}(y \theta), {}^{g_1} T_{pq}(y \theta)
		\in H$ for all $\theta \in R'$
		whenever $y$ is a multiple of $y_1$. By the Steinberg relation (R6) we have
		\begin{equation}
   	    \label{eq:lemma:spgen:delta:outside:one:row:S}
		\begin{aligned}
			{}^{g_1} T_{-p,q}(y \theta S_{-p,p}(a)) \cdot
			{}^{g_1} T_{-q,q}(- \varepsilon_p \varepsilon_q y^2 \theta^2
				S_{-p,p}(a)) \cdot
			{}^{g_1} T_{pq}(y \theta)
			\\ 
			= {}^{g_1} [T_{-p,p}(S_{-p,p}(a)), T_{pq}(y \theta)] \cdot
			{}^{g_1} T_{pq}(y \theta) \\
			= \left (g_1 T_{-p,p}(S_{-p,p}(a)) g^{-1} \right) 
			\left( g T_{pq}(y \theta) g^{-1} \right)
			\left (g T_{-p,p}(-S_{-p,p}(a)) g_1^{-1} \right).
		\end{aligned}
		\end{equation}
		The right-hand side of \eqref{eq:lemma:spgen:delta:outside:one:row:S} as well as the third 
		term of the left-hand side of \eqref{eq:lemma:spgen:delta:outside:one:row:S}
		is contained in $H$ whenever $y$ is a multiple of $y_1$ in $S$. Therefore
		\begin{equation}
		\label{eq:lemma:spgen:delta:outside:one:row:S:two}
			{}^{g_1} T_{-p,q}(y \theta S_{-p,p}(a)) \cdot
			{}^{g_1} T_{-q,q}(- \varepsilon_p \varepsilon_q y^2 \theta^2
				S_{-p,p}(a)) \in H
		\end{equation}
		for all $\theta \in S$ whenever $y$ is a multiple of $y_1$. Pick 
		$y_2 \in S$ such that ${}^{g_1} T_{-p,-t}(y_2)$, ${}^{g_1} T_{-t,-p}(y_2) \in H$.
		We get
		\begin{align*}
			{}^{g_1} T_{-p,q}(y y_2^2 \theta S_{-p,p}(a)) = 
			[{}^{g_1} T_{-p,-t}(y_2), [{}^{g_1} T_{-t,-p}(y_2),
			{}^{g_1} T_{-p,q}(y \theta S_{-p,p}(a)) \cdot \\
			{}^{g_1} T_{-q,q}(- \varepsilon_p \varepsilon_q
			y^2 \theta^2 S_{-p,p}(a))]]
		\end{align*}
		and thus by the choice of $y_2$ together with \eqref{eq:lemma:spgen:delta:outside:one:row:S:two}
		we get that ${}^{g_1} T_{-p,q}(y \theta S_{-p,p}(a)) \in H$ for all 
		$\theta \in R'$ whenever $y$ is a multiple of $y_1 y_2^2$. Combining this result
		again with \eqref{eq:lemma:spgen:delta:outside:one:row:S:two} we get that 
		${}^{g_1} T_{-q,q}(- \varepsilon_p \varepsilon_q y^2 \theta^2
				S_{-p,p}(a)) \in H$ for all
		$\theta \in S$ whenever $y$ is a multiple of $y_1 y_2^2$.
		Thus, 
		$S_{-p,p}(a) \in \Gamma_{-q}$ and by Proposition \ref{prop:symplss:change:indices} 
		$S_{-p,p}(a) \in \Gamma_{-p}$. This completes the proof.
    \end{proof}
\end{lemma}

\begin{lemma}
    \label{lemma:spgen:delta:inside:one:column}
    Assume $h(\nu) \geq (4,3)$. Let $(p,q)$ be an $A$-type base pair and let $a$ be an element of $\Sp(2n,R)$ such that
    $a_{*p} = e_{*p}$ or $a_{-p,*} = e_{-p,*}$. Assume that there exist elements $g_1, g_2 \in \Ep(\sigma, \Gamma)$ 
    such that $g_1 a g_2 \in H$. Then the inclusion $a_{qj} \in \sigma_{qj}$ holds for each $j \neq -p$. 
    If additionally $a \in \Sp(\sigma)$ then also $S_{q,-q}(a) \in \Gamma_q$.
    \begin{proof}
        As $a$ is symplectic it's easy to see that the conditions of the lemma
        provide the equalities
        $a_{*p} = a'_{*p} = e_{*p}$ and $a_{-p,*} = a_{-p,*}' = e_{-p,*}$. Choose via
        Corollary \ref{cor:spgen:common:denominator} an element $x \in S$
        such that ${}^{g_2^{-1}} T_{pq}(x) \in H$ and 
        consider the matrix
        \begin{align*}
            b = a^{-1} T_{pq}(x) a &= 
            e + a'_{*p} x a_{q*} - \varepsilon_p \varepsilon_q a'_{*,-q} x a_{-p,*} \\
            &= e + e_{*p} x a_{q*} - \varepsilon_p \varepsilon_q a'_{*,-q} x e_{-p,*}.
        \end{align*}
        It's easy to see that $b_{ij} = \delta_{ij}$ whenever $i \neq p$ and $j \neq -p$, that
        $b_{pp} = b_{-p,-p} = 1$ and that
        \[
            {}^{g_1} b = (g_1 a g_2)(g_2^{-1} T_{pq}(x) g_2)(g_2^{-1} a^{-1} g_1^{-1}) \in H.
        \]
        Therefore by Lemma
        \ref{lemma:spgen:delta:outside:one:row} the inclusion
        $b_{pj} \in \sigma_{pj}$ holds for each $j \neq \pm p$. Note that
        $b_{pj} = x a_{qj}$ whenever $j \neq \pm p$. Therefore $a_{qj} \in \sigma_{qj}$ for all $j \neq -p$.
        By Lemma \ref{lemma:spgen:delta:inside:one:column} we also get the
        inclusion $S_{p,-p}(b) \in \Gamma_q$. Assume $a \in \Sp(\sigma)$. By the
        corollary \ref{cor:symplss:length:of:root:element} we have
        \[
            S_{p,-p}(b) \equiv a^{\prime \; 2}_{pp} x^2 S_{q,-q}(a) + 
            a^{\prime \; 2}_{p,-q} x^2 S_{-p,p}(a) \mod \Gamma_p^{\min}.
        \]
        Recall that $a_{-p,*} = e_{-p,*}$. Thus $S_{-p,p}(a) = 0$. Further $a'_{pp} = 1$,
        and therefore $S_{p,-p}(b) \equiv S_{q,-q}(a) \mod \Gamma_p$. Hence, $S_{q,-q}(a) \in \Gamma_q$.
    \end{proof}
\end{lemma}

\begin{lemma}
    \label{lemma:spgen:delta:outside:two:equiv:rows}
    Assume $h(\nu) \geq (4,4)$. Let $(p,q)$ be an $A$-type base pair and $a$ be an element of $\Sp(2n,R)$
    such that $a_{ij} = \delta_{ij}$ whenever $i \neq -p, -q$ and $j \neq p,q$. Assume that there exists
    an element $g \in \Ep(\sigma, \Gamma)$ such that ${}^g a \in H$. Then the inclusion
    $a_{kp} \in \sigma_{kp}$ holds for each $k \neq -p,-q$. If additionally $a \in \Sp(\sigma)$
    then also $S_{-p,p}(a) \in \Gamma_{-p}$.
    \begin{proof}
        Fix any $k \nsim p$. As $h(\nu) \geq (4,4)$, there exists an index $h \sim k$ such that
        $h \neq \pm k, \pm p, \pm q$. Pick using Corollary
        \ref{cor:spgen:common:denominator}
        an element $x \in S$ such that ${}^g T_{hk}(x) \in H$ and
        consider the matrix 
        \[
            b = a^{-1} T_{hk}(x) a = e + a'_{*h} x a_{k*} 
            - \varepsilon_h \varepsilon_k a'_{*,-k} x a_{-h,*}.
        \]    
        It is easy to see that $b_{hp} = x a_{kp}$ and ${}^g b \in H$. 
        Further, there exists an index $l \sim k$ such that $l \neq \pm h, -p$
        and $b_{*l} = e_{*l}$. Indeed, if $k \sim -k$, one can simply take $l = -k$. 
        If the class of $k$ is non-self-conjugate then such $l$ exists due to 
        the condition $h(\nu) \geq (4,4)$ ($l$ can be equal to $-q$ if $-q \sim k$). 
        Therefore, by Lemma \ref{lemma:spgen:delta:inside:one:column}
        we get $x a_{kp} \in \sigma_{hp} = \sigma_{kp}$. Thus $a_{kp} \in \sigma_{kp}$.

        Assume $a \in \Sp(\sigma)$. If the equivalence class of $p$ is self-conjugate, 
        clearly $S_{-p,p}(a) \in R = \Gamma_{-p}$. If the equivalence class 
        of $p$ is non-self-conjugate then, as $h(\nu) \geq (4,4)$, there exists an index
        $t \sim p$ such that $(p,q,t)$ is an $\A$-type base triple. Consider the matrix
        \[
            c = T_{-p,-t}(-a_{-p,-t}) T_{-q,-t}(a_{-q,-t}) a.
        \]
        As $t \sim p \sim q$ it follows that $c \in \Sp(\sigma)$. Note that $a_{-t,-t} = 1$,
        hence $c_{*,-t} = e_{*,-t}$ and
        \[
            g T_{-q,-t}(a_{-q,-t}) T_{-p,-t}(a_{-p,-t}) c g^{-1} \in H.
        \] 
        By Lemma
        \ref{lemma:spgen:delta:inside:one:column} it follows that $S_{-p,p}(c) \in \Gamma_{-p}$. Finally,
        by Corollary \ref{cor:symplss:length:of:left:mult:by:transv} we have
        \[
            S_{-p,p}(c) \equiv S_{-p,p}(a) + a_{-p,-t}^2 S_{-t,t}(a) \mod \Gamma_{-p}
        \]
        and as $S_{-t,t}(a) = 0$ we get $S_{-p,p}(a) \in \Gamma_{-p}$.
    \end{proof}
\end{lemma}

\begin{lemma}
    \label{lemma:spgen:delta:outside:pm}
    Assume $h(\nu) \geq (4,4)$. Let $p$ be an index in $I$ with self-conjugate equivalence class and let a 
    be an element of $\Sp(2n,R)$
    such that $a_{ij} = \delta_{ij}$ whenever $i \neq \pm p$ and $j \neq \pm p$. If there exists
    an element $g \in \Ep(\sigma, \Gamma)$ such that ${}^g a \in H$, then $a_{kp} \in \sigma_{kp}$ for all $k \in I$.
    \begin{proof}
        If $k \sim p$ the inclusion $a_{kp} \in \sigma_{kp}$ is trivial.
        Assume $k \nsim p$, in particular $k \neq \pm p$. As $h(\nu) \geq (4,4)$, there exists
        another index $h \sim k$ such that $h \neq \pm k \pm p$.     
        Pick using Corollary \ref{cor:spgen:common:denominator} an element $x \in S$
        such that ${}^g T_{hk}(x) \in H$ and consider the matrix 
        \[
            b = a^{-1} T_{hk}(x) a = e + a'_{*h} m a_{k*} 
            - \varepsilon_h \varepsilon_k a'_{*,-k} m a_{-h,*}.
        \]
        We will show that it satisfies the conditions of Lemma 
        \ref{lemma:spgen:delta:inside:one:column}.
        Indeed, by choice of $x$ the inclusion ${}^g b \in H$ holds.
        Pick an index $q$ such that $(p,q)$ is a $\C$-type base pair. 
        It is easy to see that 
        $q \nsim \pm k$. Clearly $b_{*q} = e_{*q}$. 
        Applying Lemma \ref{lemma:spgen:delta:inside:one:column} to the matrix $a$, the elementary
        transvection $T_{hk}(x)$ and the pair $(-p,q)$, we get $b_{-p,j} \in \sigma_{-p,j}$ for all $j \neq -q$.
        Thus $b_{-p,-h} \in \sigma_{-p,-h} = \sigma_{kp}$ and it is only left to notice that
        $b_{-p,-h} = - \varepsilon_h \varepsilon_k a'_{-p,-k} x a_{-h,-h} = \pm x a_{kp}$.
        Therefore $a_{kp} \in \sigma_{kp}$ for all $k \in I$.
    \end{proof}
\end{lemma}

\paragraph{Extraction of transvections using root elements}

\begin{lemma}
    \label{lemma:spgen:short:alpha:beta:not:selfconj}
    Assume $h(\nu) \geq (4,4)$. Let $(p,q,h)$ be an $\A$-type base triple,
    $a$ an element of $\Sp(2n,R)$ and $T_{sr}(\xi)$ a short elementary 
    transvection. Let $b$ denote the
    short root element $a T_{sr}(\xi) a^{-1}$. Suppose that 
    $a_{p,-r} = a_{q,-r} = 0$. Assume that there exist elements 
    $g_1, g_2 \in \Ep(\sigma, \Gamma)$ such that $g_1 a g_2 \in H$ 
    and ${}^{g_2^{-1}} T_{sr}(\xi) \in H$. Then $a_{ps} b_{ih} \in \sigma_{ih}$ for all $i \neq -p, -q$.
    If additionally $a \in \Sp(\sigma)$ then also 
    $a_{ps}^2 S_{-h,h}(b^{-1}) \in \Gamma_{-h}$.
    \begin{proof}
        It is easy to see that ${}^{g_1} b \in H$. Indeed
        \[
            {}^{g_1} b = (g_1 a g_2) (g_2^{-1} T_{sr}(\xi) g_2) (g_2^{-1} a^{-1} g_1^{-1}) \in H.
        \]
        Using Corollary \ref{cor:spgen:common:denominator} pick an element $x \in S$ such that
        ${}^{g_1} T_{hp}(-a_{qs} x), {}^{g_1} T_{hq}(a_{ps} x) \in H$. Set 
        $\alpha = - a_{qs} x$ and $\beta = a_{ps} x$. Clearly, $\alpha a_{ps} + \beta a_{qs} = 0$.
        Consider the matrix 
        \begin{equation}
            \label{eq:lemma:spgen:short:alpha:beta:not:selfconj:c}
            c = b T_{hp}(\alpha) T_{hq}(\beta) b^{-1} 
              = e + b_{*h} (\alpha b'_{p*} + \beta b'_{q*}) 
              - \varepsilon_h (\varepsilon_p \alpha b_{*,-p} + 
              \varepsilon_q \beta b_{*,-q}) b'_{-h,*}.
        \end{equation}
        Obviously, ${}^{g_1} c \in H$ and by the conditions that $a_{p,-r} = a_{q,-r} = 0$
        and $\alpha a_{ps} + \beta a_{qs} = 0$ it easily
        follows that $(\alpha b'_{p*} + \beta b'_{q*}) = (\alpha e_{p*} + \beta e_{q*})$
        and $(\varepsilon_p \alpha b_{*,-p} + \varepsilon_q \beta b_{*,-q})
        = (\varepsilon_p \alpha e_{*,-p} + \varepsilon_q \beta e_{*,-q})$.
        Thus we can rewrite \eqref{eq:lemma:spgen:short:alpha:beta:not:selfconj:c} as follows
        \[
            c = e + b_{*h} (\alpha e_{p*} + \beta e_{q*}) 
              - \varepsilon_h (\varepsilon_p \alpha e_{*,-p} + 
              \varepsilon_q \beta e_{*,-q}) b'_{-h,*}.
        \]      
        In particular, $c_{ij} = \delta_{ij}$ whenever $i \neq -p,-q$ and $j \neq p,q$.
        Applying Lemma
        \ref{lemma:spgen:delta:outside:two:equiv:rows} to the matrix $c$ we get the inclusion 
        $c_{iq} \in \sigma_{iq}$ for all $i \neq -p,-q$. Notice that  
        $c_{iq} = \beta b_{ih} = x a_{ps} b_{ih}$ for $i \neq -p,-q$. This completes the 
        proof of the first part of the lemma.

        Assume $a \in \Sp(\sigma)$. It follows that $b, c \in \Sp(\sigma)$. Therefore by Lemma
        \ref{lemma:spgen:delta:outside:two:equiv:rows} we get the inclusion $S_{-q,q}(c) \in \Gamma_{-q}$.
        By Corollary \ref{cor:symplss:length:of:root:element} we have
        \begin{equation}
            \label{eq:lemma:spgen:short:alpha:beta:not:selfconj:1}
            \begin{aligned}
                S_{-q,q}(c) & \equiv
                b^2_{-q,h} x^2 (a_{qs}^2 S_{p,-p}(b^{-1}) + a_{ps}^2 S_{q,-q}(b^{-1})) \\
                &\quad + (b_{-q,-p}^2 a_{qs}^2 + b_{-q,-q}^2 a_{ps}^2) x^2 S_{-h,h}(b^{-1}) \in \Gamma_{-q}.
            \end{aligned}
        \end{equation}                
        As $S_{-h,h}(b^{-1}) \in \sigma_{-q,q}$ and $2 \sigma_{-q,q} \le \Gamma_i$ it follows that
        \begin{equation}
            \label{eq:lemma:spgen:short:alpha:beta:not:selfconj:2}
            \begin{aligned}
                (b_{-q,-p}^2 a_{qs}^2 + b_{-q,-q}^2 a_{ps}^2) x^2 S_{-h,h}(b^{-1}) \equiv \\
                (b_{-q,-p} a_{qs} + b_{-q,-q} a_{ps})^2 x^2 S_{-h,h}(b^{-1}) \mod \Gamma_{-q}.
            \end{aligned}
        \end{equation}
        A straightforward calculation shows that $b_{-q,-p} a_{qs} + b_{-q,-q} a_{ps} \in a_{ps} + 2R$,
        which together with \eqref{eq:lemma:spgen:short:alpha:beta:not:selfconj:2} yields
        \begin{equation}
            \label{eq:lemma:spgen:short:alpha:beta:not:selfconj:3}
            (b_{-q,-p}^2 a_{qs}^2 + b_{-q,-q}^2 a_{ps}^2) x^2 S_{-h,h}(b^{-1}) \equiv 
            a_{ps}^2 x^2 S_{-h,h}(b^{-1}) \mod \Gamma_{-q}.
        \end{equation}
        As $a_{p,-r} = a_{q,-r} = 0$, we have
        \begin{align*}
            b'_{pj} &= \delta_{pj} - a_{ps} \xi a'_{rj} \\
            b'_{qj} &= \delta_{qj} - a_{qs} \xi a'_{rj}.
        \end{align*}
        In particular, $b'_{p,-p} = b'_{q,-q} = b'_{p,-q} = b'_{q,-p} = 0$.
        Therefore
        \begin{align*}
            a_{qs}^2 S_{p,-p}(b^{-1}) &= - \varepsilon_p \sum_{j>0, j \neq \pm p, \pm q} a_{qs}^2 b'_{pj} b'_{p,-j}
            = - \varepsilon_p \sum_{j>0, j \neq \pm p, \pm q} a_{qs}^2 a_{ps}^2 \xi^2 a'_{rj} a'_{r,-j} \\
            &= \varepsilon_p \varepsilon_q a_{ps}^2 S_{q,-q}(b^{-1})
        \end{align*}
        and thus 
        \begin{equation}
            \label{eq:lemma:spgen:short:alpha:beta:not:selfconj:4}
           b^2_{-q,h} x^2 (a_{qs}^2 S_{p,-p}(b^{-1}) + a_{ps}^2 S_{q,-q}(b^{-1})) \in 2 \sigma_{-q,q} \le \Gamma_{-q}.
        \end{equation}
        Combining \eqref{eq:lemma:spgen:short:alpha:beta:not:selfconj:1},
        \eqref{eq:lemma:spgen:short:alpha:beta:not:selfconj:2}, \eqref{eq:lemma:spgen:short:alpha:beta:not:selfconj:3}
        and \eqref{eq:lemma:spgen:short:alpha:beta:not:selfconj:4} we get the inclusion
        $a_{ps}^2 x^2 S_{-h,h}(b^{-1}) \in \Gamma_{-h}$. By the property ($\Gamma 2$) of a form net of ideals
        it follows that $a_{ps}^2 S_{-h,h}(b^{-1}) \in \Gamma_{-h}$.
    \end{proof}
\end{lemma}

\begin{lemma}
    \label{lemma:spgen:long:alpha:beta:selfconj}
    Assume $h(\nu) \geq (4,4)$. Let $(p,q)$ be a $\C$-type base pair,
    $a$ an element of $\Sp(2n,R)$ and $T_{s,-s}(\xi)$ 
    a long elementary transvection. Let $b$ denote the
    long root element $a T_{s,-s}(\xi) a^{-1}$. If there exist elements 
    $g_1, g_2 \in \Ep(\sigma, \Gamma)$ such that $g_1 a g_2 \in H$ 
    and ${}^{g_2^{-1}} T_{s,-s}(\xi) \in H$, then $a_{ps} b_{ih} \in \sigma_{ih}$
    for all $i \in I$.
    \begin{proof}        
        Clearly ${}^{g_1} b \in H$. Indeed
        \[ {}^{g_1} b = g_1 a T_{s,-s}(\xi) a^{-1} g_1^{-1} = (g_1 a g_2) ({}^{g_2^{-1}} T_{s,-s}(\xi))
        (g_2^{-1} a^{-1} g_1^{-1}) \]
        and each term in brackets is contained in $H$.        
        Pick $x \in S$ such that ${}^{g_1} T_{hp}(-a_{-p,s} x)$, 
        ${}^{g_1} T_{h,-p}(a_{ps} x) \in H$. Set $\alpha = -a_{-p,s} x$ 
        and $\beta = a_{ps} x$. Clearly $a_{ps} \alpha + a_{-p,s} \beta = 0$.
        It easily follows that $b'_{p*} \alpha + b'_{-p,*} \beta = e_{p*} \alpha + e_{-p,*} \beta$
        and $(\varepsilon_p \alpha b_{*,-p} + \varepsilon_{-p} \beta b_{*p})
        = (\varepsilon_p \alpha e_{*,-p} + \varepsilon_{-p} \beta e_{*p})$.         
        Consider the matrix 
        \begin{align*}
            c = b T_{hp}(\alpha) T_{h,-p}(\beta) b^{-1}
              &= e + b_{*h} (b'_{p*} \alpha + b_{-p,*} \beta)
                  - \varepsilon_h (\varepsilon_p \alpha b_{*,-p} + \varepsilon_{-p} \beta b_{*p}) b'_{-h,*} \\
              &= e + b_{*h} (e_{p*} \alpha + e_{-p,*} \beta)
                  - \varepsilon_h (\varepsilon_p \alpha e_{*,-p} + \varepsilon_{-p} \beta e_{*p}) b'_{-h,*}.
        \end{align*}
        It is easy to see that $c_{ij} = \delta_{ij}$ whenever $i \neq \pm p$ and $j \neq \pm p$. By  
        Lemma \ref{lemma:spgen:delta:outside:pm} it follows that $c_{i,-p} \in \sigma_{ip}$ for all $i \in I$.
        It's only left to notice that $c_{i,-p} = \beta b_{ih} = x a_{ps} b_{ih}$ whenever $i \neq \pm p$.
        Hence $a_{ps} b_{ih} \in \sigma_{ih}$ for all $i \in I$.
    \end{proof}
\end{lemma}

\section{At the level of the Jacobson radical}
\label{sec:spgen:jac}

The whole benefit of using localization based proofs is that we can work over a local ring, where
every element is either invertible or is contained in the unique maximal ideal, which coincides with 
the Jacobson radical. Thus, the extraction of transvections over local rings rely on
certain elements either being invertible or being contained in the Jacobson radical. In this section, we
investigate the latter case. Consider a standard setting $(R,R',S)$. Let $J$ denote the Jacobson radical 
of $R$. We will show (Corollary \ref{cor:spgen:jacobson:intersection:with:H})
that the matirces in an overgroup $H$ of $\Ep(\nu, R')$ with sufficiently large $h(\nu)$, that
are contained in $\Sp(2n, R, J)$, are in fact contained in $\Sp(\sigma, \Gamma)$, where, as usual,
$(\sigma, \Gamma)$ stands for the closure of a net $S$-associated with $H$. In fact we will show
that $H \cap \Sp_J \le \Sp(\sigma, \Gamma)$, where $\Sp_J$ is some set of matrices  larger
than $\Sp(2n,R,J)$.

Throughout this section we fix a standard setting $(R, R', S)$, a unitary equivalence relation $\nu$,
a subgroup $H$ of $\Sp(2n, R)$ which contains $\Ep(\nu, R')$ and an exact major form net of ideals $(\sigma', \Gamma')$ 
over $R'$ which is $S$-associated with $H$.
Let $J$ denote the Jacobson radical of the ring $R$ and 
let $(\sigma, \Gamma)$ denote the $S$-closure in $R$ of the form net of ideals $(\sigma', \Gamma')$.
By definition, every element $x$ of the Jacobson radical $J$ is quasi-regular,
i.e. $1 + xy \in R^*$ for all $y \in R$, in particular $1 + x \in R^*$.
It follows that $R^* + J \subseteq R^*$. Indeed, let
$x$ be invertible and $y \in J$ then, as $y$ is quasi-regular, $1 + x^{-1} y \in R^*$.
Therefore, and $x + y = x(1 + x^{-1} y)$ is a unit since it is a product of two units. We will use
this property without reference.

The main idea of this section is as follows. Pick a matrix $a \in H \cap \Sp(2n, R, J)$. 
Fix an $A$-type base quintuple $(p,q,h,t,l)$. We would like to show that $a_{ip} \in \sigma_{ip}$
and $S_{-p,p}(a^{-1}) \in \Gamma_{-p}$. So far the only way of extracting transvections we know
is either from an element close to a parabolic subgroup or from a root element. There is no much hope
that an arbitrary element of $H \cap \Sp(2n, R, J)$ is any close to a parabolic subgroup. Thus we
have to get by with the second option. Let's construct a root element $b = a T_{sr}(\xi) a^{-1}$
which will preserve the information about the original element $a$. Once we show that 
$b_{a,-r} = b_{a,-r} = 0$ it follows by Lemma \ref{lemma:spgen:short:alpha:beta:not:selfconj}
that 
\begin{equation}
    \label{eq:sec:spgen:jac:1}
    a_{ps} (\delta_{ih} + a_{is} \xi a'_{rh} - \varepsilon_s \varepsilon_r a_{i,-r} \xi a'_{-s,h}) =  a_{ps} b_{ih} \in \sigma_{ih}.
\end{equation}
In order to get the inclusion $a_{ip} \in \sigma_{ip}$ we 
have to put either $s=p$ or $-r = p$. Let's choose $s = p$. Note that in this case
$a_{ps} = a_{pp}$ is invertible. In order to get one step closer it would be beneficial to have
$a'_{rh}$ also invertible. This happens for example when $r = h$. Now suppose that $a_{-p,h} = a_{-s,h} = 0$.
In this case the inclusion \eqref{eq:sec:spgen:jac:1} yields the inclusion $a_{ip} \in \sigma_{ip}$ immediately!
This is exactly what we do in Lemma \ref{lemma:spgen:jacobson:not:selfconj}. Lemmas \ref{lemma:spgen:antidiag:weak}
and \ref{lemma:spgen:jacobson:antidiag:strong} serve the goal of getting the condition $a_{-p,h} = 0$
fulfilled. The inclusion $S_{-p,p}(a^{-1}) \in \Gamma_{-p}$ is handled in a similar manner.
The case when we work with a $\C$-type pair rather than an $\A$-type quintuple is quite a bit easier
and is handled by Lemma \ref{lemma:spgen:jac:extract:C}. Everywhere in this section we try to use as little 
of the assumption that $a \in \Sp(2n, R, J)$ as possible, namely we explicitly list all the individual inclusions for
entries of $a$ that we require. This allows us in future to work not only with a matrix in $H \cap \Sp(2n, R, J)$
but with any matrix $a \in H$ having a submatrix that looks like a submatrix of a matrix in $\Sp(2n, R, J)$.

\begin{lemma}
    \label{lemma:spgen:antidiag:weak}
	Assume $h(\nu) \geq (4,4)$. Let $(p,q,h,t,l)$ be an $A$-type base 
	quintuple and $a$ an element of $\Sp(2n,R)$ such that
	$a_{-p,-t} = a_{-h,-t} = a_{-l,-t} = a_{pq} = a_{hq} = 0$,
	$a_{-p,q} a'_{t,-p}$, $a_{-h,q} a'_{t,-h} \in J$ and $a_{qq} \in R^*$.
	If there exist elements $g_1$ and $g_2$ in $\Ep(\sigma,\Gamma)$ such
	that $g_1 a g_2 \in H$ then $a_{p,-t} \in \sigma_{p,-t}$.
	\begin{proof}
		Pick using Corollary \ref{cor:spgen:common:denominator} an element $x \in S$
		such that ${}^{g_2^{-1}} T_{qt}(x) \in H$ and consider the
		matrix $b = a T_{qt}(x) a^{-1}$.
		By choice of the parameter $x$ we have
		\[
			{}^{g_1} b = (g_1 a g_2) (g_2^{-1} T_{qt}(x) g_2) (g_2^{-1} a^{-1} a_1^{-1})
			\in H.
		\]
		Pick using Corollary \ref{cor:spgen:common:denominator}
		another element $y \in S$ such that ${}^{g_1} T_{-p,-h}(y) \in H$ and consider the
		matrix $c = b T_{-p,-h}(y) b^{-1}$. Clearly, ${}^{g_1} c \in H$ for the
		same reason as above. We are going to apply Lemma 
		\ref{lemma:spgen:short:alpha:beta:not:selfconj}
		to the matrix $b$, the short elementary transvection 
		$T_{-p,-h}(y)$ and the $A$-type base triple $(-p,-l,-h)$. In order
		to do this we have to show first that $b_{-p,h} = b_{-l,h} = 0$.
		Indeed, by the assumptions of this lemma
		$a'_{th} = \varepsilon_t \varepsilon_h a_{-h,-t} = 0$ and
		$a_{-p,-t} = a_{-l,-t} = 0$. Thus
		\begin{align*}
			b_{-p,h} &= a_{-p,q} x a'_{th} - \varepsilon_q \varepsilon_t a_{-p,-t}
				x a'_{-q,h} = 0 \\
			b_{-l,h} &= a_{-l,q} x a'_{th} - \varepsilon_q \varepsilon_t a_{-l,-t}
				x a'_{-q,h} = 0.
		\end{align*}
		As $a_{-p,q} a'_{t,-p} \in J$, it follows that
		\begin{align*}
			b_{-p,-p} &= 1 + a_{-p,q} x a'_{t,-p} -
				\varepsilon_q \varepsilon_t a_{-p,-t} x a'_{-q,-p} \\
				&= 
				1 + a_{-p,q} x a'_{t,-p} \in 1 + J \subseteq R^*.
		\end{align*}
		Therefore by Lemma \ref{lemma:spgen:short:alpha:beta:not:selfconj}		
		we get the inclusions $b_{-p,-p} c_{i,-h} \in \sigma_{i,-h}$ and,
		as $b_{-p,-p}$ is invertible, also
		$c_{i,-h} \in \sigma_{i,-h}$ for all $i \neq p,l$. In particular,
		\begin{equation}
			\label{eq:lemma:spgen:antidiag:weak:c:as:b}
			b_{q,-p} y b'_{-h,-h} - \varepsilon_p \varepsilon_h
			b_{qh} y b'_{p,-h}
			= c_{q,-h} \in \sigma_{q,-h}.
		\end{equation}
		Recall that $a_{-h,q} a'_{t,-h} \in J$ Therefore
		\begin{equation}
			\label{eq:lemma:spgen:antidiag:weak:bhh}
			\begin{aligned}
    			b'_{-h,-h} &= 1 - a_{-h,q} x a'_{t,-h} +
			    \varepsilon_q \varepsilon_t
				a_{-h,-t} x a'_{-q,-h} =
				1 - a_{-h,q} x a'_{t,-h} \in 1 + J \subseteq R^*.
			\end{aligned}
		\end{equation}
		Observe that
		\begin{equation}
			\label{eq:lemma:spgen:antidiag:weak:bph}
			b'_{p,-h} = - a_{pq} x a'_{t,-h} + \varepsilon_q \varepsilon_t
				a_{p,-t} x a'_{-q,-h} = 0.
		\end{equation}
		Substituting \eqref{eq:lemma:spgen:antidiag:weak:bhh} and
		\eqref{eq:lemma:spgen:antidiag:weak:bph} into
		\eqref{eq:lemma:spgen:antidiag:weak:c:as:b} we get the inclusion
		$b_{q,-p} \in \sigma_{q,-h} = \sigma_{q,-p}$. Finally, recall
		that $a_{qq} \in R^*$ and
		$a'_{-q,-p} = \varepsilon_q \varepsilon_p a_{pq} = 0$. Therefore
		\[
			a_{qq} x a'_{t,-p} = 
			a_{qq} x a'_{t,-p} - \varepsilon_q \varepsilon_t a_{q,-t} x
			a'_{-q,-p} = b_{q,-p} \in \sigma_{q,-p}.
		\] 
		Thus $a'_{t,-p} \in \sigma_{q,-p}$. It only remains to notice that 
		\[
			a_{p,-t} = - \varepsilon_p \varepsilon_t a'_{t,-p} \in \sigma_{q,-p} = \sigma_{p,-t}.
		\]		
	\end{proof}
\end{lemma}

\begin{lemma}
    \label{lemma:spgen:jacobson:antidiag:strong}
	Assume $h(\nu) \geq (4,4)$. Let $(p,q,h,t,l)$ be an $A$-type base 
	quintuple and $a$ an element of $\Sp(2n,R)$ such that
	at least one of the following three conditions holds:
	\begin{enumerate}
		\item The elements $a_{-p,-t}, a_{-h,-t}, a_{pq},
		a_{hq}, a_{-p,q} a'_{t,-p}$ and $a_{-h,q} a'_{t,-h}$ are contained in the 
		Jacobson radical and the entries $a_{pp}, a_{qq}$ and $a_{-t,-t}$ are
		invertible.
		\item The rows $a_{p*}$, $a_{q*}$, $a_{h*}$ and the column $a_{*,-t}$
			coincide modulo the Jacobson radical with the corresponding
			rows and columns of the identity matrix.
		\item The rows $a_{-t,*}, a_{-h,*}, a_{-p,*}$ and columns $a_{*p}$
    		and $a_{*q}$ coincide modulo the Jacobson radical 
    		with the corresponding
			rows and columns of the identity matrix.
	\end{enumerate}
	If there exist $g_1, g_2 \in \Ep(\sigma, \Gamma)$ such that
	$g_1 a g_2 \in H$ then $a_{p,-t} \in \sigma_{p,-t}$.
	\begin{proof}
		Note that the first condition in the statement of this lemma
		trivially follows from any of the others. Consider the matrix
		\[
			b = a T_{pq}(-a_{pp}^{-1} a_{pq}).
		\]
		By the assumption that the entry $a_{pq}$ is contained in the Jacobson 
		radical $b \equiv a \mod J$. Clearly, $b_{pq} = 0$, 
		$b_{p,-t} = a_{p,-t}$ and $b_{qq}, b_{-t,-t} \in R^*$. Further, consider the
		matrix
		\[
			c = T_{hq}(-b_{hq} b_{qq}^{-1}) T_{-p,-t}(-b_{-p,-t} b_{-t,-t}^{-1}) 
			T_{-h,-t}(-b_{-h,-t} b_{-t,-t}^{-1}) T_{-l,-t}(-b_{-l,-t} b_{-t,-t}^{-1}) b.
		\]
		As $b \equiv a \mod J$, 
		we have $b_{hq}, b_{-p,-t}$ and $b_{-h,-t}$ in the Jacobson radical. Therefore
		$c_{i*} \equiv  a_{i*} \mod J$ whenever $i \neq t,-l$ (and also
		$c'_{*j} \equiv a'_{*j} \mod J$ whenever $j \neq l,-t$). In particular
		$c_{-p,q} c'_{t,-p}, c_{-h,q} c'_{t,-h} \in J$ and $c_{qq} \in R^*$.
		It is easy to see that $c_{pq} = c_{hq} = c_{-p,-t} = c_{-h,-t} = c_{-l,-t} = 0$.
		Finally, $g_3 c g_4 \in H$, where
		\begin{align*}
			g_3 &= g_1 \left( T_{hq}(-b_{hq} b_{qq}^{-1}) T_{-p,-t}(-b_{-p,-t} b_{-t,-t}^{-1})
			\times \right. \\
			&\quad\quad\;\; \left. 
			T_{-q,-t}(-b_{-q,-t} b_{-t,-t}^{-1}) T_{-l,-t}(-b_{-l,-t} b_{-t,-t}^{-1})
			\right)^{-1}, \\
			g_4 &= T_{pq}(a_{pp}^{-1} a_{pq}) g_2.
		\end{align*}
		Clearly, $g_3$ and $g_4$ are contained in $\Ep(\sigma, \Gamma)$.
		Therefore, $c$ satisfies the conditions of Lemma \ref{lemma:spgen:antidiag:weak}
		and it follows that $c_{p,-t} \in \sigma_{p,-t}$. It's only left to notice
		that $c_{p,-t} = a_{p,-t}$.
	\end{proof}
\end{lemma}

\begin{lemma}
    \label{lemma:spgen:jacobson:not:selfconj}
	Assume $h(\nu) \geq (4,4)$. Let $(p,q,h,t)$ be an $A$-type base 
	quadruple and $a$ an element of $\Sp(2n, R)$ such that
	$a_{p,-h}, a_{q,-h}, a_{t,-h} \in \sigma_{p,-p} \cap J$, 
	$a_{-h,p} \in \sigma_{-h,p} \cap J$. Suppose $a_{qp} \in J$ and 
	$a_{pp}, a_{-h,-h} \in R^*$ and suppose that there exists an element $g \in \Ep(\sigma, \Gamma)$
	such that ${}^g a \in H$. Then $a_{ip} \in \sigma_{ip}$ for all $i \in I$. If 
	additionally $a \in \Sp(\sigma)$ then also $S_{-h,h}(a^{-1}) \in \Gamma_{-h}$.
	\begin{proof}
		Consider the matrix 
		\[
			b = T_{-h,p}(- a_{-h,p} a_{pp}^{-1}) a.
		\] As $a_{-h,p} \in J$ it follows that
		$b \equiv a \mod J$. Additionally $b_{p,-h}, b_{q,-h}$ and $b_{t,-h}$
		are contained in $\sigma_{p,-p}$ and $b_{-h,p} = 0$. Consider the matrix
		\[
			c = T_{p,-h}(- b_{p,-h} b_{-h,-h}^{-1}) 
			T_{q,-h}(- b_{q,-h} b_{-h,-h}^{-1})
			T_{t,-h}(- b_{t,-h} b_{-h,-h}^{-1}) b.
		\]
		Again, $c \equiv a \mod J$, in particular $c_{pp}, c_{-h,-h} \in R^*$. 
		As $c$ is a symplectic matrix, 
		it also follows that $c'_{-p,h} = 0$ and $c'_{hh} \in R^*$.
		It's easy to see that $c_{p,-h} = c_{q,-h} = c_{t,-h} = c_{-h,p} = 0$. 
		Finally, $g g_1 c g^{-1}$ is contained in $H$, where
		\begin{align*}
			g_1 &= \left( T_{p,-h}(-b_{-h,-h}^{-1} b_{p,-h}) 
			T_{q,-h}(-b_{-h,-h}^{-1} b_{q,-h}) \times \right. \\
			&\quad\quad\; \left.
			T_{t,-h}(-b_{-h,-h}^{-1} b_{t,-h}) 
			T_{-h,p}(-a_{pp}^{-1} a_{-h,p}) \right)^{-1} \in \Ep(\sigma, \Gamma).
		\end{align*}
		
		Pick an element $x \in S$ such that ${}^g T_{ph}(x) \in H$.
		Applying Lemma \ref{lemma:spgen:short:alpha:beta:not:selfconj}
		 to the matrix
		$c$, the short symplectic transvection $T_{ph}(x)$ and the $A$-type base
		triple $(p,q,h)$, we get for all $i \neq -p, -q$ the inclusion $c_{pp} (c T_{ph}(x) c^{-1})_{ih} \in \sigma_{ih}$.
		As $c_{pp}$ is invertible, we also get
		\begin{equation}
			\label{eq:lemma:spgen:jacobson:not:selfconj:c}
			\delta_{ih} + c_{ip} x c'_{hh} 
			- \varepsilon_p \varepsilon_h c_{i,-h} x
			c'_{-p,h} = (c T_{ph}(x) c^{-1})_{ih} \in \sigma_{ih}
		\end{equation}
		for all $i \neq -p,-q$. We can apply Lemma \ref{lemma:spgen:short:alpha:beta:not:selfconj}
		to the same matrix and transvection, but to a different $\A$-type base triple
		$(p,t,h)$ and get the inclusion \eqref{eq:lemma:spgen:jacobson:not:selfconj:c} also for
		$i = -q$. As $c'_{hh}$ is invertible and $c'_{-p,h} = 0$, 
		it follows from \eqref{eq:lemma:spgen:jacobson:not:selfconj:c}
		that $c_{ip} \in \sigma_{ih}$ for all $i \neq h,-p$. Observe that
		$a_{ip} = c_{ip}$ for all $i \neq p,q,t, -h$. Therefore $a_{ip} \in \sigma_{ip}$ for all 
		$i \neq p,q,t,-h,-p$. The inclusions $a_{ip} \in \sigma_{ip}$ for $i = p,q$ and $t$ are trivial and
		the corresponding inclusion for $i = -h$ is provided by the assumptions of the
		lemma. Therefore $a_{ip} \in \sigma_{ip}$ for all $i \neq -p$.
		
		Pick an element $y \in S$ such that ${}^g T_{pq}(y) \in H$
		and consider the matrix $d = T_{pq}(y) a$. Clearly, it satisfies all the conditions
		of this lemma. Indeed, $d_{p,-h} = a_{p,-h} + y a_{q,-h} \in \sigma_{p,-p} \cap J$,
		$d_{pp} = a_{pp} + y a_{qp} \in R^* + J \le R^*$ and the rest of the entries 
		of $d$ involved
		in the conditions of this lemma coincide with the corresponding entries of $a$
		itself. Thus we get the inclusions $d_{ip} \in \sigma_{ip}$ for all
		$i \neq -p$. In particular, $d_{-q,p} \in \sigma_{-q,p}$. It's only left to notice
		that $d_{-q,p} = a_{-q,p} - \varepsilon_p \varepsilon_q y a_{-p,p}$ and $a_{-q,p}$ 
		is already contained in $\sigma_{-q,p}$, while $y$ is invertible. Therefore 
		$a_{-p,p} \in \sigma_{-p,p}$.
		
		Assume $a \in \Sp(\sigma)$. By Lemma \ref{lemma:spgen:short:alpha:beta:not:selfconj} we get
		the inclusion $S_{-h,h}(c^{-1}) \in \Gamma_{-h}$. As $a^{-1} = c^{-1} g_1^{-1}$, by Corollary \ref{cor:symplss:length:of:left:mult:by:transv}
		we get $S_{-h,h}(a^{-1}) \in \Gamma_{-h}$.
	\end{proof}
\end{lemma}

We combine Lemmas \ref{lemma:spgen:jacobson:antidiag:strong}  and \ref{lemma:spgen:jacobson:not:selfconj} in the
following corollary.
\begin{corollary}
    \label{cor:spgen:jacobson:not:selfconj}
	Assume $h(\nu) \geq (4,4)$. Let $(p,q,h,t,l)$ be an $A$-type base 
	quintuple and $a$ an element of $\Sp(2n,R)$. Let $I'$ denote the 
	set $\{p,q,h,t\}$. Suppose that
	$a_{i*} \equiv e_{i*} \mod J$ and $a_{*,-i} \equiv e_{*,-i} \mod J$
	whenever $i \in I'$. Further, suppose that there exists an element $g \in \Ep(\sigma,\Gamma)$
	such that ${}^g a \in H$. Then $a_{ip} \in \sigma_{ip}$ for all $i \in I$.
	If additionally $a \in \Sp(\sigma)$ then also $S_{-p,p}(a^{-1}) \in \Gamma_{-p}$.
	\begin{proof}
		It's easy to see that the matrix $a$ satisfies condition (2) of 
		Lemma \ref{lemma:spgen:jacobson:antidiag:strong}. Thus we can conclude that 
		the entries $a_{p,-h},a_{q,-h}$ and $a_{t,-h}$ are contained in $\sigma_{p,-p}$.
		Moreover, the same entries are contained in the Jacobson radical by
		assumption. Since $a$ also satisfies the condition (3) of Corollary 
		\ref{lemma:spgen:jacobson:antidiag:strong},
		it follows that $a_{-h,p}$ is contained in $\sigma_{-h,p}$. Note that by assumption,
		$a_{pp}, a_{-h,-h} \in R^*$ and $a_{qp} \in J$. Summing up, $a$ satisfies
		the conditions of Lemma \ref{lemma:spgen:jacobson:not:selfconj}. Hence 
		$a_{ip} \in \sigma_{ip}$ for all $i \in I$. If $a \in \Sp(\sigma)$ then by 
		Lemma \ref{lemma:spgen:jacobson:not:selfconj} we get the inclusion 
		$S_{-h,h}(a^{-1}) \in \Gamma_{-h}$. Switching the indices $p$ and $h$ in the reasoning
		above, we get the required inclusion $S_{-p,p}(a^{-1}) \in S_{-p,p}$.
	\end{proof}
\end{corollary}

\begin{lemma}
	\label{lemma:spgen:jac:extract:C}
	Assume $h(\nu) \geq (4,4)$. Let $(p,h)$ be a $C$-type base 
	pair and $a$ an element of $\Sp(2n,R)$ such that	
	$a_{pp} \in R^*$ and $a_{-h,-p} \in J$. If 
	there exists an element $g \in \Ep(\sigma, \Gamma)$ such that ${}^g a \in H$, then 
	$a_{ip} \in \sigma_{ip}$ for all $i \in I$.
	\begin{proof}
		Consider the matrix 
		\[
			b = T_{hp}((-a_{hp}+1) a_{pp}^{-1}) a.
		\] Clearly, $b_{pp} = a_{pp}$ is invertible and $b_{hp} = 1$.
%		 and 
%		\[
%			b_{-p,-p} = a_{-p,-p} - \varepsilon_h \varepsilon_p 
%			(-a_{hp} + 1) a_{pp}^{-1} a_{-h,-p} \in R^* + J \le R^*.
%		\]
		Pick using Corollary \ref{cor:spgen:common:denominator}
		an element $x \in S$ such that ${}^g T_{p,-p}(x^2) \in H$. By Lemma
		\ref{lemma:spgen:long:alpha:beta:selfconj} we get the inclusions
		\[
		 	b_{pp} (\delta_{i,-h} + b_{ip} x^2 b'_{-p,-h}) 
		  = b_{pp} (b T_{p,-p}(x^2) b^{-1})_{i,-h} \in \sigma_{i,-h}
		\]	
		for all $i \in I$. It's only left to notice that $b_{pp}$ and $x^2 b'_{-p,-h}$ are
		invertible and thus $b_{ip} \in \sigma_{i,-h}$ for all $i \in I$.
		 Finally, $a_{ip} = b_{ip}$ whenever $i \neq h,-p$.
		Thus $a_{ip} \in \sigma_{ip}$.
%		\red{Seems like we don't need the assumption $a_{-p,-p} \in R^*$}
	\end{proof}
\end{lemma}

The following corollary is an illustration of application of Corollary \ref{cor:spgen:jacobson:not:selfconj} and 
Lemma \ref{lemma:spgen:jac:extract:C}. Suppose $R' = R$ and $S = \{ 1 \}$. Then it is clear that 
$(\sigma', \Gamma') = (\sigma, \Gamma)$ is the net associated with $H$ in $\Sp(2n,R)$. 
Corollary \ref{cor:spgen:jacobson:not:selfconj} together with Lemma \ref{lemma:spgen:jac:extract:C}
yield the following result.

\begin{corollary}
    \label{cor:spgen:jacobson:intersection:with:H}
    Let $R$ be a commutative associative unital ring with Jacobson radical $J$. Let $\nu$ be a unitary 
    equivalence relation on the index set $I$ such that $h(\nu) \geq (4,5)$. Let $H$ be a subgroup of
    $\Sp(2n,R)$ such that $\Ep(\nu,R) \le H$ and let $(\sigma, \Gamma)$ be the form net associated with
    $H$. Then 
    \[ 
        H \cap \Sp(2n,R,J) \le \Sp(\sigma, \Gamma),
    \]
    where $\Sp(2n, R, J)$ denotes the principal congruence subgroup of $\Sp(2n, R)$ of level $J$.
\end{corollary}

\section{Over a local ring}
\label{sec:spgen:local}

In this section we will prove our main result, Theorem \ref{theorem:spgen:main}
for the case when the ground ring $R$ is local. Moreover, we will prove a version
of this result suitable for application in localization based proof of the next chapter.
Throughout this section fix a standard setting $(R, R', S)$, 
where $R$ is a commutative local ring. Let $J$ denote
the Jacobson radical of $R$ (which is the only maximal ideal of $R$). 
Further, fix a unitary equivalence relation $\nu$ on the index set $I$,
a subgroup $H$ of $\Sp(2n, R)$ and an exact major form net of ideals 
$(\sigma', \Gamma')$ which is $S$-associated with $H$. Let $(\sigma, \Gamma)$ denote the $S$-closure of
the $(\sigma', \Gamma')$ in $R$. In this section we will show
that
\[
    H \le \Transp_{\Sp(2n, R)}(\Ep(\sigma', \Gamma'), \Sp(\sigma, \Gamma))
\]
provided $h(\nu) \ge (4,5)$.

\begin{theorem}
    \label{theorem:spgen:mainforlocal}
  	Let $(R,R',S), \nu, H, (\sigma', \Gamma')$ and $(\sigma, \Gamma)$ be as above.
    Assume $h(\nu) \geq (4,5)$. Let $a$ be an element of $H$ and $T_{sr}(\xi)$ be a $(\sigma', \Gamma')$-elementary
    transvection (long or short). Then $a T_{sr}(\xi) a^{-1} \in \Sp(\sigma, \Gamma)$. In other words,
    \[
        H \le \Transp_{Sp(2n, R)}(\Ep(\sigma', \Gamma'),\Sp(\sigma, \Gamma)).
    \]
\end{theorem}
\begin{proof}
    As $h(\nu) \geq (4,5)$ every index $p \in I$ embeds either into an $\A$-type base quintuple or
    into a $\C$-type base pair. Each of these two cases fork one more time depending on
    whether $T_{sr}(\xi)$ is long or short. We consider all four options separately in form of
    Lemmas \ref{lemma:spgen:short:local:selfconj}, \ref{lemma:spgen:long:local:selfconj},
    \ref{lemma:spgen:short:local:not:selfconj} and \ref{lemma:spgen:long:local:not:selfconj}.
\end{proof}
Now apply Theorem \ref{theorem:spgen:mainforlocal} to a trivial standard setting,
namely $R=R'$ and $S=\{1\}$. Togehter with Lemma \ref{lemma:symplss:ass:net:is:a:net} this yields Theorem
\ref{theorem:spgen:main} once the ground ring $R$ is local. 

\begin{lemma}
    \label{lemma:spgen:short:local:selfconj}
    Assume $h(\nu) \geq (4,4)$. Let $a$ be an element of $\Sp(2n,R)$ and 
    $T_{sr}(\xi)$ a short $(\sigma', \Gamma')$-elementary transvection. 
    Suppose that there exists an element $g \in \Ep(\sigma, \Gamma)$
    such that $ga \in H$. Let $b$ denote
    the short root element $a T_{sr}(\xi) a^{-1}$. If $(p,h)$ is a $\C$-type
    base pair then for all $i \in I$ the inclusion
    \begin{equation*}
%        \label{eq:lemma:spgen:short:local:selfconj:statement}
         b_{ip} \in \sigma_{ip}
    \end{equation*}
    holds.
    \begin{proof}
%        We we will prove this lemma under a slightly weaker assumption that instead of $a$ being an element of $H$
%        we only require that there exists an element $g \in \Ep(\sigma, \Gamma)$ such that $g a \in H$.
%        Note that this immediately yields that ${}^{g^{-1}} b \in H$.         
        We will organize the analysis into seven steps.
            
        1. Assume that the elements $a_{-p,-r}, a_{p,-r}$ and one of the element $a_{-h,-r}$ or $a_{h,-r}$ is contained 
        in $J$. Without loss of generality, we may assume that $a_{-h,-r} \in J$.
        In this case it's easy to see that $b_{pp}$ is invertible and $b_{-h,-p} \in J$. Indeed,
        \begin{align*}
            b_{pp} &= 1 + a_{ps} \xi a'_{rp} - \varepsilon_s \varepsilon_r a_{p,-r} \xi a'_{-s,p}
                \in 1 + J \le R^*, \\
%            b_{-p,-p} &= 1 + a_{-p,s} \xi a'_{r,-p} - \varepsilon_s \varepsilon_r a_{-p,-r} \xi a'_{-s,-p}
%                \in 1 + J \le R^*, \\
            b_{-h,-p} &= a_{-h,s} \xi a'_{r,-p} - \varepsilon_s \varepsilon_r a_{-h,-r} \xi a'_{-s,-p}
                \in J.          
        \end{align*}
        Therefore, the matrix  $b$ satisfies the conditions of Lemma \ref{lemma:spgen:jac:extract:C} and
        we get the inclusions $b_{ip} \in \sigma_{ip}$ for all $i \in I$.
                
        2. Assume that $a_{-p,-r} \in J$, but either $b_{h,-r} \in R^*$ or $b_{-h,-r} \in R^*$.
        We can assume that $a_{h,-r} \in R^*$, otherwise switch $h$ and $-h$.
        Consider the matrices
        \begin{align*}
           c &= T_{ph}(-a_{p,-r} a_{h,-r}^{-1}) a
           &d &= T_{-h,h}(-c_{-h,-r} b_{h,-r}^{-1}) d.
        \end{align*}
        It's easy to see that $d_{p,-r} = d_{-h,-r} = 0$ and by definition
        $g_1 d = a \in H$, where $g_1 \in \Ep(\sigma, \Gamma)$.
        Finally, $d_{-p,-r} = a_{-p,-r} \in J$. By the previous step
        we get the inclusion $f_{ip} \in \sigma_{ip}$ for all $i \in I$, where
        $f = d T_{sr}(\xi) d^{-1}$. It's only left to notice that $b_{ip} = f_{ip}$
        whenever $i \neq p, -h$ and the rest of the required inclusions are trivial.
        
        3. Assume that $a_{-p,-r}, a_{h,-r}, a_{-h,-r} \in J$, but $a_{p,-r} \in R^*$. Then by step 1
        we get $b_{ih} \in \sigma_{ih}$ for all $i \in I$. Now consider the matrix
        \[
            c = T_{hp}(1) a, \qquad f = c T_{sr}(\xi) c^{-1}.
        \]
        Observe that $c_{h,-r} \in R^*$ and $c_{-p,-r} = a_{-p,-r} \pm a_{-h,-r} \in J$. By step 2 we get
        $f_{ip} \in \sigma_{ip}$ for all $i \in I$. Finally, $f_{ip} = b_{ip} + b_{ih}$ whenever $i \neq h,-p$
        Therefore $b_{ip} \in \sigma_{ip}$ for all $i \in I$.
        
        \smallskip
        Steps 1-3 deal prove the lemma once $a_{-p,-r} \in J$. Note that as $T_{sr}(\xi) = T_{-r,-s}(\xi)$, 
        we can replace $r$ by $-s$ and $s$ by $-r$ in all the reasonings in steps 1-3. We can sum up
        the results of cases 1-3 together with the last observation as follows.
        
        4. Assume that $a_{-p,-r} \in J$ or $a_{-p,s} \in J$. Then $b_{ip} \in \sigma_{ip}$ for all $i \in I$.

        Before we continue with the opposite assumption, we will establish another version of step 1.
        
        5. Assume $a_{ps}, a_{p,-r}, a_{-h,-r} \in J$. Then as in case 1 we get
        \begin{align*}
            b_{pp} &= 1 + a_{ps} \xi a'_{rp} - \varepsilon_s \varepsilon_r a_{p,-r} \xi a'_{-s,p}
                \in 1 + J \le R^*, \\
%            b_{-p,-p} &= 1 + a_{-p,s} \xi a'_{r,-p} - \varepsilon_s \varepsilon_r a_{-p,-r} \xi a'_{-s,-p} \\
%                &= 1 \pm a_{-p,s} \xi a_{p,-r} \pm a_{-p,-r} \xi a_{ps}
%                \in 1 + J \le R^*, \\
            b_{-h,-p} &= a_{-h,s} \xi a'_{r,-p} - \varepsilon_s \varepsilon_r a_{-h,-r} \xi a'_{-s,p}
                = \pm a_{-h,s} \xi a_{p,-r} \pm a_{-h,-r} \xi a_{-ps}
                \in J.          
        \end{align*}
        Thus by Lemma \ref{lemma:spgen:jac:extract:C} we get the inclusions $b_{ip} \in \sigma_{ip}$ for all $i \in I$.
       
        The last two steps deal with the situation when $a_{-p,-r} \in R^*$.
        
        6. Assume that $a_{-p,-r}, a_{-p,s}, a_{ps} \in R^*$ and $a_{p,-r} \in J$. By case 4 we have $b_{i,-p} \in \sigma_{i,-p} 
        = \sigma_{ip}$ for all $i \in I$.
        Consider the matrices
        \[
            c = T_{-p,p}(-a_{-p,s} a_{ps}^{-1}) a, \qquad f = c T_{sr}(\xi) c^{-1}.
        \]
        Clearly, $c_{-p,s} = 0$ and it follows by case 4 that $f_{ip} \in \sigma_{ip}$ for all $i \in I$.
        Now observe that $f_{ip} = b_{ip} - a_{-p,s} a_{ps}^{-1} b_{i,-p}$ for all $i \neq -p$. 
        As $b_{i,-p} \in \sigma_{i,-p}$ we get
        $b_{ip} \in \sigma_{ip}$ for all $i \in I$.
        
        7. Finally, assume that $a_{-p,-r}, a_{-p,s} \in R^*$. Consider the matrices
        \[
            c = T_{-h,-p}(-a_{-h,-r} a_{-p,-r}^{-1}) a,
            \qquad d = T_{p,-p}(-b_{p,-r} a_{-p,-r}^{-1}) b,
            \qquad f = d T_{sr}(\xi) c^{-1}.
        \]
        It's easy to see that $d_{-h,-r} = 0$ and $d_{p,-r} = 0$. Further $d_{-p,s} = a_{-p,s} \in R^*$
        and $d_{-p,-r} = a_{-p,-r} \in R^*$.
        If $d_{ps} \in J$ then by step 5 we get $f_{ip} \in \sigma_{ip}$ for all $i \in I$. On the other hand
        if $d_{ps} \in R^*$ then by step 6 we have again $f_{ip} \in \sigma_{ip}$. It's only left to notice
        that $b_{ip} = f_{ip}$ whenever $i \neq -h, p$.
        
        In order to complete the proof note that the assumptions of steps 
        4 and 7 exhaust all the possibilities.
    \end{proof}
\end{lemma}

\begin{lemma}
    \label{lemma:spgen:long:local:selfconj}
    Assume $h(\nu) \geq (4,4)$. Let $a$ be an element of $Sp(2n,R)$ such that $ga \in H$ for some $g \in \Ep(\sigma, \Gamma)$.
    Let $T_{s,-s}(\alpha)$ be a long $(\sigma', \Gamma')$-elementary transvection and $b$ 
    the long root element $a T_{s,-s}(\alpha) a^{-1}$. If $(p,h)$ is a $\C$-type
    base pair then for all $i \in I$ the inclusion
    \begin{equation*}
%        \label{eq:lemma:spgen:long:local:selfconj:statement}
         b_{ip} \in \sigma_{ip}
    \end{equation*}        
    holds.
    \begin{proof}
        We will organize the analysis into four cases.
            
        1. Assume that $a_{ps} \in J$. Then it's easy to see that $b_{pp}$ is invertible
        and $b_{-h,-p} \in J$. Indeed,
        \begin{align*}
            b_{pp} &= 1 + a_{ps} \alpha a'_{-s,p} \in 1 + J \le R^*, \\
%            b_{-p,-p} &= 1 + a_{-p,s} \alpha a'_{-s,-p} = 1 \pm a_{-p,s} \alpha a_{ps} \in 1 + J \le R^*, \\
            b_{-h,-p} &= a_{-h,s} \alpha a'_{-s,-p} = \pm a_{-h,s} \alpha a_{ps} \in J.          
        \end{align*}
        Therefore, the matrix  $b$ satisfies the conditions of Lemma \ref{lemma:spgen:jac:extract:C} and
        we get the inclusions $b_{ip} \in \sigma_{ip}$ for all $i \in I$.
        
        2. Assume that $a_{hs}$ or $a_{-h,s}$ is invertible. Without loss of generality, assume the first,
        otherwise, simply switch $h$ and $-h$. Consider the matrix
        \[
            c = T_{ph}(-a_{ps}a_{hs}^{-1}) a.
        \]
        Clearly, $c_{ps} = 0 \in J$. Put $f = c T_{s,-s}(\alpha) c^{-1}$.
        We can apply the first case of this lemma to the pair of matrices $c$ and $f$
        in place of $a$ and $b$ and obtain the inclusions $f_{ip} \in \sigma_{ip}$ for all $i \in I$.
        It's only left to notice that $b_{ip} = f_{ip} \in \sigma_{ip}$ whenever $i \neq p,-h$ and
        the rest of the required inclusions are trivial as the form net $(\sigma, \Gamma)$ is major.
        
        3. Finally, assume that $a_{ps}$ is invertible, but $a_{hs} \in J$ and $a_{-h,s} \in J$.
        Switching $p$ and $h$ in the second case of this lemma we get $b_{ih} \in \sigma_{ih} = \sigma_{ip}$
        for all $i \in I$. Consider 
        \[
            c = T_{ph}(1) a, \qquad\qquad f = c T_{s,-s}(\alpha) c^{-1}.
        \]
        Clearly, $c_{ps} = a_{ps} + a_{hs} \in R^* + J \in R^*$ and $c_{hs} = a_{hs} \in J$.
        Thus, again by the second case we get $f_{ih} \in \sigma_{ip}$ for all $i \in I$.
        Finally, it's easy to see that $f_{ih} = b_{ih} + b_{ip}$ for all $i \neq p,-h$. As we already
        know that $b_{ih} \in \sigma_{ip}$. As $(\sigma, \Gamma)$ is major, this yields 
        the inclusion $b_{ip} \in \sigma_{ip}$ for all $i \in I$.
        
    \end{proof}
\end{lemma}

\begin{lemma}
    \label{lemma:spgen:short:local:not:selfconj}
    Assume $h(\nu) \geq (4,4)$. Let $a$ be an element of $\Sp(2n,R)$ such that $g_1 a \in H$
    for some $g_1 \in \Ep(\sigma, \Gamma)$. Let
    $T_{sr}(\xi)$ be a short $(\sigma', \Gamma')$-elementary transvection and $b$ 
    the short root element $a T_{sr}(\xi) a^{-1}$. Let  
    $(p,q,h,t,l)$ be an $\A$-type base quintuple. Then the inclusion
    \begin{equation}
        \label{eq:lemma:spgen:short:local:not:selfconj:statement:1}
         b_{ip} \in \sigma_{ip}
    \end{equation}        
    holds for any $i \in I$. If additionally $b \in \Sp(\sigma)$ then also 
    \begin{equation}
        \label{eq:lemma:spgen:short:local:not:selfconj:statement:2}
        S_{-p,p}(b^{-1}) \in \Gamma_{-p}.
    \end{equation}
    \begin{proof}
        Denote by $I'$ the set $\{ p,q,h,t,l \}$. This proof is organized as follows.
        \begin{enumerate}
            \item We will show that if $a_{l,-r}$ is invertible then the inclusion 
                \eqref{eq:lemma:spgen:short:local:not:selfconj:statement:1} holds for any $i \in I$ and if additionally
                $b \in \Sp(\sigma)$ then also \eqref{eq:lemma:spgen:short:local:not:selfconj:statement:2} holds.                                          
            \item We will show that if there exists an index $i \in I'$ such that 
                $a_{i,-r}$ or $a_{is}$ is invertible then the inclusion                
                \eqref{eq:lemma:spgen:short:local:not:selfconj:statement:1} holds for any $i \in I$ and if
                additionally $b \in \Sp(\sigma)$ then also \eqref{eq:lemma:spgen:short:local:not:selfconj:statement:2}
                holds. This case can be reduced to the previous one.
            \item Finally, if $a_{i,-r}, a_{is} \in J$ for all $i \in I'$ then the
                required inclusions \eqref{eq:lemma:spgen:short:local:not:selfconj:statement:1} and 
                \eqref{eq:lemma:spgen:short:local:not:selfconj:statement:2} follow from Corollary
                \ref{cor:spgen:jacobson:not:selfconj}.
        \end{enumerate}                
        
        \begin{itemize}
            \item[1.] Suppose $a_{l,-r} \in R^*$. Let
                \begin{align*}
                    g &= T_{tl}(a_{t,-r} a_{l,-r}^{-1}) T_{hl} (a_{h,-r} a_{l,-r}^{-1}) T_{ql} (a_{q,-r} a_{l,-r}^{-1})
                        T_{pl}(a_{p,-r} a_{l,-r}^{-1}), \\
                    c &= g^{-1} a \quad\text{and} \\
                    d &= c T_{sr}(\xi) c^{-1} = {}^{g^{-1}} b.    
                \end{align*}
                It is easy to see that $c_{p,-r} = c_{q,-r} = c_{h,-r} = 
                c_{t,-r} = 0$ and if $b \in \Sp(\sigma)$ then also $d \in \Sp(\sigma)$.
                We will consider three subcases:
                \begin{enumerate}[i.]
                   \item There is an index  $i_1 \in I' \setminus \{p, l\}$ such that $c_{i_1,s}$ is invertible.
                   \item The only $i_1 \in I' \setminus \{ l \}$ such that $c_{i_1,s} \in R^*$ is
                       $i_1 = p$.
                   \item $c_{is} \in J$ for all $i \in I' \setminus \{ l \}$ .
               \end{enumerate}     

            For each of the cases (1.i)--(1.iii) we will prove that $d_{ip} \in \sigma_{ip}$ for all $i \in I$
            and if $b \in \Sp(\sigma)$ then $S_{-p,p}(d^{-1}) \in \Gamma_{-p}$. 
            Note that 
            $d_{ip} = b_{ip}$ for all $i \neq p,q,h,t,-l$, the inclusion $b_{ip} \in \sigma_{ip}$
            for $i = p,q,h,t$ is trivial and
            \[ 
                d_{-l,p} = b_{-l,p} \pm \zeta_p b_{-p,p} \pm \zeta_q b_{-q,p} \pm \zeta_h b_{-h,p}
                \pm \zeta_t b_{-t,p},
            \]
            where $\zeta_i \in R$. As $b_{-i,p} \in \sigma_{-p,p}$ for $i = p,q,t,h$ it follows that
            $b_{-l,p} \in \sigma_{-l, p}$. Summing up, $b_{ip} \in \sigma_{ip}$ for all $i \in I$.
            
            Assume that $b \in \Sp(\sigma)$. By Corollary
            \ref{cor:symplss:length:of:left:mult:by:transv} we get
            \[
                S_{-p,p}(b^{-1}) \equiv S_{-p,p}({}^{g^{-1}} b^{-1}) \equiv S_{-p,p}(d^{-1}) \mod \Gamma_{-p}.
            \]
            Hence,
            $S_{-p,p}(b^{-1}) \in \Gamma_{-p}$. This completes the analysis of the case 1. 
            Now we consider the cases (1.i)--(1.iii).
               
            \item[1.i.] Suppose there exists an index $i_1 \in I' \setminus \{p, l \}$ such that $c_{i_1, s} \in  R^*$.
            In this situation we can apply Lemma \ref{lemma:spgen:short:alpha:beta:not:selfconj} to
            the matrix $c$, a short $(\sigma', \Gamma')$-elementary transvection $T_{sr}(\xi)$ and 
            the $\A$-type base triple $(i_1, i_2, p)$, where $i_2$ can be chosen in $I'\setminus\{p,l,i_1\}$.
            Thus we get $c_{i_1,s} d_{ip} \in 
            \sigma_{ip}$ for all $i \neq -i_1$ and if $b \in \Sp(\sigma)$ then also
            $c_{i_1,s}^2 S_{-p,p}(d^{-1}) \in \Gamma_{-p}$. 
            As $c_{i_1,s}$ is invertible, it follows that $d_{ip} \in \sigma_{ip}$ for all 
            $i \neq -i_1$ and if $b \in \Sp(\sigma)$ then $S_{-p,p}(d^{-1}) \in \Gamma_{-p}$. 
            
            Pick an index $i_2 \in I' \setminus \{p, l, i_1 \}$. If $c_{i_2,s}$ is invertible in $R$
            we can replace $i_1$ in the reasoning above with $i_2$ and get the missing inclusion
            $d_{-i_1, p} \in \sigma_{-i_1,p}$. If $c_{i_2,s}$ is not invertible, consider the matrices $f = T_{i_1, i_2}(1) c$
            and $g = f T_{sr}(\xi) f^{-1} = {}^{T_{i_1, i_2}(1)} d$. Clearly
            $f_{p,-r} = f_{q,-r} = f_{h,-r} = f_{t,-r} = 0$ and $
            f_{i_1,s}, f_{i_2,s} \in R^*$. Moreover ${}^{g T_{i_1, i_2}(-1)} b \in H$. Therefore
            we deduce by Lemma \ref{lemma:spgen:short:alpha:beta:not:selfconj} that
            $f_{i_1,s} g_{-i_2, p} \in \sigma_{-p,p}$ and thus $g_{-i_2, p} \in \sigma_{-p,p}$.
            Finally $g_{-i_2, p} = d_{-i_2, p} + d_{-i_1, p}$ and, as 
            $d_{-i_2, p} \in \sigma_{-p,p}$, it follows also that $d_{-i_1, p} \in \sigma_{-p,p}$. Therefore
            $d_{ip} \in \sigma_{ip}$ for all $i \in I$ and if $b \in \Sp(\sigma)$ then also 
            $S_{-p,p}(d^{-1}) \in \Gamma_{-p}$. 
            
            \item[1.ii.] Suppose $c_{hs}, c_{qs}, c_{ts} \in J$, but $c_{ps} \in R^*$.
            By the case (1.i) the inclusion $d_{ih} \in \sigma_{ih}$ holds for any $i \in I$ and 
            if $b \in \Sp(\sigma)$ then also $S_{-h,h}(d^{-1}) \in \Gamma_{-h}$. Consider
            the matrices $f = T_{hp}(1) c$ and $g = {}^{T_{hp}(1)} d$. Then
            ${}^{g T_{hp}(-1)} b \in H$, $f_{p,-r} = f_{q,-r} = f_{h,-r} = f_{t,-r} = 0$
            and $f_{hs} \in R^*$. By case (1.i) we get $g_{ip} \in \sigma_{ip}$ for
            all $i \in I$ and if $b \in \Sp(\sigma)$ then also
             $S_{-p,p}(g^{-1}) \in \Gamma_{-p}$. Observe that
            $g_{ip} = d_{ip} + d_{ih}$ for all $i \neq h,-p$ and, as $d_{ih}$ is already
            contained in $\sigma_{ip}$ for all $i \in I$, we get $d_{ip} \in \sigma_{ip}$ for all $i \neq -p$.
            Finally
            \begin{equation}
                \label{eq:lemma:spgen:short:local:not:selfconj:1}
                g_{-p,p} = d_{-p,p} \pm d_{-h,p} \pm d_{-p,h} \pm d_{-h,h}
            \end{equation}
            and, as $d_{-h,p} = \pm d'_{-p,h} = \pm d_{-p,h}$, the last
            three summands in \eqref{eq:lemma:spgen:short:local:not:selfconj:1}
            are contained in $\sigma_{-p,p}$. Thus $d_{-p,p} \in \sigma_{-p,p}$. 
            If $b \in \Sp(\sigma)$ then 
            by Corollary \ref{cor:symplss:length:of:left:mult:by:transv} we have
            \[
                S_{-p,p}(g^{-1}) = S_{-p,p}({}^{T_{hp}(1)} d^{-1})
                \equiv S_{-p,p}(d^{-1}) + S_{-h,h}(d^{-1}) \mod \Gamma_{-p}.
            \]
            As $S_{-h,h}(d^{-1}) \in \Gamma_{-p}$, we get $S_{-p,p}(d^{-1}) \in \Gamma_{-p}$.
            
            \item[1.iii.] Suppose $c_{is}, c_{i,-r} \in J$ for all $i \in I' \setminus \{ l \}$.
            Then $b_{ij} \equiv \delta_{ij} \mod J$ whenever $i \in I' \setminus \{ l \}$ or
            $-j \in I' \setminus \{ l \}$. By Corollary \ref{cor:spgen:jacobson:not:selfconj}
            we get the required inclusions $d_{ip} \in \sigma_{ip}$ for all $i \in I$ and
            if $b \in \Sp(\sigma)$ then also            
            $S_{-p,p}(d^{-1}) \in \Gamma_{-p}$.
            
            \item[2.] Suppose $a_{l,-r} \in J$, but there still exists an index $i_1 \in I'$ such that
            $a_{i_1,-r} \in R^*$ or $a_{i_1,s} \in R^*$. First, assume $a_{i_1, -r} \in R^*$.
            By case 1 we have $b_{il} \in \sigma_{il}$ for all $i \in I$ and if 
            $b \in \Sp(\sigma)$ then also $S_{-l,l}(b^{-1}) \in \Gamma_{-l}$.
            Consider the matrices $c = T_{l,i_1}(1) a \in H$ and $d = {}^{T_{l,i_1}(1)} b$.
            Then $c_{l,-r} \in R^*$ and by case 1 we get $d_{ip} \in \sigma_{ip}$ and
            if $b \in \Sp(\sigma)$ then also
            $S_{-p,p}(d^{-1}) \in \Gamma_{-p}$. Note that if $a \in \Sp(\sigma)$ we have
            \[
                S_{-p,p}(d^{-1}) \equiv S_{-p,p}(T_{l,i_1}(1) b^{-1}) \equiv
                S_{-p,p}(b^{-1}) + \delta_{p,i_1} S_{-l,l}(b^{-1}) \mod \Gamma_{-p}.
            \]
            Recall that $S_{-l,l}(b^{-1}) \in \Gamma_{-l} = \Gamma_{-p}.$
            Therefore $S_{-p,p}(b^{-1}) \in \Gamma_{-p}$. As
            $d_{ip} = b_{ip} + \delta_{i_1,p} b_{il}$ for all $i \neq l,-i_1$, it follows that $b_{ip} \in \sigma_{ip}$
            for all $i \neq -i_1$. Finally we have
            \begin{equation}
                \label{eq:lemma:spgen:short:local:not:selfconj:2}
                d_{-i_1,p} = b_{-i_1, p} \pm b_{-l,p} + \delta_{i_1,p} b_{-i_1,l} \pm \delta_{i_1,p} b_{-l,l}.
            \end{equation}
            Note that $b_{-l,p} = \pm b_{-p,l}$. Thus, the last three summands in 
            \eqref{eq:lemma:spgen:short:local:not:selfconj:2} are contained in $\sigma_{-p,p}$. 
            Therefore $b_{-i_1, p} \in \sigma_{-p,p}$. We conclude that
            $b_{ip} \in \sigma_{ip}$ for all $i \in I$ and if $b \in \Sp(\sigma)$ then also
            $S_{-p,p}(b^{-1}) \in \Gamma_{-p}$.
            
            Finally, if $a_{i_1, s} \in R^*$ we can use the Steinberg relation (R1), namely 
            $T_{sr}(\xi) = T_{-r,-s}(\pm \xi)$. Set $d = a T_{-r,-s}(\xi) a^{-1}$. We have already shown that in this case 
            $d_{ip} \in \sigma_{ip}$ for all $i \in I$ and
            $S_{-p,p}(d^{-1}) \in \Gamma_{-p}$. Finally $b_{ip} = \pm d_{ip}$ for all $i \neq p$
            and if $b \in \Sp(\sigma)$ then $S_{-p,p}(d^{-1}) = S_{-p,p}(b^{-1}) \pm 2 b_{-p,p} (2-b_{-p,-p}) \equiv
            S_{-p,p}(b^{-1}) \mod \Gamma_{-p}$. Therefore $S_{-p,p}(b^{-1}) \in \Gamma_{-p}$.
            
            \item[3.] Suppose $a_{is}, a_{i,-r} \in J$ for all $i \in I'$. Exactly
            as in case (1.iii) we get all the required inclusions via Corollary
            \ref{cor:spgen:jacobson:not:selfconj}.
        \end{itemize}
        This completes the proof.
    \end{proof}
\end{lemma}

\begin{lemma}
    \label{lemma:spgen:long:local:not:selfconj}
    Assume $h(\nu) \geq (4,4)$. Let $a$ be an element of $\Sp(2n,R)$ such that
    $g a \in H$ for some $g \in \Ep(\sigma, \Gamma)$. Let 
    $T_{s,-s}(\alpha)$ be a long $(\sigma', \Gamma')$-elementary transvection and $b$ 
    the long root element $a T_{s,-s}(\alpha) a^{-1}$. Let  
    $(p,q,h,t,l)$ be an $\A$-type base quintuple. Then the inclusion
    \begin{equation}
        \label{eq:lemma:spgen:long:local:not:selfconj:statement:1}
         b_{ip} \in \sigma_{ip}
    \end{equation}        
    holds for any $i \in I$. If additionally $b \in \Sp(\sigma)$ then also 
    \begin{equation}
        \label{eq:lemma:spgen:long:local:not:selfconj:statement:2}
        S_{-p,p}(b^{-1}) \in \Gamma_{-p}.
    \end{equation}
    \begin{proof}
        The proof is based on application of Corollary \ref{cor:spgen:jacobson:not:selfconj}.
        Recall that 
        \begin{equation}
            %\label{eq:lemma:spgen:long:local:not:selfconj:b}
            \nonumber
            b_{ij} = \delta_{ij} + a_{is} \alpha a'_{-s,j} = \delta_{ij} \pm a_{is} \alpha a_{-j,s}.
        \end{equation}
        Thus
        \begin{equation}
            \label{eq:lemma:spgen:long:local:not:selfconj:b:J}
            b_{ij} \equiv \delta_{ij} \text{ whenever } a_{is} \in J \text{ or } a_{-j,s} \in J.
        \end{equation}
        Set $I' = \{p,q,h,t,l\}$.
        
        1. First we will show that once there exists an index $j$ in the set $I'\setminus\{p\}$ such that 
        $a_{js}$ is invertible we can make sure that $b$ satisfies the conditions of Corollary 
        \ref{cor:spgen:jacobson:not:selfconj}, maybe after multiplying $a$ on the left with
        an element of $\Ep(\sigma, \Gamma)$. Without loss of generality, assume that $j = l$.
        Consider the matrices
        \begin{align*}
            g_1 &= T_{pl}(-a_{ps} a_{ls}^{-1}) T_{ql}(-a_{qs} a_{ls}^{-1}) 
                T_{hl}(-a_{hs} a_{ls}^{-1}) T_{tl}(-a_{ts} a_{ls}^{-1}), \\
            c &= g_1 a, \\
            f &= c T_{s,-s}(\alpha) c^{-1} = g_1 b g_1^{-1}.
        \end{align*}
        Clearly, $a_{ps} = a_{qs} = a_{ts} = a_{hs} = 0 \in J$. This yields by 
        \eqref{eq:lemma:spgen:long:local:not:selfconj:b:J} that $b_{ij} \equiv \delta_{ij} \mod J$
        whenever $i \in I'\setminus\{l\}$ or $-j \in I'\setminus\{l\}$. It follows by Corollary \ref{cor:spgen:jacobson:not:selfconj}
        that $f_{ip} \in \sigma_{ip}$ for all $i \in I$. Further, if $b \in \Sp(\sigma)$ then
        $f \in \Sp(\sigma)$ and by Corollary \ref{cor:spgen:jacobson:not:selfconj} also
        $S_{-p,p}(f^{-1}) \in \Gamma_{-p}$. Finally, using Corollary \ref{cor:symplss:length:of:left:mult:by:transv}
        we infer that $S_{-p,p}(b^{-1}) = S_{-p,p}(g_1^{-1} f^{-1} g_1) \equiv S_{-p,p}(f^{-1}) \mod \Gamma_{-p}$.
        Therefore $S_{-p,p}(b^{-1}) \in \Gamma_{-p}$.
        
        2. Next, assume that $a_{ps}$ is invertible, but $a_{qs}, a_{hs}, a_{ts}$ and $a_{ls}$ are contained in $J$.
        It follows by case 1 that $b_{il} \in \sigma_{il}$ and if $b \in \Sp(\sigma)$ then $S_{-l,l}(b^{-1}) \in \Gamma_{-p}$.
        Consider the matrices
        \[
            c = T_{lp}(1) a, \qquad f = c T_{s,-s}(\alpha) c^{-1} = {}^{T_{lp}(1)} b.
        \]
        Clearly, $c_{ls} = a_{ls} + a_{ps} \in R^*$ and we can apply case 1 to the matrices $c$ and $f$ instead of
        $a$ and $b$. Thus we get $c_{ip} \in \sigma_{ip}$ for all $i \in I$ and if $b \in \Sp(\sigma)$ (and thus also
        $f \in \Sp(\sigma)$) then also $S_{-p,p}(f^{-1}) \in \Gamma_{-p}$. It's only left to notice that
        $f_{ip} = b_{ip} + b_{il}$ and by Corollary \ref{cor:symplss:length:of:left:mult:by:transv} we have
        $S_{-p,p}(f^{-1}) = S_{-p,p}(T_{lp}^{-1}(1) b^{-1} T_{lp}(1)) \equiv S_{-p,p}(b^{-1}) + S_{-l,l}(b^{-1})$
        whenever $b \in \Sp(\sigma)$.
        We have already noticed that $b_{il} \in \sigma_{il}$ for all $i \in I$, therefore also $b_{ip} \in \sigma_{ip}$.
        If $b \in \Sp(\sigma)$ then, as $S_{-l,l}(b^{-1}), S_{-p,p}(f^{-1}) \in \Gamma_{-p}$, also
        $S_{-p,p}(b^{-1}) \in \Gamma_{-p}$.
        
        3. Finally, if $a_{is} \in J$ for all $i \in I'$ then 
        it follows by \eqref{eq:lemma:spgen:long:local:not:selfconj:b:J}
        that $b_{ij} \equiv \delta_{ij} \mod J$ whenever $i \in I'$ or $-j \in I'$. This yields 
        by Corollary \ref{cor:spgen:jacobson:not:selfconj} the inclusion 
        \eqref{eq:lemma:spgen:long:local:not:selfconj:statement:1} for all $i \in I$. If $b \in \Sp(\sigma)$
        then for the same reason we get the inclusion \eqref{eq:lemma:spgen:long:local:not:selfconj:statement:2}.
    \end{proof}
\end{lemma}

\section{Localization}
\label{sec:spgen:localization}

In this section we will prove Theorem \ref{theorem:spgen:main} for any commutative ring. 
We will first show that it's enough to prove our theorem for a Noetherian ground ring. The
general result follows almost immediately from the Noetherian case. For the case of a Noetherian
ground ring we use the following scheme to prove that 
$H \le \Transp_{\Sp(2n,R)}(\Ep(\sigma, \Gamma), \Sp(\sigma, \Gamma)$,
where $(\sigma, \Gamma)$ is the form net of ideals over $R$, associated with $H$ (cf. Lemma \ref{lemma:symplss:ass:net:is:a:net} of Section \ref{sec:symplss:assoc:net:and:norm}).
First, pick a maximal ideal $\m$ of $R$ and consider the localization $R_\m = (R \setminus \m)^{-1} R$
of the ring $R$. This is a local ring. More importnantly, there exists an element 
$x_0 \in R \setminus \m$ such that the homomorphism $\M(F_\m)$ induced by the 
localization morphism $F_\m : R \longrightarrow R_\m, \xi \mapsto \frac{\xi}{1}$ is injective on 
$\Sp(2n, R, x_0 R)$. This allows us to push the subgroup $H$ together with the form net of ideals
$(\sigma, \Gamma)$ into the ring $R_\m$, use Theorem \ref{theorem:spgen:mainforlocal} there
(observe that $(R_\m, F_\m(R), F_\m(R\setminus \m))$ becomes a standard setting and the cooridinate-wise image of
$(\sigma,\Gamma)$ under $F_\m$ becomes a form net $F_\m(R\setminus \m)$-associated with $\M(F_\m)(H)$)
and pull the resulting inclusion back to the ring $R$ using injectivity of $M(F_\m)$ and the idea
of \textit{patching}, which amounts to being able to get the unity of a ring as a linear combination
of elements of any subset of this ring which is unimodular, i.e. not contained in any maximal ideal. The 
uniqueness part of Theorem \ref{theorem:spgen:main} follows trivially from 
Theorem \ref{theorem:symplss:transpdescr}. 

\paragraph{Noetherian reduction.}\ \\
Suppose Theorem \ref{theorem:spgen:main} holds for any Noetherian ground ring.
It is a well known fact that every commutative ring $R$ is a direct limit $\varinjlim R'$ of its Noetherian
subrings $R'$. Fix a commutative ring $R$, a unitary equivalence relation $\nu$ such that $h(\nu) \geq (4,5)$
and a subgroup $H$ such that $\Ep(\nu, R) \le H$. Let $(\sigma, \Gamma)$ denote the exact major
form net of ideals associated with $H$, cf. Lemma \ref{lemma:symplss:ass:net:is:a:net}.
For any Noetherian subring $R'$ of $R$
set $H' = H \cap \Sp(2n,R')$. Then $\Ep(\nu, R') \le \Ep(\nu, R) \cap \Sp(2n, R') \le H'$.
By Lemma \ref{lemma:symplss:ass:net:is:a:net} there exists an exact major form net
of ideals $(\sigma', \Gamma')$ over $R'$ associated with $H'$. By the construction in
Lemma \ref{lemma:symplss:ass:net:is:a:net} of a
form net associated with a subgroup it follows that if $R' \le R''$ then
$(\sigma', \Gamma') \le (\sigma'', \Gamma'')$. Clearly $\Sp(\sigma', \Gamma') \le \Sp(\sigma'', \Gamma'')$.
As any element $g$ of $\Sp(2n, R)$ is contained in $\Sp(2n, R'')$ for some Noetherian subring $R''$ of $R$
such that $R' \le R''$, it follows that
\begin{equation}
    \label{eq:sec:spgen:localization:inter}
    \Sp(\sigma, \Gamma) = \varinjlim_{R' \text{ is Noetherian}} \Sp(\sigma', \Gamma').
\end{equation}
Pick any $a \in H$ and $T_{sr}(\xi) \in \Ep(\sigma, \Gamma)$. There exists 
a Noetherian subring $R'$ of $R$ such that $a, T_{sr}(\xi) \in \Sp(2n, R')$.
Clearly, $a \in H'$ and $T_{sr}(\xi) \in \Ep(\sigma', \Gamma')$. By assumption,
Theorem \ref{theorem:spgen:main} holds for the ground ring $R'$. Therefore 
\begin{equation}
    \label{eq:sec:spgen:localization:in}
    a T_{sr}(\xi) a^{-1} \in \Sp(\sigma', \Gamma'). 
\end{equation}
Moreover the inclusion \eqref{eq:sec:spgen:localization:in} holds for any Noetherian
subring $R''$ such that $R' \le R''$.
Combining \eqref{eq:sec:spgen:localization:in} with \eqref{eq:sec:spgen:localization:inter}
we deduce that 
\[
    \Ep(\sigma, \Gamma) \le H \le \Transp_{\Sp(2n, R)}(\Ep(\sigma,\Gamma), \Sp(\sigma, \Gamma))
    = N_{\Sp(2n,R)}(\Ep(\sigma, \Gamma), \Sp(\sigma, \Gamma)),
\] 
where the last equality is due to Theorem \ref{theorem:symplss:transpdescr}.
Therefore we only need to prove that the sandwich inclusions in Theorem \ref{theorem:spgen:main} are
satisfied for the form net of ideals associated with the subgroup $H$ for 
a Noetherian ground ring.

\begin{proof}[\upshape\bfseries Proof of Theorem \ref{theorem:spgen:main}: existence]
	The remark above justifies the assumption that $R$ is Noetherian.
    Let $(\sigma, \Gamma)$ be the form net of ideals associated with $H$, cf. Lemma
    \ref{lemma:symplss:ass:net:is:a:net}.
    Pick an element $a$ in $H$, a $(\sigma, \Gamma)$-elementary transvection $T_{sr}(\xi)$,
    long or short, 
    and denote by $b$ the corresponding root element $a T_{sr}(\xi) a^{-1}$. Our goal is to show that 
    $b$ is contained in $\Sp(\sigma, \Gamma)$. For each $i,j \in I$ put
    \begin{align*}
        X_{ij} &= \{ \xi \in R \;|\; \xi b_{ij} \in \sigma_{ij} \} \\
        Z_{i} &= \{ \xi \in R \;|\; \xi^2 S_{i,-i}(b) \in \Gamma_i \}.
    \end{align*}
    We will show that the sets $X_{ij}$ and $Z_i$ are unimodular,
    i.e. generate the unit ideal $R$, for all $i,j \in I$. Fix a maximal ideal
    $\m$ of $R$ and let $S$ denote the compliment $R \setminus \m$ of $\m$ in $R$. Let $R_\m$ denote
    the localization $S^{-1} R$ of the ring $R$ at the multiplicative system $S$ and $F_\m$
    the corresponding localization morphism $R \rightarrow R_\m$. 
	Recall that for $\zeta \in R$ and $s \in S$ we denote $F_\m(\zeta)\cdot F_\m(s)^{-1}$
	by $\frac{\zeta}{s}$.
    Let $R'_\m$ denote the image of $R$ under $F_\m$ and $S_\m$ denote the image of 
    $S$ under $F_\m$. We have already discussed at the end of Section \ref{sec:spgen:stsetting}
    that $(R_\m, R'_\m, S_\m)$ is a standard setting. Let $H_\m$ denote the image of $H$ under $\M(F_\m)$. 
    Clearly $\Ep(\nu, R'_\m) \le H_\m$.
    
    We will show now that there exists an element $x_0 \in S$ such that $F_\m$ is 
    injective on $x_0 R$. In fact, this is exactly the point of the Noetherian reduction.
    For an arbitrary ground ring this fail to be true in general.
    For each $x \in S$ set $Ann(x) = \{ \xi \in R \;|\; x \xi = 0 \}$. As $R$ is
    Noetherian, the set of ideals $\{ Ann(x) \;|\; x \in S \}$ contains a maximal
    element $Ann(x_0)$. Let $\xi x_0, \zeta x_0$ be two arbitrary elements
    of $x_0 R$. Suppose $F_\m(\xi x_0) = F_\m(\zeta x_0)$. Then there exists an element
    $y \in S$ such that $y x_0 (\xi - \zeta) = 0$. Therefore $\xi - \zeta \in Ann(y x_0) \ge Ann(x_0)$
    and by the maximality of $Ann(x_0)$ it follows that $\xi - \zeta \in Ann(x_0)$.
    Consequently $\xi x_0 = \zeta x_0$. Therefore the localization morphism $F_\m$
    is injective on $x_0 R$. This allows us to \textit{lift} transections from $H_\m$ to $H$
    up to multiplication of the parameter by an element of $S$. 
    
    \begin{proposition}
    \label{prop:spgen:goingup}
    Let $\zeta \in R$ and $x \in S$. If $T_{pq}(\frac{\zeta}{x}) \in H_\m$ then
    \[
        T_{pq}(x_0^{(1+\delta_{p,-q})} \cdot x^{\delta_{p,-q}} \cdot \zeta) \in H.
    \]
    \begin{proof}
        If $p \sim q$, the inclusion
        $T_{pq}(x_0^{(1+\delta_{p,-q})} \cdot x^{\delta_{p,-q}} \cdot \zeta) \in H$
        is trivial. Assume $p \nsim q$ and $p \neq -q$. There exists another
        index $h \sim q$ such that $h \neq \pm p, \pm q$. By the Steinberg relation (R4)
        \[
            T_{pq}(F_\m(\zeta)) = [[T_{pq}(\frac{\zeta}{x}), T_{qh}(\frac{x}{1})], T_{hq}(1)] \in H_\m.
        \]
        Pick any pre-image $g$ of the matrix $T_{pq}(F_\m(\zeta))$ contained in $H$.
        Note, that $g$ neither has to coincide with $T_{pq}(\zeta)$, nor has to be a
        transvection at all. Nevertheless,
		by (R4) we get
        \[
            T_{pq}(F_\m(\zeta x_0)) = \M(F_\m)([[g, T_{qh}(x_0)], T_{hq}(1)]) \in 
                \M(F_\m)(\Sp(2n, R, x_0 R) \cap H).
        \]
        As $\M(F_\m)$ is injective on $\Sp(2n, R, x_0 R)$, it follows that 
        $T_{pq}(\zeta x_0) \in H$.
        
        Assume $q = -p$ and $p \nsim -p$. Pick two more indices $h,t \in I$ such that
        $(p,h,t)$ is an $\A$-type base triple. By the Steinberg relations (R3), (R4) and (R6) we get
        \begin{align*}
            T_{h,-h}(-\varepsilon_p \varepsilon_h F_\m(\zeta x)) &= 
            [T_{p,-p}(\frac{\zeta}{x}), T_{-p,-h}(\frac{x}{1})] \\
            &\quad \times
            [T_{pt}(-1), [T_{tp}(1), [T_{p,-p}(\frac{\zeta}{x}), T_{-p,-h}(\frac{x}{1})]]] \in H_\m.
        \end{align*}
        As before, pick any pre-image $g$ of $T_{h,-h}(-\varepsilon_p \varepsilon_h F_\m(\zeta x))$,
        which is
        contained in $H$. Then by the Steinberg relations (R3), (R4) and (R6)
        \begin{align*}
            T_{p,-p}(F_\m(\zeta x x_0^2)) &=
            \M(F_\m)([g, T_{-h,-p}(x_0)] [T_{ht}(-1), [T_{th}(1), [g, T_{-h,-p}(1)]]]) \\
            &\in \M(F_\m)(\Sp(2n, R, x_0 R) \cap H).
        \end{align*}
        By the injectivity of $\M(F_\m)$ on $\Sp(2n, R, x_0 R)$ it follows that 
        $T_{p,-p}(\zeta x x_0^2) \in H$.
    \end{proof}
    \end{proposition}

	Let $(\sigma'_\m, \Gamma'_\m)$ denote the coordinate-wise image of $(\sigma, \Gamma)$
    under $F_\m$. It is easy to see that $(\sigma'_\m, \Gamma'_\m)$ is an exact major form
    net of ideals over $R'_\m$. Proposition \ref{prop:spgen:goingup} allows us to conclude
    that $(\sigma'_\m, \Gamma'_\m)$ is $S_\m$-associated, cf. Section \ref{sec:spgen:stsetting}, with $H_\m$. Indeed, the
    inclusion $\Ep(\sigma'_\m, \Gamma'_\m) \le H_\m$ is obvious. Suppose $T_{pq}(\frac{\xi}{x}) \in H_\m$
    for some elementary transvection $T_{pq}(\frac{\xi}{x})$ in $\Ep(2n, R_\m)$. By 
    Proposition \ref{prop:spgen:goingup} it follows that 
    $T_{pq}(\xi x^{\delta_{p,-q}} x_0^{(1+\delta_{p,-q})}) \in H$.
    Thus $\xi x^{\delta_{p,-q}} x_0^{(1+\delta_{p,-q})} \in (\sigma, \Gamma)_{pq}$ and 
    $F_\m(\xi) F_\m(x^{\delta_{p,-q}} x_0^{(1+\delta_{p,-q})}) \in (\sigma'_\m, \Gamma'_\m)_{pq}$. Therefore
    $(\sigma'_\m, \Gamma'_\m)$ is indeed $S_\m$-associated with $H_\m$.    
    
    Let $(\sigma_\m, \Gamma_\m)$ denote the $S_\m$-closure of $(\sigma'_\m, \Gamma'_\m)$ in $R_\m$.
	It follows by Theorem \ref{theorem:spgen:mainforlocal} that 
	$\M(F_\m)(b) \in \Sp(\sigma_\m, \Gamma_\m)$.

    Now we will show that each set $X_{ij}$ and $Z_i$ contains an element of $S$ and thus each is not contained in $\m$.
    For $i \sim j$ it is easy to see that $X_{ij} = R$, therefore we may assume that $i \nsim j$. Let
    $i \neq -j$. As $F_\m(b_{ij}) \in (\sigma_\m)_{ij}$, there exists an element $x \in S$ such that
    $T_{ij}(F_\m(b_{ij} x)) \in H_\m$. By Proposition \ref{prop:spgen:goingup} it follows that $b_{ij} x x_0 \in \sigma_{ij}$
    which means that $x x_0 \in X_{ij}$. If $j = -i$ then there exists an index $k \sim i$ such that
    $i \neq \pm k$. Therefore $F_\m(b_{i,-i}) \in (\sigma_\m)_{i,-i} = (\sigma_\m)_{i,-k}$. Exactly as in the previous case
    we get that $b_{ij} x x_0 \in \sigma_{i,-k} = \sigma_{i,-i}$ for some $x \in S$.
    Finally, assume $i \nsim -i$. As $F_\m(S_{i,-i}(b)) \in (\Gamma_\m)_i$, there exists an element
    $x \in S$ such that $T_{i,-i}(F_\m(S_{i,-i}(b) x^2)) \in H_\m$. By Proposition \ref{prop:spgen:goingup} it follows that
    $T_{i,-i}(F_\m(S_{i,-i}(b) x^2 x_0^2)) \in H$, which yields that $x_0 x \in Z_i$.
    
    Fix some indices $i$ and $j$. We have shown that the set $X_{ij}$ is unimodular, therefore
    $b_{ij} \in \langle X_{ij} \rangle b_{ij} \le \sigma_{ij}$. Thus $b \in \Sp(\sigma)$.
    We have also shown that the set $Z_{i}$ is unimodular. Therefore there exist elements $\zeta_1, \dots, \zeta_k \in Z_i$
    and elements $\xi_1, \dots, \xi_k \in R$ such that $\sum_{t = 1}^k \xi_t \zeta_t = 1$. Thus
    $\sum_{t=1}^k \xi_t^2 \zeta_t^2 \equiv 1 \mod 2R$ and
    \[
        S_{i,-i}(b) \in \sum_{t = 1}^k \xi_t^2 (\zeta_t^2 S_{i,-i}(b)) + 2 R S_{i,-i}(b) \le 
        \Gamma_i + 2 \sigma_{i,-i} \le \Gamma_i.
    \]
    Summing up, $b \in \Sp(\sigma, \Gamma)$. This completes the proof of existence of such
    an exact form net of ideals $(\sigma, \Gamma)$ that
    \begin{equation}
        \label{eq:theorem:spgen:main:1}
        \Ep(\sigma, \Gamma) \le H \le \Transp_{\Sp(2n, R)}(\Ep(\sigma, \Gamma), \Sp(\sigma, \Gamma)).
    \end{equation}    
\end{proof}

\begin{proof}[\upshape\bfseries Proof of Theorem \ref{theorem:spgen:main}: uniqness]
	Now we will show that an exact form net of ideals $(\sigma, \Gamma)$ satisfying
	\eqref{eq:theorem:spgen:main:1} is unique. Assume the contrary: 
    let $(\tau, B)$ be an exact major form net over $(R, \Lambda)$ such that 
    \begin{equation*}
        \Ep(\tau, B) \le H \le \Transp_{\Sp(2n, R)}(\Ep(\tau,B), \Sp(\tau, B)),
    \end{equation*}
    but $(\tau, B)$ is not equal to $(\sigma, \Gamma)$. As $(\sigma, \Gamma)$ is maximal
    among exact form nets such that $\Ep(\sigma, \Gamma) \le H$, it follows
    that $(\tau, b) \le (\sigma, \Gamma)$. Pick any
    $\xi \in (\sigma, \Gamma)_{ij}$. Then $T_{ij}(\xi) \in H \le 
    \Transp_{\Sp(2n, R)}(\Ep(\tau,B), \Sp(\tau, B))$.
    First, assume $i \neq -j$. 
    By property (T1) of Theorem \ref{theorem:symplss:transpdescr} applied to the net $(\tau, B)$ it follows,
    because $(T_{ij}(\xi))_{jj} = 1$, that
    \[
        \xi = (T_{ij}(\xi))_{ij} \cdot 1 \cdot (T_{ij}(\xi))_{jj} \le 
        (T_{ij}(\xi))_{ij} \cdot \tau_{jj} \cdot (T_{ij}(\xi))_{jj} \le \tau_{ij}.
    \]
    Therefore $\tau_{ij} = \sigma_{ij}$ for all $i \neq -j$. If $i = -j$ then by property
    (T2) of Theorem \ref{theorem:symplss:transpdescr}
    \[
        \xi = (T_{i,-i}(\xi))_{ii}^2 \cdot 1^2 \cdot S_{i,-i}(T_{i,-i}(\xi)) \le
        (T_{i,-i}(\xi))_{ii}^2 \cdot \tau_{ii}^\rectangled{2} \cdot S_{i,-i}(T_{i,-i}(\xi)) \le B_i.
    \]
    Therefore $B_i = \Gamma_i$ for all $i \in I$. Finally as both form nets $(\sigma, \Gamma)$ and $(\tau, B)$
    are exact, it follows that $(\sigma, \Gamma) = (\tau, B)$.
\end{proof}

\bibliography{../../references-eng}

\end{document}